\documentclass{amsart}

\usepackage{amssymb}
\usepackage[utf8]{inputenc}
\usepackage{color}
\usepackage{hyperref}
\usepackage{tikz-cd}
\usepackage[normalem]{ulem}
\usepackage{enumitem}
\usepackage{comment}
\usepackage{dutchcal} 

\makeatletter
\newcommand{\mylabel}[2]{#2\def\@currentlabel{#2}\label{#1}}
\makeatother

\newcommand{\set}[2]{\{#1 \,\mid \, #2\}}

\newcommand{\Ker}{\operatorname{ker}}

\newcommand{\im}{\operatorname{im}}

\newcommand{\A}{\mathcal{A}}

\newcommand{\Hom}{\operatorname{Hom}}

\newcommand{\id}{\operatorname{id}}
\newcommand{\Sh}{\operatorname{Sh}}
\newcommand{\cat}[1]{\mathcal{#1}}

\makeatletter
\newcommand{\subalign}[1]{%
  \vcenter{%
    \Let@ \restore@math@cr \default@tag
    \baselineskip\fontdimen10 \scriptfont\tw@
    \advance\baselineskip\fontdimen12 \scriptfont\tw@
    \lineskip\thr@@\fontdimen8 \scriptfont\thr@@
    \lineskiplimit\lineskip
    \ialign{\hfil$\m@th\scriptstyle##$&$\m@th\scriptstyle{}##$\hfil\crcr
      #1\crcr
    }%
  }%
}
\makeatother

\makeatletter
\newtheorem*{rep@theorem}{\rep@title}
\newcommand{\newreptheorem}[2]{%
\newenvironment{rep#1}[1]{%
 \def\rep@title{#2 \ref{##1}}%
 \begin{rep@theorem}}%
 {\end{rep@theorem}}}
\makeatother

\newtheorem{theorem}{Theorem}[section]
\newreptheorem{theorem}{Theorem}
\newtheorem{lemma}[theorem]{Lemma}
\newtheorem{proposition}[theorem]{Proposition}
\newtheorem{corollary}[theorem]{Corollary}
\theoremstyle{definition}
\newtheorem{definition}[theorem]{Definition}
\newtheorem{example}[theorem]{Example}

\theoremstyle{remark}
\newtheorem{remark}[theorem]{Remark}

\numberwithin{equation}{section}

\begin{document}

\title{Extra-fine sheaves and interaction decompositions}

\author[D. Bennequin]{Daniel Bennequin}
\address{Université Paris Cité}
\email{bennequin@math.univ-paris-diderot.fr}

%    author two information
\author[O. Peltre]{Olivier Peltre}
\address{Université Paris Cité}
\curraddr{InstaDeep}
\email{}

%    author three information =
\author[G. Sergeant-Perthuis]{Gregoire Sergeant-Perthuis}
\address{Université Paris Cité}
\curraddr{CQSB, Sorbonne Université}
\email{gregoire.sergeant-perthuis@sorbonne-universite.fr}

%    author four information
\author[J.P. Vigneaux]{Juan Pablo Vigneaux}
\address{Université Paris Cité}
\curraddr{Northwestern University}
\email{jpvigneaux@northwestern.edu}

% 	14F05  	Sheaves, derived categories of sheaves and related constructions
%   55U10  	Simplicial sets and complexes in algebraic topology
% 	55N30  	Sheaf cohomology in algebraic topology
%   06A11   Algebraic aspects of posets
%   60B99  	None of the above, but in this section (Probability theory on algebraic and topological structures)


\date{October 15, 2026}

\begin{abstract}

We introduce an original notion of (upper) extra-fine sheaf on a topological space, for which \v{C}ech cohomology in strictly positive degree vanishes. We apply this notion when the topological space is a partially ordered set (poset) equipped with the Alexandroff topology and the sheaves are made of duals of decomposable functors of vector spaces. The decomposition of the functor is known as interaction decomposition and can be expressed by a novel sum-intersection condition. We use these results to compute the dimension of the space of global sections when the sheaves are derived from those freely generated over a presheaf (of {\color{black}finite} sets); this generalizes and makes topological classical counting formulae for the number of solutions of the linearized marginal problem (Kellerer and Mat{\'u}{\v{s}}). We finish with a comparison theorem between the \v{C}ech cohomology associated to a covering and the topos cohomology of the poset with coefficients in the presheaf, which is also the cohomology of a cosimplicial local system over the nerve of the poset. For that, we give a detailed treatment of cosimplicial local systems on simplicial sets. The appendixes present presheaves, sheaves relative \v{C}ech cohomology, and their application to the marginal problem.
\end{abstract}

\maketitle

\section{Introduction}\label{sec:introduction}

\subsection{Motivation}\label{sec:intro-motivation}

Scientific investigations collect and use an ever-increasing amount of data;  the desire to give them an intelligible form, pertinent for the problems at hand, motivates the introduction of mathematical methods that are able to uncover latent structures. Historically, in many domains of mathematics (such as  algebraic geometry or global analysis) these structures have corresponded to topological invariants, and in particular to the more accessible ones, defined as (co)homologies of various kinds. 

In this work, we use sheaf theory and its generalization, topos theory, as a broad framework to study the relationship between local and global information: we consider different partial and incomplete perspectives on data, forming local viewpoints, which we aim to reconcile into a coherent global one. The interactions between viewpoints are encoded by hypergraphs or, more generally, by partially ordered sets (posets). We  develop cohomological tools to study collections of data associated with such combinatorial structures. The kind of data we consider is organized in families of sets indexed by the elements of the poset, forming covariant or contravariant functors with respect to the partial ordering; notable examples are functions or probabilities defined in terms of local variables. Such functors have been applied to several problems at the crossroad of data analysis, information theory, coding theory, logic, computation, and Bayesian learning; Section \ref{sec:intro-related} reviews some of this literature.

\subsection{Our setting}

We work with functors over partially ordered sets (posets), whose objects represent measurement contexts (the `viewpoints' we referred to above).  A poset $(\A,\leq)$ can be regarded as a category in two ways, and the choice will always be specified explicitly: we write $\cat C_{\leq}(\A)$ for the posetal category whose arrows $\alpha \to \beta$ encode $\alpha \leq \beta$, and $\cat C_{\geq}(\A)$ for the one obtained with the opposite convention.\footnote{Both notations $\to$ and $\leq$ are used below, and their relation is specified case by case.} For us, the main example of poset is an \emph{hypergraph} $\mathcal H$, which is any collection of finite subsets of a set $I$ ordered by inclusion, i.e.\ any subposet of the poset $\mathcal P_f(I)$ of finite subsets of $I$.  Below, the elements of $I$ represent elementary observables or sources, and the elements $\alpha$ of $\mathcal H$ represent measurement contexts, also known as \emph{interactions} or \emph{joint measurements}. Following the usage in the literature on hypergraphs and factor graphs \cite{10.5555/2500991}, we work with the posetal category $\cat C_{\supseteq}(\mathcal H)$. 

A \emph{presheaf} of sets on a category $\mathcal C$ is a contravariant functor from $\mathcal C$ to the category of sets $\mathcal S$; similarly, one also speaks of presheaves of abelian groups or of $\mathbb K$-vector spaces (for $\mathbb K$ a field), changing the target category. The presheaves of classical sheaf theory \cite{Godement1964} are recovered when $\mathcal C$ is the poset of open sets of a topological space, with arrows $U \to U'$ whenever $U \subseteq U'$, so that restriction maps are indeed contravariant.

The situation we have in mind, which is adapted to a wide variety of problems, is the following.  To take into account the internal degrees of freedom of each object of an hypergraph $\mathcal H$ over a set $I$, we introduce a covariant set-valued functor $E:\cat{C}_\supseteq(\mathcal H) \to \mathcal S$ of possible outcomes of the corresponding joint measurement, associating to each $\alpha \in \mathcal H$ a set $E_\alpha$ and to each arrow $\alpha \to \beta$ a surjective map $\pi^{\beta\alpha}:E_\alpha\to E_\beta$. The local probabilities on each $E_\alpha$, or the functions over each $E_\alpha$, give rise to other important functors, covariant or contravariant, and the problem of reconciling local data into global data becomes the problem of computing global sections.

Our approach to the geometry of an hypergraph, and more generally of a poset, is based on sheaf theory and simplicial methods. We equip $\mathcal A$ with the lower or upper Alexandroff topology (Definition \ref{Alexandrov-topology}), obtaining the topological spaces $X_{\mathcal A}$ and $X^{\mathcal A}$ respectively. There is a well-known equivalence of categories between set-valued functors on the posetal category and sheaves on these topological spaces,
\begin{equation}
[\cat{C}_\geq(\mathcal A), \mathcal S] \cong \Sh(X_{\mathcal A}), \qquad [\cat{C}_\geq(\mathcal A)^{op}, \mathcal S] \cong \Sh(X^{\mathcal A}),
\end{equation}
the morphisms being the natural transformations in all cases (Proposition \ref{prop:functors_are_sheaves}). Hence both a covariant and a contravariant functor on a poset can be seen as a classical sheaf on a topological space, to which \v{C}ech cohomology can be applied; this cohomology is convenient from a computational viewpoint and well adapted to the study of global sections.

\subsection{Main results}

\emph{Extra-fine sheaves are acyclic.} We introduce in Section \ref{sec:extra-fine} the notion of \emph{upper extra-fine} presheaf over a topological space $X$ (Definition \ref{def:upper-extra-fine-presheaf}), which is reminiscent of the classical notion of fine sheaf, but whose ``locality'' condition is of the opposite kind: local sections of a fine sheaf are decomposed by a partition of unity $(e_V)_{V\in \mathcal U \text{ open cover}}$, such that each component $e_V$ restrict to zero on the open sets that do not meet $\overline V$, whereas for an upper extra-fine sheaf the analogue of $e_V$ restricts to zero on the open sets $U$ that do not contain $V$ (Remark \ref{remark-extrafine-locality}). 
%Although this formulation might seem counterintuitive, it reflects a statement about cosheaves \cite{Curry2013} and about partitions of unity of cosheaves localized on open subsets of a neighborhood. 
Despite this sort of reversal, the classical conclusion survives: as for fine presheaves on paracompact spaces, upper extra-fineness implies acyclicity for \v{C}ech cohomology (Theorem \ref{thm1}), and this holds on \emph{any} topological space. For us, the upper extra-fine condition is remarkable because it provides a topological reformulation of the ``interaction decomposition'' of certain presheaves, which cannot be done by
working with the ordinary notion of fine sheaves.  The notion is however not functorial, being preserved neither by pullback nor by pushforward of a sheaf along a continuous map. Section \ref{sec:posets_alexandroff} characterizes upper extra-fine sheaves on the Alexandroff spaces $X_{\mathcal A}$.

\emph{Interaction decompositions and the sum-intersection condition.} Let $V$ be a presheaf on $\cat{C}_\geq(\mathcal A)$ with values in $\textbf{Vect}_\mathbb{K}$ whose structural maps $j_{\alpha\beta}:V_\beta \to V_\alpha$ (for $\alpha \to \beta$) are injective, and set $V_{\alpha\beta} = j_{\alpha\beta}(V_\beta)$. For each $\alpha$, the \emph{boundary observables} form the subspace $V_\alpha' = \sum_{\beta: \alpha \xrightarrow{\neq} \beta} V_{\alpha\beta}$, where $\alpha \xrightarrow{\neq} \beta$ means $\alpha \to \beta$ with $\beta \neq \alpha$, and an \emph{interaction subspace} is any supplement $S_\alpha$ of $V'_\alpha$, so that $V_\alpha = V_\alpha' \oplus S_\alpha$. One says that $V$ admits an \emph{interaction decomposition} when $V_\alpha = \bigoplus_{\beta:\alpha \to \beta} j_{\alpha\beta}(S_\beta)$ for every $\alpha$; equivalently, when $V$ splits as a direct sum of presheaves $V = \bigoplus_{\gamma \in \mathcal A} S^\gamma$, where $S^\gamma(\alpha) = j_{\alpha\gamma}(S_\gamma)$ whenever $\alpha \to \gamma$ and $S^\gamma(\alpha) = 0$ otherwise (Section \ref{sec:interaction-decomposition-condition-G}). The first main result of the article, Theorem \ref{thm2:extrafine_and_conditionG}, characterizes the existence of such a decomposition by the following sum-intersection property, discovered by G.~Sergeant-Perthuis, that we call \emph{condition G}:
\begin{equation}\label{G}\tag{G}
\forall \alpha, \beta \in \mathcal A \text{ such that }\alpha \to \beta, \quad V_{\alpha\beta} \cap \left( \sum_{\substack{\gamma:  \alpha \xrightarrow{\neq} \gamma \\ \gamma  \not\to  \beta}} V_{\alpha\gamma} \right) \subseteq \sum_{\substack{\gamma:   \alpha \xrightarrow{\neq} \gamma \\  \beta  \xrightarrow{\neq}\gamma }} V_{\alpha\gamma}.
\end{equation}
The existence of an interaction decomposition always implies \eqref{G}, and \eqref{G} implies an interaction decomposition as soon as $\mathcal A$ is of locally finite dimension. We impose no finiteness condition on the vector spaces, and only weak finiteness conditions on the poset. What makes \eqref{G} valuable is that it involves only the  data that defines $V$, whereas the interaction decomposition asserts the \emph{existence} of a splitting which is not unique. We also show, under some restrictive hypotheses (Corollary \ref{dual:upper-extra-fine}), that when $V$ admits an interaction decomposition and it arises as the linear dual of a covariant functor $F$, then $F$ is an upper extra-fine sheaf on $X_{\mathcal A}$ (i.e. on $\mathcal A$ with the lower Alexandroff topology). This is what ties the algebra of interaction decompositions to the topology of extra-fine sheaves.

\emph{Free functors and their cohomology.} Section \ref{sec:free-sheaves} specializes to the \emph{free} functors generated over $\mathbb K$ by a covariant set-valued functor $E$ as above --- the usual case in data analysis over hypergraphs. Here $V$ associates to $\alpha$ the space of $\mathbb K$-valued functions on $E_\alpha$ and to $\alpha \to \beta$ the injection $j_{\alpha\beta}: V_\beta \to V_\alpha$ induced by $\pi^{\beta\alpha}$. We establish \eqref{G} for these functors when $\mathcal A$ has conditional coproducts (for hypergraphs, this is stability under non-empty intersections), and deduce the existence of an interaction decomposition when $\mathcal A$ has all finite meets, including the empty one (Theorem \ref{thm3:IP_implies_extrafine}). The second main result of the article, Theorem \ref{thm4}, computes the resulting cohomology when $\mathcal A$ is stable under all finite meets, excluding the empty one: it is the sum of the  \v{C}ech cohomology of $X_{\mathcal A}$ in all degrees and of the degree-zero cohomology of the restricted sheaf of functions whose coordinates sum to zero.

\emph{The linearized marginal problem.} These results have a concrete counterpart. Suppose $I$ finite, let $(E_i)_{i\in I}$ be finite sets and $E_\beta = \prod_{i\in\beta} E_i$. The functor of probabilities $P$ on $\cat{C}_\supseteq(\mathcal H)$ assigns to $\beta$ the probability measures on $E_\beta$ and to each arrow the corresponding marginalization map; a section of $P$, i.e.\ an element of $\Gamma_{\mathcal H}(P)$, is thus a coherent collection of local probability laws, compatible under marginalization --- what is often called a \emph{pseudomarginal}. The (discrete) \emph{probabilistic marginal problem} asks: given $p \in \Gamma_{\mathcal H}(P)$, when does there exist a section $q$ over the full abstract simplex $\Delta(I)$ that restricts to $p$? Its \emph{linearized} version allows signed measures subject to a normalization condition, and is governed by the covariant functor $\overline F$, the predual of the reduced free functor $\overline V$ obtained by quotienting $V$ by the constant functions. The interaction decomposition of $\overline V$ implies that $\overline F$ is upper extra-fine, whence Theorem \ref{thm5}, giving the surjectivity of $\Gamma_{\Delta(I)}(P)\to \Gamma_{\mathcal H}(P)$, and the index formula of Theorem \ref{thm6}, which quantifies the dimension of these spaces in terms of the dimensions of the interaction subspaces and the Euler characteristic of the poset. This is a novel topological reformulation and extension --- to more general functors than $P$, and to a more general class of posets --- of the classical results of Kellerer \cite{Kellerer1964} and Mat{\'u}{\v{s}} \cite{Matus1988}, and we hope it will lead to new ideas to tackle the probabilistic version of the problem.

\emph{\v{C}ech cohomology and topos cohomology.} Finally, Section \ref{sec:nerves-comparison} returns to a general topological space $X$ and presheaves of abelian groups, and establishes a homotopy equivalence between the \v{C}ech cochain complex of an open covering and the cochain complex of the \emph{nerve of the category} generated by the covering, for a general notion of cosimplicial coefficients (Theorem \ref{them7}). This answers a natural question in our framework: it identifies the \v{C}ech cohomology on $X_\mathcal{A}$ with the topos cohomology induced by the global sections functor on $\cat{C}_\leq(\mathcal A)$, provided $\mathcal A$ is conditionally closed under meets. The proof is surprisingly cumbersome, which is reminiscent of the known fact that a finite simplicial complex is homotopy equivalent to its barycentric subdivision, but non-canonically.

\subsection{Related work}\label{sec:intro-related}

\emph{Sheaves in applied topology.} Simplicial complexes and sheaves have become standard tools of applied algebraic topology. \emph{Persistent homology} \cite{Carlsson2006, Ghrist2008} replaces a cloud of data points in a metric space by a simplicial complex determined by a proximity parameter, and studies the resulting filtration homologically. Curry's dissertation \cite{Curry2013} extended it in several directions involving sheaves on posets of parameters, and gave a systematic treatment of cellular sheaves and cosheaves, which he traces back to Zeeman's thesis \cite{Zeeman1962}; a spectral theory of these was later developed by Hansen and Ghrist \cite{Hansen2018}. Both works list many situations modeled by cellular (co)sheaves: network coding, sensor networks, distributed consensus, flocking, synchronization, opinion dynamics. Kashiwara and Schapira \cite{Kashiwara}, in turn, recover and extend persistence theory, including multidimensional filtrations and barcodes, in the derived category of sheaves.

\emph{Sheaves and data integration.} Robinson and collaborators \cite{Robinson2013, Robinson2015, Robinson2017} argued that sheaves are a canonical model for the integration of information provided by interconnected sensors: the vertices of a simplicial complex represent heterogeneous data sources and the simplices their interactions, sheaves combining the generality needed to represent all sensors of interest with the ability to summarize information faithfully; see also \cite{Mansourbeigi2018}. Independently, Abramsky and his collaborators (see e.g.\ \cite{Abramsky2011sheaf, Abramsky2015}) used sheaves on simplicial complexes to study \emph{contextuality}, where the vertices are observables and the simplices measurement contexts: some coherent families of local probabilities do not extend to a global law. This is what we call here the \emph{probabilistic marginal problem}, which also admits linearized and ``possibilistic'' versions.

\emph{Invariants of hypergraphs.} Simplicial complexes are geometrically realizable as CW-complexes, hence accessible to the standard tools of algebraic topology; hypergraphs, which come from combinatorics rather than geometry, are less so, and several (co)homological invariants have been proposed for them. Bressan, Li, Ren and Wu \cite{Bressan2019} define the \emph{embedded homology} of a hypergraph $\mathcal H$ as the homology of the smallest simplicial complex containing $\mathcal H$, and Chung and Graham \cite{Chung1992} a cohomology of $k$-regular hypergraphs. Closest to our concerns, Everitt and Turner \cite{EVERITT2015134} attach a cochain complex to a graded poset equipped with a presheaf, and compare its cohomology with the derived functors of the limit and with Khovanov cohomology \cite{khovanov2000categorification}.

\emph{Interaction decompositions.} The term refers to the decomposition into interaction subspaces of factor spaces \cite{Speed,Lauritzen}, an essential ingredient of the Hammersley--Clifford theorem in the theory of graphical models \cite{Lauritzen}; further examples are Wick's theorem and the Wiener chaos decomposition, in particular the expansion of functions into Hermite polynomials \cite{sinai2014theory}. Condition \ref{condition_G} was introduced by Sergeant-Perthuis in \cite{sergeantperthuis2021intersectionpropertyinteractiondecomposition} to characterize the arrangements of planes, in a given vector space, that satisfy the interaction decomposition, and was extended there to well-founded posets of subvector spaces; Theorem \ref{thm2:extrafine_and_conditionG} extends this to the categorical setting under a local finite-dimensionality assumption. He also observed in \cite[Ch.~7]{sergeant2021intersection} that an interaction decomposition makes the presheaf a projective object in the category of contravariant functors; that the two are equivalent for finite posetal categories was known in combinatorics \cite{YANAGAWA2001341}, where the terminology and the link with graphical models were not made explicit. See also \cite{Sergeant2020presheaves,sergeantperthuis2021interactiondecompositionhilbertspaces}.

\emph{Information cohomology and algorithms.} Real-valued functions of the probability laws on the sets $E_\alpha$ over a simplicial complex admit a topos-theoretic and cohomological interpretation \cite{Artin1972} of the information quantities of Shannon and Kullback, and of von Neumann in the quantum case \cite{Baudot2015,Vigneaux2020information}, later extended to functions of statistical frequencies and to Gaussian laws \cite{Vigneaux2019-thesis}; those cohomologies are not of \v{C}ech type, being based on the action of variables on probabilities by conditioning. We also expect sheaf-theoretic constructions on hypergraphs to clarify certain algorithms in statistics and machine learning: Peltre \cite{Peltre2019,peltre2021belief,peltre2022local,peltre2023local,OP-phd} gave a cohomological account of the generalized Belief Propagation algorithm \cite{Yedidia2005} as a non-linear dispersion flow, and Sergeant-Perthuis \cite{Sergeant2019bayesian,Sergeant2019intersection,Sergeant2020presheaves,sergeantperthuis2024compositionalstatisticalmechanicsentropy,sergeantperthuis:hal-04456412,Syco-12,sergeantperthuis:hal-04527780,sergeantperthuis:hal-04527775,sergeant2021intersection} provided elements toward defining the thermodynamic limit in the category of Markov kernels, extending the decomposition into interaction subspaces, the space of Hamiltonians, and Gibbs states.

\emph{\v{C}ech versus topos cohomology on posets.} After the first version of this
article appeared (arXiv:2009.12646v1, September 2020), Husainov \cite{husainov2023cech}
compared the \v{C}ech cohomology of a presheaf on a posetal category---seen as a sheaf on the associated Alexandroff space---to the topos cohomology of the its sheafification, building on
Corollary~\ref{cor1:isomorphsm-poset}. He obtains necessary and sufficient conditions
for the two to be isomorphic; for finite posets they reduce to the acyclicity of
certain upper sections of Dedekind--MacNeille cuts, and can be checked algorithmically.
These conditions hold in particular when meets exist conditionally, which recovers
Corollary~\ref{cor1:isomorphsm-poset}.

\subsection{Structure of the article}\label{sec:intro-structure}

Section \ref{sec:extra-fine} introduces upper extra-fine presheaves and proves their \v{C}ech acyclicity (Theorem \ref{thm1}). Section \ref{sec:posets_alexandroff} recalls the Alexandroff topologies on a poset and the induced identification of functors with sheaves, characterizes the upper extra-fine sheaves on $X_{\mathcal A}$, and collects the finiteness conditions used later. Section \ref{sec:interaction-decomposition} proves the equivalence between interaction decompositions and Condition \ref{G} (Theorem \ref{thm2:extrafine_and_conditionG}) and treats linear duality and preduality. Section \ref{sec:free-sheaves} studies free functors: condition \ref{condition_G} and the interaction decomposition under intersection properties of the poset (Theorem \ref{thm3:IP_implies_extrafine}), the computation of their cohomology (Theorem \ref{thm4}), the marginal theorem (Theorem \ref{thm5}), and the index theorem (Theorem \ref{thm6}). Section \ref{sec:nerves-comparison} proves the comparison theorem \ref{them7}, after a detailed treatment of cosimplicial local systems on simplicial sets. The first two appendices summarize the functoriality of sheaves and \v{C}ech cohomology, the third proves the index formula by M\"obius inversion, and the last gives a counterexample to acyclicity for the upper Alexandroff topology.

The main results and illustrations are in Sections \ref{sec:interaction-decomposition} and \ref{sec:free-sheaves}; the other sections are of a more expository nature.

\section{Fine, upper extra-fine and acyclic}\label{sec:extra-fine}

In this section, we introduce the notion of an upper extra-fine presheaf. We show that for \v{C}ech cohomology, such objects are acyclic. Given a presheaf $F$ over a base space $X$, we write $x|V$ for the restriction of 
a section $x \in F(W)$ to $V \subseteq W$,
see Appendix \ref{appendix1:topology} for some basic topological definitions and notations. We use \v{C}ech cohomology as presented in any standard reference,  e.g. \cite{Spanier1966}, but all relevant definitions can also be found in Sections \ref{appendix2:nerve} and \ref{sec:simplicial-systems} under the formalism of cosimplicial local systems.

\begin{definition}[\emph{Fine presheaf}, cf. \protect{\cite[Sec.~6.8]{Spanier1966}}]\label{def:fine} A presheaf $F$ of abelian groups over a topological space $X$ is said to be \emph{fine} if for every {locally finite} \footnote{{By definition, an open cover $\mathcal{U}$ of $X$ is locally finite when every point of $X$ has a neighborhood which meets only a finite
number of elements of  $\mathcal{U}$.}} open cover $\mathcal{U}$ of $X$, there exists a family $\{e_V\}_{V\in \mathcal{U}}$ of endomorphisms of $F$ (i.e. natural transformations $e_V:F\to F$, whose components $e_V(W)$ we denote by $e_{V|W}$), such that:
\begin{enumerate}

\item[\mylabel{condition:local}{(i)}] 
For every  open $W$,
$x \in F(W)$ and $V \in {\mathcal U}$, one has $e_{V|W}(x)|(W\setminus \bar V)=0$, in other words the restriction of $e_{V|W}(x)$ to $(W \setminus \overline{V})$ is zero
\item[\mylabel{condition:resolution_identity}{(ii)}] 
For every open $W$ that encounters only a finite number of closures $\overline{V}$  of members $V$ of the cover ${\mathcal U}$,
and every $x\in F(W)$, we have $x=\sum_{V\in \mathcal U} e_{V|W}(x)$.  
\end{enumerate}

\end{definition}

Fine presheaves are part of the classical literature on sheaf theory, see also \cite[Sec.~3.7]{Godement1964} and \cite[p.~42]{Griffiths1978}, although the classical definitions require {that the space $X$ is paracompact}. Positive dimensional cohomology of a paracompact topological space with coefficients in a fine presheaf vanishes \cite[Thm.~6.8.4]{Spanier1966}, and this fact has important implications, {for instance} in the comparison of Alexander and \v{C}ech cohomology.  We propose here an analogue of this notion that plays a fundamental role in our investigations; we call such sheaves upper extra-fine. We shall see that \v{C}ech cohomology in positive degree of an upper extra-fine presheaf vanishes. Let us restate that, when referring to presheaves, we consider contravariant functors taking values in the category of sets, abelian groups, or $\mathbb{K}$-vector spaces, with $\mathbb{K}$ any field, or any of their subcategory in a natural canonical manner. To fix some notations, we summarize here the construction of \v{C}ech cohomology; more details can be found in Sections \ref{appendix2:nerve} and \ref{sec:simplicial-systems}.
For that purpose, 
let us fix a presheaf of abelian groups $F$ over a topological space $X$ equipped with an open covering ${\mathcal U}$. 

For each $n\in \mathbb{N}$, let $K_n(\mathcal{U})$ denote the set of sequences of length $n+1$, $u=(U_0,...,U_n)$, of elements of $\mathcal{U}$ such that the intersection $U_u=U_0\cap ...\cap U_n$ is non-empty. 
For $n\in \mathbb{N}$, a \emph{\v{C}ech cochain} of $F$ of degree $n$ with respect to $\mathcal{U}$ is an element $\{\psi(u)\}_{u\in K_n(\mathcal{U})}$ of $\prod_{u\in K_n(\mathcal{U})} F(U_u)$.  The set of $n$-cochains is denoted $C^{n}(\mathcal{U};F)$; it is an abelian group.

A coboundary operator $\delta^{n+1}_{n,F}: C^{n}(\mathcal{U};F)\rightarrow C^{n+1}(\mathcal{U};F)$ is then introduced, as a linear map such that
\begin{equation}
\delta^{n+1}_{n,F} \psi(U_0,...,U_{n+1})=
\sum_{i=0}^{n+1}(-1)^{i} \:
\psi(U_0,...,\widehat{U_i},...,U_{n+1})|U_0\cap ...\cap U_{n+1},
\end{equation}
where $\widehat{U_i}$ means that $U_i$ is omitted. When the indices are clear from context, we write $\delta$ or $\delta_n^{n+1}$ instead of $\delta_{n,F}^{n+1}$.

It is well known that  $\delta\circ \delta=0$, which allows one to define the \v{C}ech cohomology of $F$ over $\mathcal{U}$ in degree $n\geq 0$ as the quotient abelian group
$H^{n}(\mathcal{U};F)=\Ker (\delta^{n+1}_{n})/\im (\delta^{n}_{n-1})$; by definition, $im(\delta_{-1}^0)$ is $0$. The set of open coverings of $X$ with the relation of refinement
is a \emph{directed set}{\color{black};} the \v{C}ech cohomology of $F$  over $X$ is defined as
\begin{equation}
\forall n\in \mathbb{N},\quad H^{n}(X;F)=\varinjlim H^{n}(\mathcal{U};F).
\end{equation}
where $\varinjlim$ denotes the directed colimit over the directed set; see \cite[Ch.~5]{Godement1964}, \cite[Sec.~6.7.11]{Spanier1966} or \cite[Sec.~0.3]{Griffiths1978}.

From the definition of $\delta^{1}_0$, it is clear that the group $H^{0}(\mathcal{U};F)$ can be identified with the group of global sections of $F$
over $X$, for any open covering $\mathcal U$. Hence $H^{0}(X;F)$  coincides with every $H^{0}(\mathcal{U};F)$ and also corresponds to global sections.

A presheaf is called \emph{acyclic} if its cohomology is zero for every degree $n\geq 1$.

\begin{definition}\label{def:functor_generated}

Let \(\mathcal{A}\) be a partially ordered set and let $\cat{C}_\leq(\A)$ be the associated posetal category with $\beta\to \alpha$ whenever $\beta\leq\alpha$. A contravariant functor \(S\) from $\cat{C}_\leq(\mathcal{A})$ to the category of abelian groups will be called \emph{generated by} $\gamma$ \emph{or at} $\gamma\in \mathcal{A}$ if:

\begin{itemize}
    \item for all \(\alpha \in \mathcal{A}\) satisfying \(\gamma \not\leq \alpha\), \(S(\alpha) = 0\).
    \item for any two \(\alpha, \beta \in \mathcal{A}\) such that \(\gamma \leq \beta \leq \alpha\), the morphism \(S(\beta\to \alpha):S(\alpha)\to S(\beta)\) is an isomorphism.
\end{itemize}

\end{definition}

Let \(\mathcal{U}\) be an open cover of a topological space \(X\). The inclusion relation between elements of \(\mathcal{U}\) defines a partial order on the open sets \(V, U \in \mathcal{U}\) by \(V \subseteq U\). The associated poset can, in turn, be seen as a category. Let \(F\) be a presheaf on a topological space \(X\) taking values in abelian groups. Given an open cover \(\mathcal{U}\), we denote by \(F\vert_{\mathcal{U}}\) the restriction of \(F\) to $\cat C_{\subseteq}(\mathcal U)$.

\begin{definition}[Upper Extra-Fine presheaf]\label{def:upper-extra-fine-presheaf}
A presheaf \(F\) over a topological space \(X\) taking values in abelian groups is called \emph{upper extra-fine} if, for any open cover \(\mathcal{V}\), there exists a finer open cover \(\mathcal{U}\) that is closed under finite intersections, and associated to $\mathcal{U}$ there exists a collection of functors $(S_U: \cat{C}_{\subseteq}(\mathcal U)^{op} \to \cat{Ab})_{U\in \mathcal U}$ such that each $S_U$ is generated by $U$ and
\[
F\vert_{\mathcal{U}} \cong \bigoplus_{U \in \mathcal{U}} S_U.
\]
In other words, for such a finer cover \(\mathcal{U}\), there is a natural transformation 
\[
\phi_{\mathcal{U}}: F\vert_{\mathcal{U}} \to \bigoplus_{U \in \mathcal{U}} S_U
\]
such that for all \(V \in \mathcal{U}\), the map 
\[
\phi_{\mathcal{U}}(V): F(V) \to \bigoplus_{U \in \mathcal{U}} S_U(V)
\]
is an isomorphism.
\end{definition}

\begin{remark}[Lack of functoriality]
Let $f:X\rightarrow Y$ be a continuous map and $\mathcal{F}$ a fine presheaf of abelian groups over $X$; the presheaf $\mathcal{G}=f_*\mathcal{F}$ on $Y$ is fine \cite[Thm.~6.8.3]{Spanier1966}, but it can happen that $\mathcal{F}$ is upper extra-fine on $X$ and that $\mathcal{G}=f_*\mathcal{F}$ is not upper extra-fine. For the inverse
image of a presheaf $\mathcal{G}$ over $Y$, both fine and extra-fine fail to be transmitted from $\mathcal{G}$
to $f^*\mathcal{G}$.
\end{remark}

\begin{remark}\label{remark-extrafine-locality}
The locality of fine sheaves (condition (i)) requires that \(e_V\), when restricted to the complement of \(\overline{V}\), vanishes. In particular, the projection of elements \(x \in F(W)\) outside of \(\overline{V}\) reduces to \(0\), so only those inside \(V\) remain. For upper extra-fine sheaves, it is required that there exists a partition of unity of \(F\) when restricted to \(\mathcal{U}\), which is similar to condition (ii) in the case of fine sheaves and corresponds to the functors \(S_V\) from which \(F\) is built. However, \(S_{V}(W) = 0\) whenever \(V \not\subseteq W\), which implies that only those \(V\) containing \(W\) remain. This is analogous to condition (i) for fine sheaves. In fact, if one were to impose that \(S_V(W) = 0\) whenever \(W \not\subseteq V\), then the associated sheaves would not be acyclic. For a poset $\mathcal{A}$ equipped with the upper Alexandroff topology, denoted $X^\mathcal{A}$ (cf. Definition \ref{Alexandrov-topology}), one can build sheaves $F$ with non trivial \(H^{1}(X^\mathcal{A}; F)\); it can be seen in the example of the poset \(\mathcal{A}\) that has four elements \(\alpha, \beta, \gamma, \delta\) and such that \(\alpha \geq \beta\), \(\alpha \geq \gamma\), and \(\alpha \geq \delta\) (see Appendix \ref{counter-example} for an explicit example).  However, for cosheaves from a poset with the lower Alexander topology to the opposite category of abelian groups (i.e., when the final category is opposite, and the gluing property with limits is replaced by colimits \cite{Curry2013}), the existence of such a decomposition with \(S_{V}(W) = 0\) when \(W \not\subseteq V\) implies Čech acyclicity. It is expected that from this remark, one could propose a definition of lower extra-fine cosheaves that are acyclic for Čech cohomology. However, we prefer not to focus on cosheaves, as our aim is to provide a topological interpretation and generalization of Kellerer's result on the linear marginal extension problem, which relies on upper extra-fine sheaves. 
\end{remark}

\begin{theorem} \label{thm1}
An upper extra-fine presheaf is Čech acyclic.
\end{theorem}

\begin{proof}

Let $\mathcal{V}$ be an open cover of $X$. Then, there exists a cover $\mathcal{U}$, stable under intersection, that is finer than $\mathcal{V}$ and that satisfies the other properties in Definition \ref{def:upper-extra-fine-presheaf}. Denote by $\cat{C}_{\subseteq}(\mathcal U)$ the posetal category associated with the cover $\mathcal U$, partially ordered by inclusions. Let $S_V$, $V \in \mathcal{U}$, be a collection of contravariant functors over $\mathcal{U}$  generated respectively by $V$ such that $F \cong \bigoplus_{V \in \mathcal{U}} S_V$.

The chain complex $\delta_{F}: C^{n}(\mathcal{U}, F) \to C^{n+1}(\mathcal{U}, F)$ that defines the Čech cohomology groups $H^{\ast}(\mathcal{U}; F)$ simplifies as we will see now.  $\prod_{a_0, \ldots, a_n \in K_n(\mathcal{U})} F\left(\bigcap_{k=0}^{n} a_k\right)$  identifies with $ \prod_{a_0, \ldots, a_n \in K_n(\mathcal{U})} \bigoplus_{V \in \mathcal{U}}S_V\left(\bigcap_{k=0}^{n} a_k\right)$; let $a_0,...,a_{n+1}\in K_{n+1}(\mathcal{U})$ and $x(a_0,...,a_{n+1}) \in \bigoplus_{V \in \mathcal{U}}S_V\left(\bigcap_{k=0}^{n} a_k\right)$, then:

\[
\delta^{n+1}_{n,F} \bigoplus_{V\in \mathcal{U}} x_V (a_0, \ldots, a_{n+1}) = \bigoplus_{V}\sum_{i=0}^{n+1} (-1)^{i} x_{V}(a_0, \ldots, \widehat{a_i}, \ldots, a_{n+1}) \big\vert_{\bigcap_{k=0}^{n+1} a_k}.
\]

Define $\delta_{n,S_V}^{n+1}:\prod_{a_0, \ldots, a_n \in K_n(\mathcal{U})} S_V\left(\bigcap_{k=0}^{n} a_k\right)\to \prod_{a_0, \ldots, a_{n+1}\in K_{n+1}(\mathcal{U})} S_V\left(\bigcap_{k=0}^{n+1} a_k\right)$ as, for any $a_0,...,a_{n+1}\in K_{n+1}(\mathcal{U})$:

$$\delta_{n,S_V}^{n+1}(x_{V})(a_0,...,a_{n+1})=\sum_{i=0}^{n+1} (-1)^{i} x_{V}(a_0, \ldots, \widehat{a_i}, \ldots, a_{n+1}) \big\vert_{\bigcap_{k=0}^{n+1} a_k}. $$

with $ x_V \in  \prod_{a_0, \ldots, a_n \in K_n(\mathcal{U})} S_V\left(\bigcap_{k=0}^{n} a_k\right)$.

One has $\delta_{n,F}^{n+1}\left(\bigoplus_{V} x_V\right) \cong \bigoplus_{V \in \mathcal{U}} \delta_{n,S_V}^{n+1}(x_V)$. Therefore, $$H^{n}(\mathcal{U}; F) \cong \bigoplus_{V \in \mathcal{U}} \ker \delta^{n+1}_{n,S_V} / \im \delta^{n}_{n-1,S_V}.$$

Let us now show that for any $n\geq 1$, $\ker \delta^{n+1}_{n,S_V}= \im \delta^{n}_{n-1,S_V}$.

Let us first remark that for any $V\in \mathcal{U}$, $S_V$ is isomorphic to the contravariant functor $T_{V}$ defined on objects as $T_{V}(U):= S_V(V)$ for all $V\subseteq  U\in \mathcal{U}$ and $T_{V}(U):=0$ if $V\not\subseteq U$ and its morphisms are for $U, U_1\in \mathcal{U}$ such that $V\subseteq  U\subseteq  U_1$, $T_{U,U_1}:=\id$. \\

Set $a=V$. Let \(n \geq 1\) and let \(\psi \in \prod_{a_0, \ldots, a_n \in K_n(\mathcal{U})} T_V\left(\bigcap_{k=0}^{n} a_k\right)\) such that \(\delta_{n, S_V}^{n+1} \psi = 0\). If $\bigcap_{k=1}^{n+1} a_k$  contains \(a\),
\[
\psi(a_1, a_2, \ldots, a_{n+1}) = \sum_{k=1}^{n+1} (-1)^{k+1} \psi(a, a_1, \ldots, \widehat{a_k}, \ldots, a_{n+1}).
\]

If \(\bigcap_{k=1}^{n+1} a_k\) does not contain \(a\), then \(\psi(a_1, a_2, \ldots, a_{n+1}) = 0\) since $T_V$ is generated by $V$. \\

Define \(\varphi \in \prod_{a_0, \ldots, a_{n-1} \in K_{n-1}(\mathcal{U})} S_V\left(\bigcap_{k=0}^{n-1} a_k\right)\) as the \((n-1)\)-cochain given by, for any \(\beta_1, \ldots, \beta_n \in \mathcal{A}\) such that \(\bigcap_{k=1}^{n} \beta_k \supseteq a\) and $\bigcap_{k=1}^{n+1} \beta_k\neq \emptyset$,
\begin{equation}\label{eq:def_varphi}
\varphi(\beta_1, \ldots, \beta_n) = \psi(a, \beta_1, \ldots, \beta_{n}).
\end{equation}

The previous equation is well defined as $\bigcap_{k=1}^{n+1} a_k \cap a\neq \emptyset $. If \(\bigcap_{k=1}^{n} \beta_k \not\subseteq a\) and $\bigcap_{k=1}^{n} \beta_k \neq \emptyset$ then we set \(\varphi(\beta_1, \ldots, \beta_n) = 0\).

By combining \eqref{eq:def_varphi} with the definition of $\delta_{{n-1}, T_V}^n$, we get

\[
\delta \varphi(a_1, a_2, \ldots, a_{n+1}) = \sum_{k=1}^{n+1} (-1)^{k+1} \psi(a, a_1, \ldots, \widehat{a_k}, \ldots, a_{n+1}),
\]
therefore $\delta \varphi = \psi$ when \(\bigcap_{k=1}^{n+1} a_k \supseteq a\).
If \(\wedge_{j=1}^{n+1} a_j \not\supseteq a\), then \(\psi(a, a_1, \ldots, \widehat{a_k}, \ldots, a_{n+1}) \big| \wedge_{j=1}^{n+1} a_j = 0\) for any \(k = 1, \ldots, n+1\). Therefore, \(\delta^n_{n-1, S_V} \varphi(a_1, a_2, \ldots, a_{n+1}) = 0\). We just showed that \(\ker \delta^{n+1}_{n, S_V} = \operatorname{im} \delta^{n}_{n-1, S_V}\) for any \(n \geq 1\).\\

Therefore, \(H^{n}(\mathcal{U}; F) = 0\) for any \(n \geq 1\), and so \(H^{n}(X; F) = \varinjlim H^{n}(\mathcal{U}; F) = 0\).
\end{proof}

\section{Alexandroff topologies and sheaves}\label{sec:posets_alexandroff}

\subsection{Basic definitions} 

{\color{black} As we saw in the introduction, a poset $\mathcal{A}$ can be seen as a posetal category in two ways. In particular, $\cat{C}_\geq(\mathcal{A})$ denotes the category whose arrows go in the reverse direction to the relation $\leq$, i.e.\ we have an arrow $\alpha \to \beta$ whenever $\beta \leq \alpha$. We denote by $\alpha\wedge\beta$ (resp.\ $\alpha\vee\beta$) the meet (resp.\ join) of two elements in the poset $\mathcal{A}$, when it exists. The meet (resp.\ join) corresponds to the coproduct in $\cat{C}_{\geq}(\mathcal{A})$ (resp.\ the product).}
 
 We shall not suppose that our categories have all finite coproducts, but sometimes we impose the following \emph{conditional existence of coproducts}: for any $\alpha, \alpha'\in  \mathcal A$, if there exists $\omega\in  \mathcal A$ such that $\alpha\to \omega $ and $\alpha'\to \omega$, then $\alpha \wedge \alpha'$ exists. \emph{Conditional existence of products} is defined dually; this property played an important role in \cite{Vigneaux2020information} to define \emph{information cohomology}.

\begin{example} Let $\mathcal K$ be an abstract simplicial complex i.e. a family of subsets of a given set $I$ such that if $\alpha\in \mathcal K$, then every subset of $\alpha$ is also in $\mathcal K$. We see $\mathcal K$ as a poset setting $\beta\leq \alpha$ iff $\beta\subseteq \alpha$. In this setting $\alpha \to \beta$ whenever $\beta\subseteq \alpha$. In this structure all coproducts (meets) exist, $\alpha\wedge \beta= \alpha\cap \beta$, but products (joins) only exist conditionally: $\alpha \vee \beta$ corresponds to the smallest set in $\mathcal K$ than contains $\alpha$ and $\beta$.
\end{example}

\begin{definition}[Alexandroff's topologies on a poset]\label{Alexandrov-topology}
P. S. Alexandroff introduced a natural topology on a poset $\mathcal{A}$, given by a basis of open sets called \textit{down-sets} $U_\alpha=\set{\beta }{\beta\leq \alpha }$, indexed by $\alpha\in  \mathcal A$; there are sometimes denoted as $\alpha \downarrow$ in the poset literature. The open subsets for the associated topology are the sets $U\subseteq \mathcal{A}$ such that whenever $b\leq a$ with $a\in U$ then $b\in U$; those opens subsets  are called the \textit{lower-sets} of $\mathcal{A}$.

We will name this topology the \emph{lower Alexandroff topology (A-topology)} of \(\mathcal{A}\), and denote \(X_{\mathcal{A}}\) as the topological space obtained in this way.

Dually, the \emph{up-sets} are $U^{\beta}=\set{\alpha}{\beta\leq \alpha}$, indexed by objects $\beta\in \mathcal A$; those sets are sometimes denoted as \(\beta \uparrow\) in the literature. They form the basis
of a topology that we call \emph{upper A-topology} of $\mathcal{A}$. The corresponding topological space is denoted $X^{\mathcal A}$.
\end{definition}

Remark that if $\alpha\rightarrow\beta$ then $U_\alpha\supseteq U_\beta$ and $U^{\alpha}\subseteq U^{\beta}$. Also, whenever  $\cat{C}_\geq(\mathcal A)$ possesses conditional products (i.e. $\mathcal{A}$ possesses conditional joins)
and $U^\alpha\cap U^{\alpha'}$ is non-empty, one has $U^\alpha\cap U^{\alpha'}=U^{\alpha\vee \alpha'}$; the element
$\alpha \vee \alpha'$ 
in $U^\alpha\cap U^{\alpha'}$ is greater than both $\alpha$ and $\alpha'$, 
and lower than any $\beta \in U^\alpha \cap U^{\alpha'}$. Dually, whenever \(\mathcal{A}\) possesses conditional meets, then if \(U_\alpha \cap U_{\alpha'}\) is non-empty, we have \(U_\alpha \cap U_{\alpha'} = U_{\alpha \wedge \alpha'}\).

A general reference for Alexandroff spaces, finite topological spaces, and their relations to simplical complexes is \cite{Barmak2011}. For instance, the reader can find there the following result.

\begin{lemma}[\protect{\cite[Prop.~1.2.1]{Barmak2011}}] Let $\mathcal A, \mathcal B$ be posets. A map $f:\A\to \mathcal B$ is order preserving (equivalently, defines a covariant functor from $\mathcal C(\mathcal A)$ to  $\mathcal C(\mathcal B$)) if and only if $f$ is continuous for the lower (or upper) A-topology.
\end{lemma}

It is a classical result, corollary of a more general result of Grothendieck \emph{et al.}\cite{Artin1972}, that the functors on a poset $\mathcal{A}$ can be seen as classical sheaves on the associated topological space. This equivalence is a consequence  of the Comparison lemma, {see \cite[Thm.~1.1.8]{Caramello2019theories}, a restatement of the original \cite[Thm.~III.4.1]{Artin1972-2}}. However, we give {just below} a simple explicit proof.

\begin{proposition}\label{prop:functors_are_sheaves}
Let $\A$ be a poset. For every covariant functor $F$ from the posetal category $\cat{C}_{\geq}(\mathcal{A})$ to the category of sets $\mathcal{S}$, there is a unique sheaf $\widetilde F$ on $X_{\mathcal A}$ such that $\widetilde F(U_\alpha) = F(\alpha)$ for all $\alpha \in {\mathcal A}$. 
\end{proposition}

\begin{proof}
Let $F$ be a covariant functor on $\cat{C}_{\geq}(\mathcal A$) with arrows $\alpha \to \beta$ whenever $\beta\leq\alpha$ with $\alpha,\beta\in \mathcal{A}$, whose image on
objects and morphisms 
we write $F(\alpha)$ and $F_{\alpha \to \beta}$ respectively.
If $F$ is extended by a sheaf $\widetilde F$ on $X_{\mathcal A}$,  
then for any open set  $U=\bigcup_{\alpha\in U} U_\alpha$,
we must have $\widetilde F(U)=\varprojlim_{\alpha \in U} F(\alpha)$. 
This proves the uniqueness of the extension; let us
recall that this projective limit is 
canonically identified with the set of \emph{coherent sections} $(s_\alpha)_{\alpha \in U}$ in $\prod_{\alpha \in { U}} F(\alpha)$ such that $F_{\alpha \to \beta}(s_\alpha) = F_{\alpha' \to \beta}(s_{\alpha'})$
for all $\alpha,\alpha' \in U$ and all $\beta \in U_\alpha\cap U_{\alpha'}$.
Reciprocally, letting $\widetilde F(U)=\varprojlim_{\alpha \in U} F(\alpha)$ defines a presheaf $\widetilde F$ on $X_{\mathcal A}$ (i.e. for the lower A-topology). Let us verify that $\widetilde F$ is a sheaf, 
fixing an open $U$ covered by a collection of opens ${\mathcal U}$.
First, if $s,s' \in \widetilde F(U)$ satisfy $s|V = s'|V$ for all $V\in \mathcal{U}$, then in particular $s_\alpha = s'_\alpha$ for all $\alpha \in U$ and $s = s'$ which proves the locality axiom.  
Concerning the gluing axiom, 
given a collection $s^V \in \widetilde F(V)$ for $V \in {\mathcal U}$ 
such that $s^V|{V\cap W} = s^W|{V\cap W}$ whenever $V,W\in \mathcal U$ have nonempty intersection, letting 
$s_\alpha = s^{V}_\alpha$ for any $V$ containing $\alpha$ 
defines a coherent section in $\widetilde F(U)$. 
\end{proof}

Conversely, if we start with a sheaf $\widetilde F$ on $X_{\A}$, we can define a covariant functor $F$ on $\cat{C}_{\geq}(\A)$ such that $F(\alpha) =\widetilde F(U_\alpha)$ and $F(\alpha \to \beta)$ is the restriction map from $F(U_\alpha)$ to $F(U_\beta)$. In this sense, below we use the same letter $F$ both for the covariant functor on the poset and the corresponding sheaf on the Alexandroff space. A similar result holds for the upper topology, but in this case the sheaves over $X^\A$ are in correspondence with contravariant functors on $\cat{C}_{\geq}(\mathcal A)$.

Neither the (conditional) existence of coproducts  or products in $\cat{C}_{\geq}(\A)$ nor any finiteness hypothesis involving $\A$ are used in the previous proof. 

\noindent \begin{remark}\label{categoryequivalence} The construction in the above proposition establishes an equivalence of the category $\operatorname{PSh}(\cat{C}_\leq(\mathcal A))$ of covariant functors of sets, abelian groups, or \(\mathbb{K}\)-vector spaces, over $\cat{C}_{\geq}(\mathcal{A})$, having for morphisms the natural transformations,  with the category of sheaves {\color{black}$\operatorname{Sh}(X_\mathcal{A})$} over $X_\mathcal{A}$,
having for morphisms the natural transformations. In fact these two categories are isomorphic toposes cf. {\cite[Prop.~4.1]{Caramello2011Stone},} \cite{Artin1972-2}, \cite{Moerdijk2006}.
\end{remark}

\begin{remark}[Functoriality]
In the case of posets and their associated Alexandroff topologies, the direct images and inverse images of sheaves (or presheaves)
are easy to handle. For definitions, we refer the reader to Appendix \ref{appendix1:topology}. Let $f:\mathcal{A}\rightarrow \mathcal{B}$ be a morphism of posets;  $f$ is continuous for the lower and the upper
topologies.

\begin{itemize}
    \item  If $F$ is a sheaf of sets $X_{\A}$, its direct image $f_*F$ is defined by $(f_{*}F)(U_\beta)=
F(f^{-1}(U_\beta))$, where the open set $f^{-1}(U_\beta)$ is the set of elements $\alpha\in \mathcal{A}$ such that $\beta \geq f(\alpha)$.

 If $U_\beta$ does not meet the image of $f$, this is the empty set. For sheaves of abelian groups, and $U_\beta$ non-intersecting $f(\mathcal{A})$, we have $f_*F(U_\beta)=0$. For instance, if $\mathcal{A}$ is a
sub-poset of $\mathcal{B}$, and $J$ the injection: $J_*F$ coincides with $F$ on $\mathcal{A}$.

\item If $G$ is a sheaf of sets on $\mathcal{B}$ for the lower $A$-topology, its inverse image is defined by the formula
$(f^{*}G)(\alpha)=G(f(\alpha))$, which gives the stalk at $\alpha$.
\end{itemize}

\end{remark}

\subsection{Upper extra-fine presheaves on posets}\label{sec:extrafine_posets}

Consider a poset $\mathcal A$, and the induced topological space $X_{\mathcal A}$ whose underlying set is $\A$, equipped with the lower A-topology. Let us denote by $\mathcal{U}_{\mathcal{A}}$ the  covering of $X_\mathcal A$ by the open sets $U_\alpha$ for $\alpha\in\mathcal{A}$. By definition of the lower A-topology, \(\mathcal{U}_{\mathcal{A}}\) refines any other open covering as the \(U_\alpha\) form a basis of the lower A-topology. Let us give another proof of this fact: for any open cover \(\mathcal{U}\) of \(X_\mathcal{A}\) and for any \(\alpha \in \mathcal{A}\), there exists \(V \in \mathcal{U}\) such that \(\alpha \in V\); as \(V\) is a lower set, \(U_\alpha \subseteq V\). If, furthermore, \(\mathcal{A}\) possesses conditional meets, then, due to the minimality (i.e., there is no strictly finer cover) of the open covering \(\mathcal{U}_\mathcal{A}\), the existence of a decomposition of \(\widetilde{F}\vert_{\mathcal{U}_{\mathcal{A}}}\) into contravariant functors \(S_a\) generated by $a$ for \(a \in \mathcal{A}\) is equivalent to \(\widetilde{F}\) being upper extra-fine; then it would imply by Theorem \ref{thm1} that $\widetilde{F}$ is acyclic. However, in general, \(\mathcal{U}_{\mathcal{A}}\) is not the only finest open covering of \(\mathcal{A}\); for example, if \(\mathcal{A} = (\mathbb{N}, \leq)\), the natural numbers ordered by inclusion, the lower sets associated with even numbers define a minimal cover under refinement.
$\widetilde{F}\vert_{\mathcal{U}_{\mathcal{A}}}$ is isomorphic to $F$. In the applications we have in mind, described in the introduction and Section \ref{sec:free-sheaves}, \(F\vert_{\mathcal{U}_\mathcal{A}}\) can be decomposed into a sum of functors over $\cat{C}_{\geq}(\mathcal \A)$ generated by elements $a\in \mathcal{U}_\mathcal{A}$; in this case, we say that \(F\) is \emph{canonically upper extra-fine} over \(X_\mathcal{A}\). In the previous sentence, we implicitly identified \(\mathcal{U}_{\mathcal{A}}\) with \(\mathcal{A}\) and presheaves on \(\mathcal{U}_{\mathcal{A}}\) with functors on \(\cat{C}_{\geq}(\mathcal{A})\).

In the next lemma, we will provide a direct proof that a covariant functor \(S_a\) from \(\cat{C}_{\geq}(\mathcal{A})\) generated at $a$, taking values in abelian groups, is \v{C}ech acyclic and characterize the section of such functor.

\begin{lemma}\label{lemma-extra-fine}
 Let $\mathcal{A}$ be a poset with finite meets. Let $S_a$ be a covariant functor over $\cat{C}_{\geq}(\mathcal{A})$ generated by $a\in \mathcal{A}$, and  $\widetilde{S_a}$ denote its unique extension as a sheaf on $X_{\mathcal{A}}$ equipped with the \emph{lower} A-topology. $S_a$ is acyclic and $H^{0}(X_{\mathcal A},\widetilde{S_a})\cong S_a(a)$.
\end{lemma}

\begin{proof}
Let us first remark that $S_a$ is isomorphic to the covariant functor $T$ defined as $T(\beta):= S_a(a)$ for all $\beta \geq a$ and $T(\beta):=0$ if $\beta \not \geq a$ and for $\beta, \beta_1\in \mathcal{A}$ such that $a\leq \beta\leq \beta_1$, $T_{\beta,\beta_1}:=\id$.
Let us show that $\widetilde{T}$ is acyclic. $\mathcal{U}_{\mathcal{A}}=\{U_\alpha\}_{\alpha\in \mathcal{A}}$ is finer than any cover. It is sufficient to prove the acyclicity for the covering by the $\mathcal{U}_{\mathcal{A}}$ as $\varinjlim H^{n}(X_\mathcal{A}; T)= H^{n}(\mathcal{U}_{\mathcal{A}}; T)$, with $n\geq 0$. Let \(\psi \in C^{n}(\mathcal{U}_{\mathcal{A}}, \widetilde{T})\) be a cocycle. Then, for any \(\alpha_1, \alpha_2, \ldots, \alpha_{n+1}\) that are all greater than \(a\),

$$\psi(\alpha_1,\alpha_2,...,\alpha_{n+1})= \sum_{k=1}^{n+1} (-1)^{k+1} \psi(a,\alpha_1,...,\widehat{\alpha_k},...,\alpha_{n+1})$$

If $\wedge_{k=1}^{n+1}\alpha_k$ is not greater than $a$ then by definition $\psi(\alpha_1,\alpha_2,...,\alpha_{n+1})=0$. \\

Define \(\varphi\) as the \((n-1)\)-cochain given by, for any \(\beta_1, \ldots, \beta_n \in \mathcal{A}\) such that \(\wedge_{k=1}^{n} \beta_k \geq a\),
\begin{equation}
\varphi(\beta_1, \ldots, \beta_n) = \psi(a, \beta_1, \ldots, \beta_{n}).
\end{equation}

If \(\wedge_{k=1}^{n} \beta_k \not\geq a\), then $\varphi(\beta_1, \ldots, \beta_n) = 0$.

Then, if \(\wedge_{k=1}^{n+1} \alpha_k \geq a\),  
\[
\delta \varphi(\alpha_1, \alpha_2, \ldots, \alpha_{n+1}) = \sum_{k=1}^{n+1} (-1)^{k+1} \psi(a,\alpha_1, \ldots, \widehat{\alpha_k}, \ldots, \alpha_{n+1}).
\]

If \(\wedge_{j=1}^{n+1} \alpha_k \not\geq a\), then \(\psi(a,\alpha_1, \ldots, \widehat{\alpha_k}, \ldots, \alpha_{n+1}) \vert \wedge_{j=1}^{n+1} \alpha_j = 0\) for any \(k = 1, \ldots, n+1\). Therefore, \(\delta \varphi(\alpha_1, \alpha_2, \ldots, \alpha_{n+1}) = 0\). We just showed that $H^{n}(X_\A, \widetilde{S_a})=0$ for any $n\geq 1$.

$H^{0}(\mathcal{U}_{\mathcal{A}}; \widetilde{S_a})$ is the set of global sections over $\mathcal{U}_{\mathcal{A}}$, which corresponds to $\lim_{\beta\in \mathcal{A}} S_a(\beta)$. The projection that sends $\lim_{\beta\in \mathcal{A}} S_a(\beta)$ to $S_a(a)$, defined as $(x_b)_{b \in \mathcal{A}} \mapsto x_a$, is an isomorphism, and so $\lim_{\beta\in \mathcal{A}} S_a(\beta) \cong S_a(a)$.
\end{proof}

\begin{remark}
Lemma \ref{lemma-extra-fine} and its proof allows one to get more insight on the proof of Theorem \ref{thm1}. Consider now a topological space $X$ and one of its covers $\mathcal{V}$, closed under finite intersections, that refines a cover $\mathcal{U}$. Assume that $\mathcal{U}$ contains no empty set; such an assumption is not restrictive, as the open cover obtained by removing the empty sets from \(\mathcal{U}\) is finer. Denote $\mathcal{A}$ as the poset that underlies \(\mathcal{V}\) and consider an upper extra-fine presheaf $F$ over $X$. Then a functor generated by $V \in \mathcal{V}$, denoted $S_V$, can be extended into a sheaf on $X_{\mathcal{V}}$, denoted $\widetilde{S}_V$. However, we do not see any way for it to be extended into a sheaf or even a presheaf over $X$. The chain complex defined by $\delta^{n+1}_{n,F}: C^{n}(\mathcal{V}, F) \to C^{n+1}(\mathcal{V}, F)$ for $n \geq 0$ is a direct sum over $V$ of 
$\delta^{n+1}_{n,S_V}$; recall that $\delta_{n,S_V}^{n+1}:\prod_{\alpha_0, \ldots, \alpha_n \in K_n(\mathcal{U})} S_V\left(\bigcap_{k=0}^{n} \alpha_k\right)\to \prod_{\alpha_0, \ldots, \alpha_{n+1}\in K_{n+1}(\mathcal{U})} S_V\left(\bigcap_{k=0}^{n+1} \alpha_k\right)$ is defined as, for any $(a_0,...,a_{n+1})\in K_{n+1}(\mathcal{U})$:

$$\delta_{n,S_V}^{n+1}(x_V(a_0,...,a_{n+1}))=\sum_{i=0}^{n+1} (-1)^{i} x_V(a_0, \ldots, \widehat{a_i}, \ldots, a_{n+1}) \big\vert_{\bigcap_{k=0}^{n+1} a_k}. $$ 

One remarks that $\delta_{n,S_V}^{n+1}$ is isomorphic to $\delta_{n,\widetilde{S_V}}^{n+1}:C^{n}(\mathcal{U}_{\A}; \widetilde{S_V})\to C^{n+1}(\mathcal{U}_{\A}; \widetilde{S_V})$. Therefore, the proof that \(F\) is acyclic relies on the fact that, when restricted to the cover \(\mathcal{U}\), which is closed under finite intersection, the differential of the Čech cohomology for this cover coincides with the differential of the Čech cohomology of \(\widetilde{F}\vert_\mathcal{A}\); $\widetilde{F}\vert_\mathcal{A}$ being the sheafification of \(F\) over \(X_\mathcal{A}\) restricted to the canonical cover \(\mathcal{U}_\mathcal{A}\) of \(\mathcal{A}\). It is this identification that allows the acyclicity of the sheaf supported on elements of \(\mathcal{A}\) to be lifted to the acyclicity of \(F\) for the covering \(\mathcal{U}\), which in turn implies the acyclicity of \(F\) over \(X\).

\end{remark}

\begin{proposition}\label{grouphzero} Let $\A$ be a poset with finite meets. Let $F$ be a canonically upper extra-fine sheaf over $X_{\mathcal A}$, where $X_{\mathcal A}$ denotes the topological space defined by $\mathcal A$ equipped with its lower A-topology. Then, 
$$H^{0}(X_{\mathcal{A}}; F) = \bigoplus_{\alpha \in \mathcal{A}} H^{0}(\mathcal{U}_{\mathcal{A}}; S_\alpha) \cong \bigoplus_{\alpha \in \mathcal{A}} S_\alpha(\alpha).$$

\end{proposition}
\begin{proof}
Recall that $H^0(X_\A; F) = H^0(\mathcal{U}_\A; F)$ as $\mathcal{U}_\A$ is a maximal open covering. $\delta^{1}_{0,F}: C^{0}(\mathcal{U}_\A; F) \to C^{1}(\mathcal{U}_\A; F)$ is a direct sum $\bigoplus_{a \in \mathcal{A}} \delta_{S_a}$, where $\delta_{S_a}: C^{0}(\mathcal{U}_\A; S_a) \to C^{1}(\mathcal{U}_\A; S_a)$. Therefore, $\ker \delta_F \cong \bigoplus_{a \in \mathcal{A}} \ker \delta_{S_a}$ and by Lemma \ref{lemma-extra-fine}, $\ker \delta_{S_a}= \lim_{\beta \in \A} S_a(\beta)$ $\cong S_a(a)$.
\end{proof}

\begin{remark} In the above results, the groups $F(\alpha)=F(U_\alpha)$, for $\alpha\in \mathcal{A}$, or $F(U)$, for $U\in \mathcal{U}$, are not supposed finitely generated.
This is a good point because, starting with a covariant functor $F$ of finitely generated abelian groups over $\cat{C}_{\geq}(\mathcal{A})$, it can happen that the sheaf extending $F$ to $X_\mathcal{A}$
in Proposition \ref{prop:functors_are_sheaves} is not made of finitely generated abelian groups.
\end{remark}

\subsection{Finiteness conditions}\label{remark:finiteness_conditions}
In what follows we will sometimes consider posets $\mathcal{A}$ that satisfy some finiteness condition. We say that \(\mathcal{A}\) is \emph{locally finite} if, for any $\beta\leq \alpha$, the intersection \(U^\beta_\alpha = U_\alpha \cap U^\beta\) referred to as the interval \([\beta, \alpha]\) in the literature, is finite. In other terms, there exist only a finite number of chains in $\cat{C}_\geq(\mathcal A)$, i.e. sequences of arrows $\alpha \xrightarrow{\neq} \beta_1\xrightarrow{\neq} \beta_2 \xrightarrow{\neq} \cdots \xrightarrow{\neq}\beta$, without repetition beginning at \(\alpha\) and ending at \(\beta\). A stronger property is (lower) \emph{closure finite}, meaning that every $U_\alpha$ is finite. This is the case for the poset associated with a CW complex  (cf. \cite{Whitehead1949}), when the partial order is reverse containment of cells: $\alpha \to \beta$ iff $\beta$ is a face of $\alpha$. In the next section, when proving the existence of an interaction decomposition, we consider the hypothesis of \emph{locally finite dimension}:  {for every object $\alpha $ of $\mathcal A$, the non-degenerate chains $\alpha \xrightarrow{\neq} \beta_1\xrightarrow{\neq} \beta_2 \xrightarrow{\neq} \cdots $  are all finite and  there is an upper bound for their length that depends only on $\alpha$; the least upper bound is called the \emph{dimension} of $\alpha$.}

Note that the condition of local finiteness is self-dual, i.e. it holds for $\mathcal{A}$ if and only if it holds for $\mathcal{A}^{op}$, the poset with reversed order. This is not the case for closure finiteness or locally finite dimension. 
The posets $\mathcal{A}$ and $\mathcal{A}^{op}$ are both of locally finite dimension if and only if there exists a number $d$ such that any sequence $\alpha\rightarrow...\rightarrow\beta$
of length bigger than $d+1$ has a repetition; in this case we say that $\mathcal{A}$ has finite dimension, or finite depth.

\indent The most elegant finiteness condition is lower well-foundedness: there exists no infinite chain without repetition; it is in this setting that the equivalence between condition \ref{condition_G} and the existence of an interaction decomposition was first proven (cf. \cite{Sergeant2019intersection}).

\section{Interaction decomposition}\label{sec:interaction-decomposition}

\subsection{Condition G and the equivalence theorem}\label{sec:interaction-decomposition-condition-G}

Let $\mathcal A$ be an arbitrary poset, and $(V, j)$ be a contravariant functor on $\cat{C}_{\geq}(\mathcal{A})$, where $a\to b$ whenever $b\leq a$, valued in the category of vector spaces over a  field $\mathbb{K}$. 
We suppose that for each $\alpha \to \beta$ in $\cat{C}_{\geq}(\mathcal A)$, the map $j_{\alpha\beta}: V_\beta \to V_\alpha$ is injective, and write $V_{\alpha\beta}$ instead of $j_{\alpha\beta}(V_\beta)$ in the following.

\begin{definition}\label{formulainteractiondecomposition}
An \emph{interaction decomposition} of a contravariant functor of injective linear maps, $(V,j)$, is a family of vector subspaces $S_\gamma \subseteq V_\gamma$ for $\gamma\in \mathcal{A}$,
such that
\begin{equation}
\forall \alpha \in \mathcal{A},\quad V_\alpha=\bigoplus_{{\alpha \to \beta}} j_{\alpha\beta} S_\beta.
\end{equation}
\end{definition}

Let us introduce, for every $\gamma\in \mathcal{A}$ 
and $\alpha \in {\mathcal A}$, the vector spaces
\begin{equation}
S^\gamma(\alpha) = \begin{cases}S_{\alpha\gamma}=j_{\alpha\gamma}S_\gamma &\text{if }\alpha \to \gamma \\
0 & \text{if } \alpha \nrightarrow\gamma\end{cases},
\end{equation}
which define subpresheaves $(S^{\gamma})_{\gamma \in {\mathcal A}}$
of $V$ over $\cat{C}_{\geq}({\mathcal A})$.
The interaction decomposition corresponds to a decomposition of the contravariant functor $V$:
\begin{equation}\label{eq:sheaf-interaction-decomposition}
V=\bigoplus_{\gamma \in \mathcal{A}} S^{\gamma}.
\end{equation}

We will see that an interaction decomposition implies the following statement, that we call \emph{condition G}:

\nopagebreak
\begin{equation}\tag{G}\label{condition_G}
\forall \alpha, \beta \in \mathcal A \text{ such that }\alpha \to \beta, \quad V_{\alpha\beta} \cap \left( \sum_{\substack{\gamma:  \alpha \xrightarrow{\neq} \gamma \\ \gamma  \not\to  \beta}} V_{\alpha\gamma} \right) \subseteq \sum_{\substack{\gamma:   \alpha \xrightarrow{\neq} \gamma \\  \beta  \xrightarrow{\neq}\gamma }} V_{\alpha\gamma},
\end{equation}

where $\beta  \xrightarrow{\neq}\gamma$ means that $\beta \to \gamma$ and $\beta \neq \gamma$.

Condition \ref{condition_G} can be re-expressed as follows: for any $\beta \lneqq \alpha$, if there exists a collection $y_{\gamma} \in V_{\gamma}$ for $\gamma \not\geq \beta$ and $\gamma \lneqq \alpha$ such that 
\[
x = \sum_{\substack{\gamma \lneqq \alpha\\\gamma \not\geq \beta}} j_{\alpha\gamma}(y_\gamma),
\]
and if there exists $z \in V_\beta$ such that $x = j_{\alpha\beta}(z)$,
then there exists a collection $(v_\gamma \in V_\gamma; \gamma \lneqq \beta)$ such that

\[
x = \sum_{\gamma \lneqq \beta} j_{\alpha\gamma}(v_\gamma).
\]

{The condition \ref{condition_G} implies in turn the existence of an interaction decomposition, if one supposes that the poset $\mathcal A$ is of local finite dimension  (see Section \ref{remark:finiteness_conditions}).
}

\begin{theorem}\label{thm2:extrafine_and_conditionG}
Let $V$ be a contravariant functor from $\cat{C}_{\geq}(\mathcal{A})$ to the category of $\mathbb{K}$-vector spaces, with maps $j_{\alpha\beta}:V_\beta\to V_\alpha$ that are injective.
\begin{enumerate}
\item  If the condition \ref{condition_G} is satisfied and if $\mathcal{A}$ is of locally finite dimension, the presheaf $V$ on $\mathcal{A}$ admits an interaction decomposition;
\item If the presheaf $V$ on $\mathcal{A}$ admits an interaction decomposition, the condition \ref{condition_G} is satisfied.
\end{enumerate}
\end{theorem}

\begin{proof}
In this proof, we consider $\mathcal{C}_\ge (\mathcal{A})$, i.e., $\alpha \to \beta$ whenever $\beta \le \alpha$. Let us prove the first claim. For each $\alpha \in \mathcal  A$, we define the \emph{boundary sum} $V'_\alpha=\sum_{\beta: \alpha \to \beta, \alpha \neq \beta} V_{\alpha\beta}$,
and we choose any supplementary space $S_\alpha$ of it.
Hence it remains to prove that $V'_\alpha$ is the direct sum  $\bigoplus_{\beta: \alpha \to \beta, \beta\neq \alpha} S_{\alpha\beta}$. We prove this by induction in the dimension of $\alpha$.

First, if $\alpha$ has dimension zero, then it is minimal, which means that $\alpha\to \beta$ implies $\alpha = \beta$. So $V_\alpha'=0$ and the claim is then trivially true.

Let us suppose now that the induction hypothesis holds true in dimension smaller or equal than $r-1$, for some $r\geq 1$, and consider $\alpha$ of dimension  $r$.

Let $B$ be the set of elements $\beta \in \mathcal A$ that cover $\alpha$, in the sense that $\alpha \to \beta$, $\alpha \neq \beta$, and there is no $\gamma\in \mathcal A$ such that $\alpha \to \gamma \to \beta$. And for $\beta\in B$, consider $x \in V_{\alpha\beta}$, and suppose it also belongs to the algebraic sum $\sum_{\gamma:\alpha\to\gamma,\beta\neq \gamma,\gamma \neq\alpha } V_{\alpha\gamma}$. As $\beta$ covers $\alpha$, we have
\begin{equation}
x \in \sum_{\substack{\gamma: \alpha \xrightarrow{\neq} \gamma \\ \gamma \not\to \beta }} V_{\alpha\gamma},
\end{equation}

indeed, the possibility $\alpha \xrightarrow{\neq} \gamma \xrightarrow{\neq} \beta$ is ruled out.

Then, applying the condition \ref{condition_G} to $x$, we deduce that $x$ belongs to the sum of $V_{\alpha\gamma}$ over the $\gamma\in \mathcal A$ such that $\alpha  \xrightarrow{\neq}\gamma$ and $  \beta  \xrightarrow{\neq}\gamma$, which coincides with $V'_{\alpha\beta}= j_{\alpha\beta} V_\beta'$, consequently
\begin{equation}
V'_\alpha=S_{\alpha\beta}\oplus (V'_{\alpha\beta}+\sum_{\beta'\neq\beta,\beta'\in B}V_{\alpha\beta'}).
\end{equation}
For $\beta'\in B$, $\beta'\neq \beta$, consider $x'\in V_{\alpha\beta'}$ and suppose it also belongs to the algebraic sum $\sum_{\beta"\in B, \beta"\neq \beta,\beta"\neq\beta' } V_{\alpha\gamma}$. As $\beta'$ covers $\alpha$, we have
\begin{equation}
x' \in \sum_{\substack{\gamma: \alpha \xrightarrow{\neq} \gamma \\ \gamma \not\to \beta' }} V_{\alpha\gamma}
\end{equation}
Then, applying the condition \ref{condition_G} to $x'$, we deduce that $x'$ belongs to the sum of $V_{\alpha\gamma}$ over the $\gamma\in \mathcal A$ such that $\alpha  \xrightarrow{\neq}\gamma$ and $  \beta'  \xrightarrow{\neq}\gamma$, which is the space $V'_{\alpha\beta'}= j_{\alpha\beta'} V'_{\beta'}$, consequently
\begin{equation}
V'_\alpha=S_{\alpha\beta}\oplus S_{\alpha\beta'}\oplus (V'_{\alpha\beta}+V'_{\alpha\beta'}+\sum_{\beta"\in B\setminus\{\beta,\beta'\}}V_{\alpha\beta"}).
\end{equation}
By (possibly transfinite) induction, we get
\begin{equation}
V'_\alpha=\bigoplus_{\beta\in B}S_{\alpha\beta}\oplus \sum_{\beta\in B}V'_{\alpha\beta}.
\end{equation}
Then we conclude by applying the induction hypothesis to the spaces $V'_{\beta}$, and the transitivity, $j_{\alpha\gamma}=j_{\alpha\beta}\circ j_{\beta\gamma}$.

We prove now the second claim. {Suppose that there exists an interaction decomposition of $V$, as in \eqref{formulainteractiondecomposition}.} Let us fix $\alpha\in \mathcal{A}$, look at $\alpha\to\beta$, and consider a vector $x$
in $S_{\alpha\beta}$. Suppose that this vector is equal to a finite sum $y_1+...+y_m$ of elements of $V_{\alpha\beta_1},...,V_{\alpha\beta_m}$ respectively, with
$\alpha \to \beta_i$ but $\beta \not\to \beta_i$. {Applying to each of these vectors $y_i$ the projector $e_{\beta|\alpha}$, we find zero}: however $e_{\beta|\alpha}(x)=x$ by definition of $S_{\alpha \beta}$, thus $x=0$. Now consider any vector $z$ in $V_{\alpha\beta}$; by interaction decomposition, $z$ is
(in a unique way) the sum of a vector $x\in S_{\alpha\beta}$, and a vector $y$ in the space $V'_{\alpha\beta}$, equal to the sum of the $S_{\alpha\gamma}$ for $\beta\xrightarrow{\neq}\gamma$,
which is included in the sum of all the $V_{\alpha\gamma}$ for $\beta\xrightarrow{\neq}\gamma$. Then the condition \ref{condition_G} is proved.
\end{proof}

The theorem above appears for the first time in \cite{Sergeant2019intersection}, by G. Sergeant-Perthuis. The author considers collections of vector subspaces indexed by a well-founded poset. The result we show here holds true if we replace the property of local finite dimension by the Noetherian property of
well-foundedness.

We also refer to condition \ref{condition_G} as the \emph{sum-intersection property}. Before the work of G. Sergeant-Perthuis, a particular case of this property appeared in the book of Lauritzen (see Proposition B.5 in Appendix B of \cite{Lauritzen}), as a corollary of the interaction decomposition in the context of factor subspaces and hierarchical model subspaces; factor spaces correspond to the free functors we introduce in Section \ref{sec:free-sheaves} of this paper. Lauritzen considers the poset \(\mathcal{A}\) to be the power set \(\mathcal{P}(I)\) of a finite set \(I\). However, it is in \cite{Sergeant2019intersection} that the full treatment and characterization of such interaction decomposition for functors over posets to vector spaces is considered; in this article, the vector spaces \(V(\alpha)\), \(\alpha \in \mathcal{A}\), belong to a fixed vector space \(W\). 
What makes condition \ref{condition_G} of interest is that it implies the interaction decomposition for presheaves over a posetal category (which is finite, locally of finite dimension, or well-founded) taking values in \(\mathbb{K}\)-vector spaces. This implication allows for the characterization of subcollections of factor spaces, i.e., sub-functors of the free functor, that have an interaction decomposition. However, it extends to more general collections of vector spaces. It is, in fact, closely related to characterizing projective objects in the category of contravariant functors from a posetal category \(\mathcal{C}_{\geq}(\mathcal{A})\) to vector spaces \cite{sergeant2021intersection}. In \cite{Sergeant2020presheaves}, a characterization of injective objects in the category of covariant functors from \(\mathcal{C}_{\geq}(\mathcal{A})\), with $\A$ (closure) finite, to \(\mathbb{K}\)-vector spaces was also given. Such injective objects relate closely to canonically upper extra-fine sheaves over \(X_{\mathcal{A}}\), as canonically upper extra-fine sheaves are particular examples of such.

\subsection{Linear duality and preduality}\label{sec:linear-predual}

We fix a commutative field $\mathbb{K}$, and the vector spaces considered are defined over this field. We suppose that the contravariant functor $V=(V,j)$ on $\cat{C}_{\geq}(\mathcal{A})$ admits an interaction decomposition. For each $\alpha\in \mathcal{A}$, let $V_\alpha^{*}$ be the (algebraic) dual vector space of the vector space $V_\alpha=V(\alpha)$; the transpose maps $^{t}j_{\alpha\beta}:V(\alpha)^{*}\to V(\beta)^{*}$ for $\alpha\to\beta$ define a covariant functor on $\cat{C}_{\geq}(\mathcal{A})$, and therefore a sheaf on {\color{black}$X_\A$}, for the lower A-topology on $\mathcal{A}$. But the dual $V_\alpha^{*}$ decomposes into $\prod_{\gamma \in \mathcal{A}} \left(S^{\gamma}(\alpha)\right)^{\ast}$ instead of a direct sum; we would need that the decomposition to be in direct sums if we were to apply the results regarding upper extra-fine sheaves discussed in Section \ref{sec:extra-fine}; this is where the closure finiteness property comes into play.

\begin{lemma}
    If $V$ admits an interaction decomposition and $\A$ is closure finite, i.e. \(U_a\) is finite for every \(a\in\mathcal{A}\), then the dual functor $V^*$ induces an extra-fine sheaf on $X_\mathcal A$. 
\end{lemma}
\begin{proof} In this case, we have $V= \bigoplus_{\gamma \in \mathcal A} S^\gamma$ as in \eqref{eq:sheaf-interaction-decomposition}, and the functor $(S^\gamma)^*: \mathcal C_{\subseteq}(\mathcal U_{\mathcal A})^{op} \to \mathcal{Vect}$ that associates to $\alpha \in \mathcal A$ the space $(S^{\gamma}(\alpha))^* $ is generated by $\gamma$ as in Definition \ref{def:functor_generated}: for instance, if $U_\gamma \not\subseteq U_{\alpha}$ then $\gamma \not\leq \alpha$ and indeed $(S^{\gamma}(\alpha))^* = 0$. It follows that
    $$V^* (\alpha) = \prod_{\gamma\in\mathcal{A}} (S^\gamma)^{\ast}(\alpha)= \bigoplus_{\gamma \in\mathcal{A}} (S^\gamma)^{\ast}(\alpha)$$ as there is only a finite number of $\gamma\leq \alpha$ such that $\left(S^\gamma(\alpha)\right)^\ast\neq 0$. This decomposition is natural in $\alpha$.
\end{proof}

In the applications that motivated this work, the sets $U_\alpha$ are finite subsets of a set $I$ that indexes observables, each $V_\alpha$ is the finite dimensional space of real-valued functions on a finite set $E_\alpha$, and the dual $V^*_\alpha$ is the space of measures on $E_\alpha$, where probabilities in particular live. We discuss this construction and its generalizations in Section \ref{sec:free-sheaves}. If $V$ is infinite dimensional, taking duals considerably increases the dimensionality of the spaces, so it is easier to induce a decomposition on the predual $F=V^\ast$, although we will also need additional hypotheses in that case. This is the situation we discuss in the rest of this section. 

Let us suppose that there exists a covariant functor \(F\) on \(\cat{C}_{\geq}(\mathcal{A})\) (equivalently, a sheaf on \(X_{\mathcal{A}}\)), with surjective arrows \(\pi^{\beta\alpha}\) for all pairs \(\alpha, \beta\) with \(\alpha \to \beta\), such that \(V_\alpha \cong F(\alpha)^{*}\), i.e. $V_\alpha$ is isomorphic to $F(\alpha)^{*}$, and \(j_{\alpha\beta} = {}^{t}\pi^{\beta\alpha}\), that is: 
\begin{equation}
    \forall v_\beta \in V_\beta, \forall f_\alpha \in F(\alpha) \quad  j_{\alpha\beta}(v_\beta)(f_\alpha)  = v_\beta (\pi^{\beta\alpha} (f_\alpha)).
\end{equation}

Then, for every \(\alpha \in \mathcal{A}\), the space \(F_\alpha = F(\alpha)\) embeds naturally in \(V_\alpha^{*}\), in such a manner that, for every pair \(\alpha, \beta\) with \(\alpha \to \beta\), $^{t}j_{\alpha\beta}$ induces the map \(\pi^{\beta\alpha}\) when restricted to \(F_\alpha\); therefore \(F\) can be seen as a covariant sub-functor of \(V^{*}\).

Consider the monomorphism $h: F \to V^{*} \cong \prod_{\gamma \in \mathcal{A}} S^{\gamma \ast}$ made of the inclusion $F_\alpha \to V_\alpha^{*}$. Let $pr_\alpha$ be the projection $pr_\alpha: \prod_{\gamma \in \mathcal{A}} S^{\gamma \ast} \to S^{\alpha \ast}$; it is the dual of the inclusion $i_\alpha: S^\alpha \to \bigoplus_{\gamma \in \mathcal{A}} S^\gamma$.

{\color{black}Let us denote $ \operatorname{im}(pr_\gamma \circ h)$ as $T_a$; $T_\gamma$ is a subfunctor of $S^{\gamma*}$. The first hypothesis we make is that the map $S^{\gamma*}_{\beta\alpha}$ is an isomorphism from $T_\gamma(\alpha)$ to $T_\gamma(\beta)$; in that case, $T_\gamma$ is generated at $\gamma$.}

{\color{black} The natural transformation $h: F\to V^*$ factors through the inclusion
$\mathcal{j}:\bigoplus_{\gamma\in \A} T_\gamma \hookrightarrow V^*$, i.e. there is a unique map
$\tilde{h}$ such that $h=\mathcal{j}\circ \tilde{h}$. This implies that
$\bigoplus_{\gamma\in \A} pr_\gamma\circ h$ sends $F$ into $\bigoplus_{\gamma\in \A} T_\gamma$.
However, that inclusion might be strict, i.e.
$\operatorname{im} h \subsetneq \bigoplus_{\gamma\in \A} T_\gamma$.
In Section~\ref{subsection:duality}, $T_\gamma$ is generated at $\gamma$, and furthermore, $h$ is an isomorphism from $F$ onto $V^*$; therefore, $T_\alpha = S^{\alpha *}$. This is why we prove the following corollary with the additional assumption that $h$ is onto $\bigoplus_{\gamma} T_\gamma$.}

{\color{black}
\begin{corollary}\label{dual:upper-extra-fine}
 Let $\A$ be a poset that is closure finite, i.e. such that every $U_\alpha$ is finite, with $\alpha\in \A$. Let $F$ be a covariant functor from a posetal category $\cat{C}_{\geq}(\mathcal{A})$ into the category of abelian groups. Let $V \cong F^{*}$; if $V$ admits an interaction decomposition, if the $T_\gamma$ are generated at $\gamma$ for all $\gamma \in \mathcal{A}$, and if $h \colon F \to V^*$ is onto $\bigoplus_{\gamma \in \mathcal{A}} T_\gamma$, then $F$ is acyclic for the Čech cohomology. Furthermore,

$$H^0(X_\A;F)\cong \bigoplus_{\gamma\in \A} T_\gamma(\gamma)$$

\end{corollary}
}

\begin{proof}
 $F$ is canonically upper extra-fine, as it is isomorphic to $\bigoplus_{\gamma\in \A} T_\gamma$. Therefore, by Theorem~\ref{thm1}, $F$ is acyclic and apply Proposition \ref{grouphzero}.
\end{proof}

\indent In the following section we will need a variant of the lemma \ref{lemma-extra-fine}, concerning the relative cohomology.  Suppose that $\mathcal{A}$ is a sub-poset of $\mathcal{B}$, and that we have an contravariant functor $T$ on $X_{\mathcal B}$ (i.e. for the lower $A$-topology {on} $\mathcal{B}$), which is generated by $\gamma\in \mathcal{B}$. Then we consider the sheaf $S$ on $X_{\mathcal{A}}$, obtained from the restriction of the respective functor $T$ over $\cat{C}_{\geq}(\mathcal{B})$ (defined stalkwise) to $\cat{C}_{\geq}(\mathcal{A})$.

\begin{remark}
    The previous hypotheses (that the $T_\gamma$ are generated at $\gamma$ and that $h$ is onto $\bigoplus_{\gamma} T_\gamma$) are automatically satisfied for contravariant functors \(V\) from a closure-finite poset \(\A\) to finite-dimensional vector spaces.
    On the other hand, if V is Hilbertian and if we replace the algebraic dual by the topological dual, then Corollary~\ref{dual:upper-extra-fine} extends. More precisely, if \(V_\alpha\) are Hilbert spaces and each \(S^\gamma(\alpha)\) is a Hilbert subspace of $V_\alpha$ such that $V=\bigoplus_{\gamma\in \A} S^\gamma$, then the topological dual \(V^{'}\) induces an extra-fine sheaf on $X_\A$. For further discussion of this setting, see \cite{sergeantperthuis2021interactiondecompositionhilbertspaces}.
\end{remark}

\noindent We assume that both $\mathcal{A}$ and $\mathcal{B}$ are closed by finite meets.

\begin{lemma}\label{lem5bis}
Under the above hypotheses, $\forall n\geq 1, H^{n}(\mathcal{B},\mathcal{A};T,S)=0$,  where we identify $\mathcal{U}_{\mathcal{B}}$ (resp. $\mathcal{U}_{\mathcal{A}}$) with $\mathcal{B}$ (resp. $\mathcal{A}$).
\end{lemma}
See Appendix \ref{appendix3:cech} for the definition of  relative cohomology and the associated complexes.
\begin{proof}
It is sufficient to prove the result for
the cohomology of the covering by the open sets $\{U_\beta\}_{\beta\in \mathcal{B}}$, and their traces on $\mathcal{A}$.
By definition, a relative cochain $c\in C^{n}(\mathcal{B},\mathcal{A};T,S)$ takes the value
$0$ on every family of $n+1$ elements of $\mathcal{A}$. If it is a cocycle, for any
family $\alpha,\alpha_1,...,\alpha_{n+1}$ of $n+2$ elements in $\mathcal{B}$, we have, for any $\gamma \in \A$, in $T_\gamma$:
\begin{equation}
c(\alpha_1,...,\alpha_{n+1})=\sum_{k=1}^{n+1}(-1)^{k+1}c(\alpha,\alpha_1,...,\widehat{\alpha_k},...,\alpha_{n+1}),
\end{equation}
which tells that $c$ is equal to $\delta \phi$, where $\phi$ is the $(n-1)$-cochain defined by
\begin{equation}
\forall \beta_1,...,\beta_{n} \in \mathcal{B},\quad \phi(\beta_1,...,\beta_{n})=c(\alpha,\beta_1,...,\beta_{n}).
\end{equation}
If $U_{\gamma}$ has empty intersection with $\mathcal{A}$, taking $\alpha=\gamma$, we have $\phi\in C^{n-1}(\mathcal{B},\mathcal{A};T,S)$,
and $c=d\phi$. And if $\alpha$ belongs to $\mathcal{A}\cap U_{\gamma}$, the cochain
$\phi$ belongs to $C^{n-1}(\mathcal{B},\mathcal{A};T,S)$, and $c=d\phi$. This establishes the lemma.
\end{proof}

\section{Factorization of free functors and intersection properties}\label{sec:free-sheaves}

{\color{black}
In this section, we describe how the previous algebraic results arise in applications; in particular, several functors appear naturally in these settings. We now give a quick overview of them. In applications to statistical physics or Bayesian learning, the presheaves that arise are often free modules generated by subsets of a fixed set \(I\). A set \(I\) is given (not necessarily finite); it corresponds to the collection of variables in the problem to be modeled. There are restrictions on the subsets of variables that can be jointly measured; these restrictions are captured by a sub-poset \(\mathcal{A}\) of \((\mathcal{P}_f(I), \subseteq)\), the poset of finite subsets of \(I\), ordered by inclusion. The associated posetal category \(\cat{C}_{\geq}(\mathcal{A})\) has arrows \(\alpha \to \beta\) whenever \(\beta \subseteq \alpha\). The poset \(\mathcal{A}\) automatically satisfies the conditions required in Section \ref{sec:linear-predual}, as it has locally finite dimension. \(\mathcal{A}\) defines a hypergraph with vertex set \(I\). The states of collections of variables and the mappings between them are encoded in a functor $E$; each individual variable takes values in a finite set $E_i$, and the state set of a collection of variables $\alpha \in \A$ is $E_\alpha := \prod_{i \in \alpha} E_i$. To each arrow $\alpha \to \beta$ corresponds the projection $\pi^{\beta\alpha}: E_\alpha \to E_\beta$, which corresponds to omitting all variables indexed by elements in \(\alpha \setminus \beta\). An observable (or random variable) over variables indexed by $\alpha$ is a map from $E_\alpha$ to $\mathbb{R}$. The space of observables defines a contravariant functor $V$, which associates to any $\alpha$ the vector space of functions $V_\alpha := \mathbb{R}^{E_\alpha}$, and to any projection $\pi^{\beta\alpha}$ the map $V_{\alpha\beta}$ induced by precomposition with $\pi^{\beta\alpha}$. The dual of $V$ is denoted $F$, and each $F_\alpha$ corresponds to the set of measures on $E_\alpha$, viewed as linear functionals on observables via integration. Probability measures live in \(F\), and compatible
marginals are particular sections of $F$. The reduced covariant functor denoted $\overline{F}$ accounts for the constraint that probability measures sum to one, and that the associated tangent spaces consist of measures summing to zero, or equivalently, linear functionals that factor through the quotient by constant observables. The functors $F$ and $\overline{F}$ are naturally extended to sheaves over $X_\A$.}

\subsection{Free functors and intersection properties}

{\color{black}We consider a covariant functor of sets on \(\cat{C}_\supseteq(\mathcal{P}_f(I))\), which associates to any \(\alpha \subseteq I\) the Cartesian product \(\prod_{i \in \alpha} E_i\), where each $E_i$ is a finite set, and to any arrow $\alpha \to \beta$ the canonical projection \(\pi^{\beta\alpha}\) from \(\prod_{i \in \alpha} E_i\) to \(\prod_{i \in \beta} E_i\). For \(\mathcal{A} \subseteq \mathcal{P}_f(I)\), we consider \(E\) to be the restriction of that functor to \(\cat{C}_{\supseteq}(\mathcal{A})\).} {\color{black} We make a slight abuse of notation when $\{i\}\notin \A$, as $E_i$ is not the stalk of $E$ at $\{i\}$}. By naturality, all the maps \(\pi^{\beta\alpha}: E_\alpha \to E_\beta\) are surjective. If the empty set \(\emptyset\) belongs to \(\mathcal{A}\), the set \(E_{\emptyset}\) is a singleton \(* = \{\emptyset\}\). In this case, for every element \(\alpha \in \mathcal{A}\), there exists a unique map \(\pi^{\emptyset\alpha}: E_\alpha \to E_\emptyset\).

 Note that $E$ extends to a sheaf of sets for the lower $A$-topology on $\mathcal{A}$, and for every arrow $\alpha\rightarrow\beta$, the map $\pi^{\beta\alpha}$ is the restriction of sections from the open set $U_\alpha$ to the open set $U_\beta$. 

 For every $\alpha\in \mathcal{A}$, we define $V_\alpha$ as the space of all functions from $E_\alpha$ to $\mathbb{K}$. We say that $V$ is the \emph{free} \emph{contravariant functor} generated by $E$.  If $\emptyset\in \mathcal{A}$, the space $V_\emptyset$ is canonically isomorphic to $\mathbb{K}$.  If $\alpha\rightarrow\beta$, i.e. $\beta\subseteq \alpha$, we get a natural application $j_{\alpha\beta}:V_\beta\rightarrow V_\alpha$, by right composition with $\pi^{\beta\alpha}$, 
which is linear and injective.
Therefore $V$ is a particular case of a contravariant functor with morphisms in linear injective maps over $\mathcal{A}$. As before, $V_{\alpha\beta}$ designates the image of $j_{\alpha\beta}$ in $V_\alpha$. 
We can identify $V_{\alpha\beta}$ with the space of numerical functions of $x_\alpha$ that depend only on the variables $x_\beta = \pi^{\beta\alpha}(x_\alpha)$, these functions are named the \emph{cylindrical functions}\footnote{
Subsets of the form $U_\beta \times E_{\gamma}$
for $U_\beta \subseteq E_\beta$ are often called \emph{cylinders} in $E_\beta \times E_{\gamma}$.
}
\emph{with respect to $\pi^{\beta\alpha}$}.

\begin{definition}[Reduced contravariant functor]
The sub-presheaf of constants $\mathbb{K}_\mathcal{A}$ maps each $\alpha \in \mathcal A$ to  the one dimensional
vector subspace $\mathbb{K}_\alpha$ of constant functions, embedded in $V_\alpha$. The \emph{reduced contravariant functor} (or reduced free functor) $\overline{V}_\alpha;\alpha\in \mathcal{A}$ is made of the quotient vector spaces $V_\alpha/\mathbb{K}_\alpha$.
\end{definition}

If $\emptyset\in \mathcal{A}$, for every $\alpha\in \mathcal{A}$, we have $\mathbb{K}_\alpha=V_{\alpha\emptyset}$.

\begin{definition}[Intersection property]
The hypergraph $\A$ satisfies the \emph{strong (resp. weak) intersection property}, if, for every pair $(\alpha,\alpha')$ in $\mathcal{A}$ (resp. every pair having non-empty intersection in $\mathcal{P}(I)$), the intersection $\alpha\cap \alpha'$ belongs to $\mathcal{A}$.
\end{definition}

\begin{remark} 
If \(\mathcal{A}\) satisfies the strong intersection property, all the pairwise meets \(\alpha \wedge \alpha'\) exist. If \(\mathcal{A}\) satisfies the weak intersection property, the pairwise meets \(\alpha \wedge \alpha'\) exist conditionally, i.e. \(\alpha \wedge \alpha'\) exists as soon as \(\alpha\) and \(\alpha'\) have a common minorant.
\end{remark}

If $\mathcal{A}$ has non-intersecting elements, the strong intersection property implies that the empty set $\emptyset$ belongs to $\mathcal{A}$, then $\cat{C}_{\supseteq}(\mathcal{A})$ possesses a unique final element, that is $\emptyset$.
If $\mathcal{A}$ satisfies the weak intersection property, it possesses conditional coproducts (here intersections) in the categorical sense of Section \ref{sec:posets_alexandroff}.

\begin{proposition}\label{prop3:strong_intersection_property}
If $\mathcal{A}$ has the strong intersection property, the condition \ref{condition_G} is satisfied by the free contravariant functor $V$.
\end{proposition}
\begin{proof}
Consider $\alpha\to \beta$ in $\mathcal A$ (i.e. $\beta \subseteq \alpha$), and a vector $v$ in $V_{\alpha\beta}$ that satisfies
\begin{equation}
v=\sum_{\substack{\gamma: \alpha\rightarrow \gamma, \gamma\neq\alpha,  \\
\gamma  \not\to \beta }} v_\gamma,
\end{equation}
for some $v_\gamma\in V_{\alpha\gamma}$.

The above decomposition tells that for every $x_\beta\in E_\beta$, and for any collection of elements $\{y_j\}_{j\in \alpha\backslash\beta}$, where  $y_j\in E_j$, we have
\begin{equation}
v(x_\beta,y_{\alpha\setminus\beta})= \sum_{\substack{\gamma: \alpha\rightarrow \gamma, \gamma\neq\alpha,  \\
\gamma \not\to \beta }} v_\gamma(x_\gamma);
\end{equation}
where on the right, the components of $x_\gamma$ are $x_i$ with $i\in \beta\cap\gamma$ and $y_j$ with $j\in (\alpha\setminus\beta)\cap\gamma$.\\
\indent For each index $k\in \alpha\setminus\beta$, we choose a fixed $y^{0}_k$, and replace everywhere in the formula the variable $x_k$
by this value. The formula continues to hold true. In the expression
$v_\gamma(x_\gamma)$, the variables $x_i$ that do not belong to $\beta\cap\gamma$, are
constants $y^{0}_k; k\in \alpha\backslash\beta$.  Moreover the intersection
of $\beta$ and $\gamma$ is a strict subset of $\beta$, because $\beta$ is assumed to be not
included in $\gamma$.\\
This gives
\begin{equation}
v(x_\beta, y^{0}_{\alpha\setminus\beta})=\sum_{\substack{\gamma: \alpha\rightarrow \gamma, \gamma\neq\alpha, \\
\beta\to \beta\cap \gamma, \beta\neq\beta\cap\gamma }} v^{0}_{\beta\cap\gamma}(x_{\beta\cap \gamma}).
\end{equation}
And, for all possible $\omega\in \mathcal{A}$, $\omega\subset\beta$, $\beta\neq\omega$, if we bring together the $\gamma$ such that
$\alpha\rightarrow \gamma, \gamma\neq\alpha,
\beta\cap \gamma=\omega$, this gives
\begin{equation}
v(x_\beta, y^{0}_{\alpha\setminus\beta})=\sum_{\substack{\omega: \alpha\rightarrow \omega, \omega\neq\alpha,  \\
\beta\to \omega, \beta\neq\omega }} v_{\omega}(x_\omega).
\end{equation}
Which is the expected result.
\end{proof}

\begin{remark}
Without the strong intersection property the result is false. Take for instance, $I=\{ i,j\}$, $\mathcal{A}=\{ \{i\};\{j\};\alpha=\{i,j\}\}$, where $\A$ is the set that contains the subsets  $\{i\},\{j\}$ and $\alpha=\{i,j\}$. A non-zero constant function belongs to $V_{\alpha i}$, but cannot belong to the image of a strict subset of $\{ j\}$.
\end{remark}

\begin{proposition}\label{prop4:weak_IP_reduced_V}
If $\mathcal{A}$ has the weak intersection property, the condition \ref{condition_G} is satisfied by the reduced presheaf $\overline{V}$.
\end{proposition}
\begin{proof}
Repeat the proof of Proposition \ref{prop3:strong_intersection_property}, but distinguish the cases where $\beta\cap\gamma$ is empty or not. When it is empty
the respective function $v_\beta$ of $(x_\gamma,y_{\gamma'})$ belongs to the constants.
\end{proof}

 Now recall that, by construction, the poset $\mathcal{A}$ has finite down-sets as it is a subposet of \(\mathcal{P}_f(I)\). Then the following Theorem results directly from the Prop. \ref{prop3:strong_intersection_property}
(resp. \ref{prop4:weak_IP_reduced_V})
and the Theorem \ref{thm2:extrafine_and_conditionG}.

\begin{theorem}\label{thm3:IP_implies_extrafine}
If $\mathcal{A}$ has  the strong (resp. weak) intersection property, the free contravariant functor $V$ (resp. $\overline{V}$) has an interaction decomposition.
\end{theorem}

Theorem \ref{thm3:IP_implies_extrafine} generalizes a theorem of existence of a  \emph{decomposition into interaction subspaces} for \emph{factor spaces} that, under different forms, has been known for a long time in probability theory, but only for finite posets and finite dimensional vector spaces, cf. \cite{Kellerer1964,Matus1988,Speed,Lauritzen}. {\color{black}This result is traditionally the main ingredient of the proof of the Hammersley-Clifford theorem \cite[ Chapter 3]{Lauritzen}, which states that a strictly positive distribution satisfying the conditional‑independence relations encoded by an undirected graph \(G\) factorizes into terms depending only on the joint variables in the cliques of \(G\). Moreover, this theorem ensures the equivalence of different ways to read the conditional-independence relations from the graph, for instance, by graph separation or by removing parts of the graph and checking for disconnection, yielding equivalent statements about (conditional) independences among the variables. It also plays an important role in characterizing the space, and its dimension, of probability distributions having the same compatible marginals \cite{Kellerer1964,Matus1988}. Extending the computation of the dimension of the solution to the  linearized marginal problem, and placing it in a canonical algebraic setting, is one of the motivations of the present paper.}

As in the preceding section, denote by $V'_\alpha$ the sum of the $V_{\alpha\beta}$ over $\beta\subsetneq \alpha$, (resp.
$\overline{V}'_\alpha$ the sum of the $\overline{V}_{\alpha\beta}$ over $\beta\subsetneq \alpha$)
and take a
supplementary subspace $S_\alpha$ of $V'_\alpha$ in $V_\alpha$ (resp. $\overline{S}_\alpha$ of $\overline{V}'_\alpha$ in $\overline{V}_\alpha$). The \emph{interaction decomposition} gives
\begin{equation}
\forall\alpha\in \mathcal{A},\quad V_\alpha=\bigoplus_{\beta\subseteq \alpha}S_\beta,
\end{equation}
resp.
\begin{equation}
\forall\alpha\in \mathcal{A},\quad \overline{V}_\alpha=\bigoplus_{\beta\subseteq \alpha}\overline{S}_\beta.
\end{equation}

\subsection{Duality: Free covariant functor}\label{subsection:duality}

Until the end of Section~5, we assume that for all $i\in I$, $E_i$ is a finite set. Note $F_\alpha=\mathbb{K}^{(E_\alpha)}$ the space of functions with finite supports, which can be
seen as the vector spaces freely generated by the set $E_\alpha$ over the field $\mathbb{K}$. Its dual space is $V_\alpha=\mathbb{K}^{E_\alpha}$
and the transpose of the natural map $\pi^{\beta\alpha}:F_\alpha\rightarrow F_\beta$ is $j_{\alpha\beta}$. The vector spaces $F_\alpha$ and the maps $\pi^{\beta\alpha}$
define a covariant functor over $\cat{C}_{\supseteq}(\mathcal{A})$ (resp. a sheaf on $X_{\mathcal A}$) named the \emph{free covariant functor} (resp. the \emph{free sheaf}) generated by $E$.

As each $E_\alpha$ is finite, $V_\alpha$ is a finite-dimensional vector space; in particular, $F_\alpha \cong V_\alpha^*$. Furthermore as $\A\subseteq \mathcal{P}_f(I)$, then $\A$ is closure finite.
We can apply Corollary \ref{dual:upper-extra-fine} and Proposition \ref{prop3:strong_intersection_property} from the preceding section to get the following result.

\begin{proposition}\label{prop5}
When $\mathcal{A}$ satisfies the strong intersection property, the free covariant functor $F$ is acyclic and $H^{0}(X_{\mathcal A},F)\cong \bigoplus_{\alpha\in \mathcal{A}}S_\alpha^*$.
\end{proposition}

\noindent Respectively, denote by $\overline{F}_\alpha$ the subspace of $F_\alpha=\mathbb{K}^{(E_\alpha)}$ defined by annihilating the sum of
the coordinates in the canonical basis. Its dual space is $\overline{V}_\alpha=\mathbb{K}^{E_\alpha}/\mathbb{K}_\alpha$. The transpose of the natural map $\pi^{\beta\alpha}:\overline{F}_\alpha\rightarrow \overline{F}_\beta$ is again $j_{\alpha\beta}$. This forms a sheaf over $X_\mathcal{A}$, named
the \emph{restricted free sheaf} generated by $E$. As before, we obtain the following.

\begin{proposition}\label{prop6}
When $\mathcal{A}$ satisfies the weak intersection property, the restricted free sheaf $\overline{F}$ over $X_\mathcal{A}$ is acyclic and $H^{0}(X_{\mathcal A},\overline{F})\cong \bigoplus_{\alpha\in \mathcal{A}}\overline{S}_\alpha^*$.
\end{proposition}

 In the case of finite sets $\{E_i\}_{i\in I}$ and $I$ finite, this result was established in H.G. Kellerer \cite{Kellerer1964}. See also \cite{Matus1988} and Appendix \ref{appendix5:finite_proba_functors} below.

\begin{theorem}\label{thm4}
 If the hypergraph $\A$ satisfies the weak intersection property, for any covariant functor of finite sets $E$ on the category $\cat{C}_{\supseteq}(\mathcal{A})$, the \v{C}ech cohomology $H^{*}(X_\mathcal{A};F)$ of the induced free sheaf $F$ is naturally
isomorphic to the sum of $H^\ast(X_\mathcal{A};\overline{F})$ which is concentrated in degree zero, and of the full
\v{C}ech cohomology (with trivial coefficients $\mathbb{K}$) of the topological space $X_{\mathcal A}$ (i.e. the poset $\mathcal A$ equipped with the lower Alexandroff topology).
\end{theorem}
\begin{proof}
The sheaf $F$ over $X_\mathcal{A}$ is decomposed into the sum of the sheaf $\overline{F}$ and the constant sheaf $\mathbb{K}_A$; this
induces a decomposition in direct sum of the cochain complexes. One of them gives 
$H^\ast(X_\mathcal{A};\overline{F})$, which is concentrated in degree zero as
just said by the preceding proposition, whereas
the other one gives the standard \v{C}ech cohomology of $\mathcal{A}$.
\end{proof}

 When $\mathcal{A}$ is the poset of a finite simplicial complex, it satisfies the weak intersection property, and the standard \v{C}ech cohomology (with constant coefficients) on $X_{\mathcal A}$ is isomorphic to the
singular or simplicial cohomology with coefficients in $\mathbb{K}$. See the introduction to Section \ref{sec:nerves-comparison}.

\subsection{Relative cohomology and marginal theorem}  In addition to $\mathcal{A}\subseteq \mathcal{P}_f(I)$, consider another poset $\mathcal{B}$ satisfying the
same kind of hypotheses, with respect to a set $J$, i.e. $\mathcal{B}\subseteq \mathcal{P}_f(J)$.

\begin{definition}
A \emph{strict morphism}
from $(\mathcal{A},I)$ to $(\mathcal{B},J)$ is a pair $(f,f_I)$ of an order preserving map $f:\mathcal{A}\rightarrow \mathcal{B}$, and a map $f_I:I\rightarrow J$,
such that $\forall i\in I$ and all  $\alpha \in \mathcal{A}$ such that $i\in \alpha$, one has  $f_I(i)\in f(\alpha)\subseteq J$. For simplicity, we will denote the couple simply as $f$.
\end{definition}

 As before, let $E$ be the covariant functor of not necessarily finite sets over $\cat{C}_{\supseteq}(\mathcal{A})$ given by products of the sets $\{E_i\}_{i\in I}$ i.e. such that $\alpha \mapsto E_\alpha \cong \prod_{i\in\alpha} E_i$; we call the $E_i$ \emph{basic sets}. Consider a strict morphism $f:\mathcal{A}\rightarrow \mathcal{B}$. For every $j\in J$, let us define $E'_j$ as the product of the $E_i$ for $i\in I$ such that $f_I(i)=j$.

\begin{proposition}\label{prop7}
The direct image $f_*E$ over $X_{\mathcal{B}}$ (see Appendix \ref{appendix1:topology}) is given by the products of the basic sets $\{E'_j\}_{j\in J}$.
\end{proposition}
\begin{proof}
For $\beta\in \mathcal{B}$, the set $f_*E(U_\beta)$ (in the lower $A$-topology) is the subset of
the product of the $E_\alpha$ over $\alpha\in f^{-1}(\beta)$ formed by the families $s_\alpha; f(\alpha)\subseteq \beta$,
that are compatible on the intersections $U_\alpha\cap U_{\alpha'}$. Each set $E_\alpha$ is the product of the sets $E_i;i\in \alpha$,
and the compatibility condition tells that for any pair $\alpha,\alpha'$ with $f(\alpha)\subseteq \beta$ and $f(\alpha')\subseteq \beta$,
the restriction of $s_\alpha$ and $s_{\alpha'}$ to their common terminal points coincide. This implies that $E(f^{-1}(U_\beta))$
is the product of the $E_i;i\in I$ such that $f(i)\in \beta$, then it is the product of the $E'_j$ for $j\in \beta$.
\end{proof}

 In particular, $E'_j$ coincides with the set $(f_*E)_j$ which corresponds to the direct image of sheaves.

\begin{definition}\label{def:simplicial-morphism}
A \emph{simplicial morphism} from $\mathcal{A}$ to $\mathcal{B}$ is a strict morphism $f:\mathcal{A}\rightarrow \mathcal{B}$,
such that $\forall \alpha\in \mathcal{A}$, the restriction of $f_I$ to the set $\alpha\in \mathcal{P}_f(I)$ is surjective onto the set
$f(\alpha)\in \mathcal{P}_f(J)$, i.e. $f_I(\alpha)=f(\alpha)$.
\end{definition}

\begin{remark}
    We propose calling the map defined in Definition \ref{def:simplicial-morphism} a simplicial map. When \(\mathcal{A}\) and \(\mathcal{B}\) are simplicial complexes, this map corresponds to a simplicial map between two simplicial complexes.
\end{remark}

\begin{proposition}
Let $f:\mathcal{A}\rightarrow \mathcal{B}$ be injective and simplicial, and let $F'$ be a sheaf on $X_\mathcal{B}$,
given by products of the basic sets $E'_j;j\in J$. The inverse image $f^{*}F'$ over $X_\mathcal{A}$ (see Appendix \ref{appendix1:topology}) is given by the
products of the basic sets $(E_i=E'_{f(i)})_{i\in I}$, i.e. $f^*F^{'}(\alpha)= \prod_{i\in \alpha}E_i$.
\end{proposition}
\begin{proof}
For $\alpha\in \mathcal{A}$, by definition of $f^{-1}F'$ (which coincides with $f^{*}F'$ in the case of posets),
$(f^{*}E')_\alpha=E'_{f(\alpha)}$ is the product of the sets $E'_j;j\in f(\alpha)$, and this product coincides with the product of the
sets $E'_{f(i)}$ for $i\in \alpha$ because $f$ is simplicial and injective.
\end{proof}

 In the following result, we consider the restricted subsheaves $\overline{F}$ and $\overline{F}'$, and we assume that both
$\A$ and $\mathcal{B}$ verify the weak intersection property.

\begin{theorem}\label{thm5}
Let $\A$ and $\mathcal{B}$ be two posets with conditional meets. Let $J:\mathcal{A}\rightarrow \mathcal{B}$ be an inclusion of posets, strict and simplicial. Assume that $\overline{F}'$ is a restricted free sheaf
over $X_\mathcal{B}$ given by products of the basic finite sets $E_j^{'}; j\in J$. Then the inverse image $\overline{F}=J^{*}\overline{F}'$  over $\cat{C}_{\supseteq}(\mathcal{A})$ is restricted, and we have a natural surjection from $\overline{F}'$ to $J_*\overline{F}$, and the induced
natural map in cohomology $J^{*}:H^{0}(X_\mathcal{B};\overline{F}')\rightarrow H^{0}(X_\mathcal{A};\overline{F})$ is surjective.
\end{theorem}
\begin{proof}
Along $\mathcal{A}$, the stalk of $\overline{F}'$ and $J_*\overline{F}$ coincide. From Theorem \ref{thm2:extrafine_and_conditionG} and the long exact sequence in \v{C}ech cohomology (Appendix \ref{appendix3:cech}), we get the following exact sequence:
{\color{black}
\begin{equation}
0\rightarrow  H^{0}(X_\mathcal{B},X_\mathcal{A};\overline{F}',\overline{F})\rightarrow H^{0}(X_\mathcal{B};\overline{F}')\rightarrow H^{0}(X_\mathcal{A};\overline{F})\rightarrow H^{1}(X_\mathcal{B},X_\mathcal{A};\overline{F}',\overline{F})\rightarrow 0.
\end{equation}
}
Then the theorem is equivalent to the vanishing of $H^{1}(X_\mathcal{B},X_\mathcal{A};\overline{F}',\overline{F})=0$. To prove the latter, we proceed
as in the proof of Corollary \ref{dual:upper-extra-fine}. Let $\overline{\mathcal{A}}$ and $\overline{\mathcal{B}}$ denote the posets respectively defined by taking all the finite parts of $I$ that are contained in {\color{black} some set of} $\mathcal{A}$ and ${\mathcal{B}}$. They are open sets in the lower Alexandroff topology {\color{black} for $\mathcal{P}_f(I)$}. As $\overline{\mathcal{A}}$ and $\overline{\mathcal{B}}$ are closed by intersections of cells, they define the same spaces $X$ as those defined by $\mathcal{A}$ and $\mathcal{B}$. Every open covering of $X_{\mathcal{A}}$ (resp. $X_{\mathcal{B}}$)
has a refinement which is an open covering of 
$\overline{\mathcal{A}}$ (resp. $\overline{\mathcal{B}}$). We decompose $\overline{F}'$ over $X_{\overline{\mathcal{B}}}$, as it is canonically upper extra-fine, and then $\overline{F}$ accordingly over $X_{\overline{\A}}$ in direct sums of sheaves $T_{\beta};\beta\in \overline{\mathcal{B}}$ and $S_{\alpha};\alpha\in \overline{\mathcal{A}}$ respectively, which satisfy the hypotheses of Lemma \ref{lem5bis}. Then we conclude
by applying Lemma \ref{lem5bis} and the natural isomorphism in cohomology between $\A$ (resp. $\mathcal{B}$) and $\overline{\A}$ (resp. $\overline{\mathcal{B}}$) for the restricted sheaves.
\end{proof}
In the context of finite probabilities, {quotienting $F$ to obtain $\overline F$} corresponds to the tangent equation of the probability restriction of sum $1$, and {the surjection of} Theorem \ref{thm5} is equivalent to a result of H.G. Kellerer \cite{Kellerer1964}.

In the  {Appendix \ref{appendix5:finite_proba_functors}}, based on the preceding sections, we prove the following index formula, which generalizes the result of Kellerer \cite{Kellerer1964} and Mat{\'u}{\v{s}} \cite{Matus1988}.

\begin{theorem}\label{thm6}
If the poset $\mathcal{A}$ is finite and satisfies the weak intersection property, and if the $\{E_i\}_{i\in I}$ are finite sets, then
\begin{equation}
\chi(\mathcal{A};F)=\sum_{k=0}^\infty (-1)^{k}\dim_{\mathbb{K}}H^{k}(\mathcal{U}_{\mathcal{A}};{F})=\sum_{\alpha, \beta\in \mathcal{A}}\mu_{\alpha\beta}N_\beta;
\end{equation}
where $\mu_{\alpha, \beta}$ is the M\"obius function of $\mathcal{A}$, and, for each $\alpha\in \mathcal{A}$, $N_\alpha$ denotes the cardinality of $E_\alpha$.
\end{theorem}

We will also prove that
\begin{equation}
\chi(\mathcal{A};{F})=\dim_{\mathbb{K}}H^{0}(\mathcal{U}_{\mathcal{A}};\overline{{ F}})+\chi(\mathcal{A});
\end{equation}
where $\chi(\mathcal{A})$ denotes the Euler characteristic of $\mathcal{A}$, in every possible sense: as a metric subspace of the simplex
$\mathcal{P}(I)$, as the lower or upper Alexandroff topological space in \v{C}ech cohomology, or as an abstract poset; this is also the Euler characteristic
of the nerve of the category $\cat{C}_{\supseteq}(\mathcal{A})$.

\section{Nerves of categories and nerves of coverings}\label{sec:nerves-comparison}

{\color{black} Let $G$ be a contravariant functor of abelian groups over the category $\cat{C}_\leq(\mathcal A)$ induced by a poset $(\mathcal A,\leq)$. The functor $G$ is also an abelian object in the topos $\operatorname{PSh}(\mathcal A)= \operatorname{PSh}(\cat{C}_\leq(\mathcal A))$, cf. \cite{Artin1972-2}, \cite{Moerdijk2006}. And in the context of topos theory, it is customary to study another cohomology}, that is the graded derived functor $H^\ast(\mathcal A, -)$ of
\begin{equation}
\Gamma_{\mathcal A}(-) =  \Hom_{\operatorname{Ab}(\mathcal A)}(\mathbb Z, -)\cong \Hom_{\operatorname{PSh}(\mathcal A)} (\ast, -);
\end{equation}
cf. \cite{Weibel1994}. In the following lines, we give a more explicit and topological definition of this functor, according to \cite{Artin1972-2}, \cite{Moerdijk2006}.\\

The nerve of a small category $\mathcal C$ is the simplicial set whose $n$ simplices are sequences $c_0 \to \cdots \to c_n$ of composable arrows in $\mathcal C$, and whose face operators are
\begin{equation}
d_i(c_0 \overset{f_1}\to  \cdots \overset{f_n}\to c_n) = \begin{cases}
c_1 \to \cdots \to c_n & \text{if } i=0 \\
c_0 \to \cdots c_{i-1} \overset{f_{i+1}\circ f_i}\to c_{i+1} \to \cdots \to c_n & \text{if } 0 < i < n\\
c_0 \to \cdots \to c_{n-1} & \text{if } i=n
\end{cases}.
\end{equation}
For background and details, see Section \ref{appendix2:nerve} below. This permits to define a \emph{canonical cochain complex} $(C^n(\mathcal A, G),d)$ whose cohomology is precisely $H^\ast(\mathcal A,G)$, cf. \cite[Prop.~6.1]{Moerdijk2006}. This complex
comes from a  canonical projective resolution of the constant presheaf $\mathbb Z$ \cite[Ex.~V.2.3.6]{Artin1972-2}.\\

The $n$-cochains are
\begin{equation}\label{eq:complex-sheaf-cohomology}
C^n(\mathcal A,G) = \prod_{a_n \to \cdots \to a_0 \text{ in }\mathcal A} G(a_n) = \prod_{a_0 \to \cdots \to a_n \text{ in }\mathcal A^{op}} G(a_n)
\end{equation}
and the coboundary $\delta: C^{n-1}(\mathcal A,G)\to C^n(\mathcal A,G)$ is given by
\begin{equation}\label{eq:coboundary-sheaf-cohomology}
(\delta g)_{a_0\to \cdots \to a_n} = \sum_{i=0}^{n-1} (-1)^i g_{d_i(a_0 \to \cdots \to a_n)} + (-1)^n G(\varphi_n)g_{d_n(a_0\to \cdots \to a_n)},
\end{equation}
where $\varphi_n$ is the $\mathcal A$-morphism from $a_{n}$ to $a_{n-1}$ in the sequence $a_n\to \cdots \to a_0$.

\begin{remark}
This complex and its analog for a covariant homology were rediscovered by O. Peltre in the context of his doctoral work \cite{OP-phd}, which gives a homological interpretation of the generalized Belief Propagation algorithm \cite{Yedidia2005}, which is applied in statistical physics,  Bayesian Learning and decoding processes. One of the initial motivations behind the present article was to understand better the connections of the Belief Propagation algorithm with \v{C}ech cohomology and sheaf cohomology.
\end{remark}

\begin{example}\label{acyclicityfactorinteractiondecomposition}
Consider a sheaf $S_{\gamma}$ for $\gamma\in \mathcal{A}$ of abelian groups over $\mathcal{A}$, which is zero outside $U_{\gamma}$ and such
that for any arrow $\alpha\rightarrow\beta$ in $U_{\gamma}$ the morphism $j_{\alpha\beta}$ is an isomorphism of groups. This was the situation
for the factors of an interaction decomposition. Then the above cohomology $H^\ast(\mathcal A, -)$ is null in degree $\geq 1$, and equal to zero or to $S_{\gamma}(\gamma)$ in degree 0,
according to the fact that $\gamma$ has or not a successor in $\mathcal{A}$. The proof is the exact copy of the Lemma \ref{lemma-extra-fine}.

\end{example}

In this section, our aim is to compare the topos cohomology of a presheaf $G$ over $\mathcal C_{\leq}(\mathcal A)$ with the topological \v{C}ech cohomology of the induced sheaf $\tilde G$ on $X_{\mathcal A}$ that we have studied in the previous sections.
 We will prove (Corollary \ref{cor1:isomorphsm-poset})  that, when $\mathcal{A}$ possesses conditional products (i.e. the covering $\mathcal{U}_\mathcal{A}$ is closed
by finite non-empty intersections), these two cohomologies are naturally isomorphic.

The equivalence of category between sheaves over $X_{\mathcal{A}}$ and $\operatorname{PSh}(\mathcal A)$, reminded in \eqref{categoryequivalence}
does not imply in general the isomorphism of these two cohomology theories see the (counter) example in \cite{Grothendieck1957}, page 177.

For any topological space, the isomorphism between \v{C}ech and topos cohomology in degree $0$ and $1$ is a general result, due to \cite[Sec.~3.8]{Grothendieck1957}. The fact that in every degree, the topos cohomology of sheaves over a topological space coincides with the topos cohomology of any
equivalent topos is proved in \cite[Ch.~V]{Artin1972-2} see also \cite{stacks-project}, cohomology of sheaves. Grothendieck also established an injection of \v{C}ech cohomology group of degree $2$ into the sheaf cohomology group of degree $2$. This injectivity is sufficient to establish the vanishing of $H^2$ for the sheaves $S^\gamma$ or $T_{\gamma}$ without appealing to the existence of conditional products or coproducts in $\mathcal A$. Moreover, Grothendieck in \cite{stacks-project} gave criteria for an isomorphism in higher degree.

 The methods {of comparison based on spectral sequences \cite{Segal1968},} developed in {Tohoku \cite[Ch.~III]{Grothendieck1957} and the SGA4
\cite{Artin1972-2}}, {might} apply in our situation, but our aim in this last section is to describe an explicit natural isomorphism between the two cohomology theories, when it exists. We will achieve this goal by putting the two theories in a same more general context
of cosimplicial sheaves over simplicial sets. The explicit comparison of the cohomology theories is the content of the theorem \ref{them7} below.

When $G$ is the constant sheaf $\mathbb Z$,  $H^\ast(C(N(\mathcal A), \mathbb Z),d)$ corresponds to the simplicial cohomology of $|N(\mathcal A)|$, the geometric realization of the nerve $N(\mathcal A)$ (see  \cite[Ch.~III]{May1992}). $N(\A)$ is a simplicial set then $|N(\A)|$ is a  CW-complex \cite{May1992}; it is known this cohomology to be naturally isomorphic to the singular cohomology of $|N(\mathcal A)|$ (cf. \cite{Eilenberg1952}). In turn, for a paracompact space $X$, as are CW-complexes \cite{miyazaki1952paracompactness}, the \v{C}ech cohomology $\check{H}^\ast(X, \mathbb Z)$ is isomorphic to the singular cohomology of $X$. Hence, the isomorphism $H^\ast(C(N(\mathcal A), \mathbb Z),d) \cong \check{H}^\ast(X_{\mathcal A}, \mathbb Z) $ is implied by the homotopy equivalence between $|N(\mathcal A)|$ and $X_{\mathcal A}$, see May \cite{May1992}. We are looking here for an extension of this result in the context of sheaves, and without paracompactness.

For that purpose, we introduce a general framework of \emph{cosimplicial local systems} on simplicial sets. We will remind below the definition of simplicial sets and simplicial objects in a category. The nerve $K_\ast(\mathcal U)$ of a covering $\mathcal U$ introduced in Section \ref{sec:extra-fine} and the nerve $N(\mathcal{C})$ of a category $\mathcal{C}$ are examples of simplicial sets.  Cosimplicial local systems are functorial assignments of local data to the simplexes and morphisms of a simplicial set. It becomes apparent that both \v{C}ech cohomology and the cohomology introduced by \eqref{eq:complex-sheaf-cohomology}-\eqref{eq:coboundary-sheaf-cohomology} become particular cases of this general construction and can be compared in this framework.

\subsection{Simplicial sets and nerves of coverings}\label{appendix2:nerve}  We denote by $\Delta$ the category whose objects are the finite ordered sets $[n]=\{0 < 1 <...<n\}$, for each $n\in \mathbb N$, and whose morphisms are nondecreasing monotone functions. The map $\varphi: [m] \to [n]$ belongs to $\text{Hom}([m], [n])$ when $\forall k, l \in [m]$ such that $k \leq l$, we have $\varphi(k) \leq \varphi(l)$. Given any category $\mathcal C$, a simplicial object $S$ in $\mathcal C$ is a contravariant functor from $\Delta$ to $\mathcal C$ i.e. $S:\Delta^{op}\to \mathcal C$. The morphisms between simplicial objects are the natural transformations between functors. When $\mathcal C$ is the category of sets, $S$ is called a \emph{simplicial set} (see \cite{May1992}, \cite[Ch.~8]{Weibel1994}). Simplicial groups, modules, etc. are defined analogously.

Recall that all morphisms of $\Delta$ can be expressed as compositions of \emph{face and degeneracy maps}. For each $n\in \mathbb N$ and $i\in [n]$, the face map $d^{n}_i:[n]\to [n+1]$ is given by
\begin{equation}
d^{n}_i(j) = j \quad \text{if}\quad  j< i, \qquad d^{n}_i(j) = j+1 \quad \text{if} \quad j\geq i.
\end{equation}
Similarly, for each $i\in [n+1]$, the degeneracy map $s^{n+1}_i:[n+1]\to [n]$ is
\begin{equation}
s_i^{n+1}(j) = j \quad \text{if} \quad j\leq  i, \qquad s_i^{n+1}(j) = j-1 \quad \text{if}\quad j> i.
\end{equation}
Normally, we drop the super-index.

Given a morphism $\varphi:[m]\to [n]$  of $\Delta$, let $i_1,...,i_s$ be the elements of $[n]$ not in $\varphi([m])$, in reverse order, and let $j_1,...,j_t$, in order, be the elements of $[m]$ such that $\varphi(j) = \varphi(j+1)$. Then
\begin{equation}
\varphi = d_{i_1} \cdots d_{i_s}s_{j_1}\cdots s_{j_t}.
\end{equation}
Remark that $m-t+s = n$. This factorization is unique \cite[Sec.~I.2]{May1992}.

The face and degeneracy maps satisfy some relations:
\begin{align}\label{rcds}
\forall n\in \mathbb{N}, 0\leq j< k\leq n,\quad & s^{n+1}_j\circ d^{n}_k=d^{n}_{k-1}\circ s^{n+1}_j,\\
\forall n\in \mathbb{N}, 0\leq j\leq n+1,\quad & s^{n+1}_j\circ d^{n}_j=s^{n+1}_{j+1}\circ d^{n}_j=\id_{n+1},\\
\forall n\in \mathbb{N}, n+1\geq j> k+1\geq 1,\quad & s^{n+1}_j\circ d^{n}_k=d^{n}_k\circ s^{n+1}_{j-1}.
\end{align}

A (simplicial) morphism from a simplicial set $S$ to a simplicial set $S'$ is a natural transformation of functors.

\begin{example}[Simplex] A basic example of simplicial set is the \emph{$k$-simplex} $\Delta^k$ \cite[Def.~I.5.4]{May1992}, which is the presheaf represented by $[k]$. This means that $\Delta^k_n = \Delta^k([n])$ equals $\Hom([n],[k])$, and the map $\Delta^k\varphi : \Hom([m],[k]) \to \Hom([n],[k])$ induced by $\varphi:[m]\to[n]$ is given by precomposition with $\varphi$.
\end{example}

\begin{example}[Nerve of a covering] \label{example:nerve-covering}
Let $X$ be a topological space and $\mathcal{U}$ an open covering of $X$. The \emph{nerve of the covering} $\mathcal{U}$ is the set $K(\mathcal{U})$ of finite sequences of elements of $\mathcal{U}$ having a non-empty intersection. It has a natural structure of simplicial set: $K_n=K([n])$ is the set of sequences of length $n+1$, denoted $(U_0,...,U_n)$, and for any nondecreasing  function $\varphi_{m,n}$ from
$m$ to $n$, there is a map $\varphi^*_{m,n}:K_n\to K_m$ given by
\begin{equation}
\varphi^{*}_{m,n}(V_0,...,V_n)=(V_{\varphi(0)},...,V_{\varphi(m)}).
\end{equation}
In other terms, $K_n(\mathcal{U})$ is the set of maps $u:[n]\rightarrow \mathcal{U}$ such that the intersection of the images are non-empty,
and if $\varphi:[m]\rightarrow [n]$ is a morphism and $v\in K_n(\mathcal{U})$, then $\varphi^{*}_{m,n}(v)=v\circ\varphi$.

Hence the map $s^*_i=K(s_i^{n+1})$ is given by
\begin{equation}\label{commutdegeneracies}
s^*_i(U_0,...,U_n)=(U_0,...,U_{i-1}, U_i,U_i,U_{i+1},...,U_{n});
\end{equation}
is also called \emph{degeneracy map}, whereas $d^*_i = K(d^{n+1}_i)$ is given by
\begin{equation}\label{commutfaces}
d^*_i(U_0,...,U_{n+1})=(U_0,...,U_{i-1},\widehat{U_i},U_{i+1},...,U_{n+1}),
\end{equation}
and called \emph{face map.}

 For each $u=(U_0,...,U_n)$, we denote by $U_u$ the intersection $U_0\cap ...\cap U_n$. It is easily verified that, for every morphism $\varphi$, a collection of maps $\varphi_{m,n}:[m]\to [n]$ for $m,n\in \mathbb{N}$ and every $v\in K_n(\mathcal{U})$, one has $U_v \subseteq U_{\varphi^{*}(v)}$. In particular, $U_{\varphi^{*}(v)}$ is non-empty if $U_v$ is non-empty.
\end{example}

\begin{example}[Nerve of a category] \label{example:nerve-category}
 To any small category $\mathcal{C}$ is naturally associated a simplicial set $N(\mathcal{C})$, named its nerve: the elements of $N_n(\mathcal{C})$ are the covariant functors from the poset $[n]$ to $\mathcal{C}$, and the morphisms are obtained by right composition.

Concretely, an element of degree $n$ is a sequence
\begin{equation}
a=\alpha_0\rightarrow \alpha_1\rightarrow ...\rightarrow \alpha_n.
\end{equation}
The action $s_i^{*}$ of $s_i$ is the repetition of the object  $\alpha_i$ via the insertion of an identity $\id_{\alpha_i}$; the action $d_i^{*}$ of $d_i$ is the deletion of $\alpha_i$ via the composition of $\alpha_{i-1}\to \alpha_i$ and $\alpha_i\to \alpha_{i+1}$.

More generally if
\begin{equation}
b=\beta_0\rightarrow \beta_1\rightarrow ...\rightarrow \beta_m
\end{equation}
belongs to $N_m(\mathcal{C})$, and $\varphi:n\rightarrow m$ is non-decreasing, then
\begin{equation}
\varphi^{*}(b)=\beta_{\varphi(0)}\rightarrow \beta_{\varphi(1)}\rightarrow ...\rightarrow \beta_{\varphi(n)}.
\end{equation}
\end{example}

\begin{example}[Barycentric subdivision of the nerve of a covering] \label{example:barycentric-nerve-covering}
Consider the category {$\mathcal{A}(\mathcal{U})$} which has for objects the non-empty intersections of the elements
of $\mathcal{U}$, and for morphisms the inclusions, then the nerve {$N(\mathcal{A}(\mathcal{U}))$} is the barycentric subdivision of the simplicial
set $K(\mathcal{U})$. This was remarked by Segal \cite{Segal1968}, interpreting \cite{Eilenberg1950}.
\end{example}

 The cartesian product of two simplicial sets $K$ and $L$ is taken as it must be for functors to $\mathcal{S}$, that is
term by term: $(K\times L)([n])=K([n])\times L([n])$ at the level of objects, and similarly for the maps. To every object $x$ of $\mathcal{C}$ there corresponds
the object $x^{\wedge}=\Hom(.,x)$ of $\mathcal{C}^{\wedge}$, {the category of presheaves over $C$ taking values in $\mathcal{S}$
}, and to every morphism
$u:x\rightarrow y$ the morphism $u^{\wedge}:x^{\wedge}\rightarrow y^{\wedge}$,
given by $u^{\wedge}(v)=u\circ v$, when $v\in Hom(F,x)$. The Yoneda lemma \cite{Bell1990-BELTAL} tells
that this construction gives a bijection from $Hom(x,y)$ on $\Hom(x^{\wedge},y^{\wedge})$, then an embedding of $\mathcal{C}$ into
$\mathcal{C}^{\wedge}$. In the following definition we apply the construction to the
maps $d_0$ and $d_1$ from $[0]$  to $[1]$ for obtaining morphisms
also denoted by $d_0$ and $d_1$ from $\Delta^{0}$  to $\Delta^{1}$,
because for every $n$ in $\mathbb{N}$, we have $\Delta^{n}=[n]^{\wedge}$.
As $[0]$ is reduced to a singleton, $\Delta^{0}(m)$ is a singleton for
any $m\in \mathbb{N}$, therefore $K\times \Delta^{0}$ can be identified
with $K$. And $d_0^{n}$ goes from $K_n$ to $K_{n+1}$ for any $n\in \mathbb{N}$.

\begin{definition}
Let $f:K\rightarrow L$ and $g:K\rightarrow L$ be two simplicial maps, a \emph{simplicial homotopy}
 from $f$ to $g$ is a simplicial map $h:K\times \Delta^{1}\rightarrow L$, such that $f=h\circ (\id_K\times d_0)$ and $g=h\circ (\id_K\times d_1)$.
\end{definition}

In the previous definition,    $d_i:\Delta^0 \to \Delta^1$ is a natural transformation (between set-valued functors on $\Delta$); it  consists in postcomposition with  $d_i^0$, this is,
\begin{equation}
    d_i([n]):\Hom([n],[0])\to \Hom([n],[1]), \quad \varphi \mapsto d_i^0 \circ \varphi.
\end{equation}

Simplicial homotopy is an equivalence relation, compatible with composition of simplicial maps.

\begin{example}[Homotopy induced by a projection of coverings]\label{example:projections-covering}
 A covering $\mathcal{U}$ is called a  \emph{refinement} of another covering $\mathcal{U}'$ when every set of $\mathcal U$ is contained in some set of $\mathcal U'$. In that case, there exists a  map  $\lambda: \mathcal{U}\rightarrow \mathcal{U}'$, called \emph{projection}, such that for every $ U\in \mathcal{U}$ one has $U\subseteq \lambda(U)$. It is also said that
$\mathcal{U}$ is  \emph{finer} than $\mathcal{U}'$ \cite{Eilenberg1952}.

A projection map $\lambda: \mathcal{U} \rightarrow\mathcal{U}'$ induces  a simplicial morphism $\lambda_{*}$ from the
simplicial set $K(\mathcal{U})$ to the simplicial set $K(\mathcal{U}')$:
\begin{equation}
\lambda_*(u)=\lambda\circ u;
\end{equation}

\begin{proposition}\label{prop:homotopy_from_projection} If $\mathcal{U}$ is a refinement of $\mathcal{U}'$, two projections $\lambda, \mu$
from $\mathcal{U}$ to $\mathcal{U}'$ induce homotopic simplicial maps $\lambda_*, \mu_*$ from $K(\mathcal{U})$ to $K(\mathcal{U}')$.
\end{proposition}
\begin{proof}
Let $u=(U_0,...,U_n)$ be an element of $K_n(\mathcal{U})$, and $\varphi_i=(0,..,0,1,...,1)\in \Delta^{1}_n$ with the first $1$
at place $i$ between $0$ and $n+1$, we put
\begin{equation}
h(u,\varphi_i)=(\lambda(U_0),...,\lambda(U_{i-1}),\mu(U_i),...,\mu(U_n)).
\end{equation}
\end{proof}
\end{example}

\subsection{Cosimplicial local systems and their cohomology}\label{sec:simplicial-systems}
 
We present here a general definition of cohomology for cosimplicial local systems on simplicial sets. In what follows, we write "application" or "map," as usual, for morphisms in the category of sets.

\begin{definition} A \emph{cosimplicial local system} of sets $F$ over the simplicial set $K$ is a family $F_u$ indexed by the
elements $u$ of $K_m=K([m])$, for any $m\in \mathbb N$, and, for any morphism $\varphi:[m]\rightarrow [n]$ and any $v\in K_n$, a given application $F(\varphi,v):F_u\rightarrow F_v$,
where $u=\varphi^{*}_{m,n}(v)$, such that $F(\psi,w)\circ F(\varphi, v)=F(\psi\circ \varphi,w)$, for $\varphi:[m]\rightarrow [n]$, $\psi:[n]\rightarrow [p]$, $w\in K_p$,
$v=\psi^{*}_{n,p}(w)$, $u=\varphi^{*}_{m,n}(v)$.
\end{definition}

\begin{remark} A definition of simplicial local systems appeared in the work of Halperin \cite{Halperin1983}. In his case the arrows are in the reverse direction,
i.e. for $\varphi:[m]\rightarrow [n]$, $v\in K_n$, a map $\varphi^{*}_v:F_v\rightarrow F_{\varphi^{*}(v)}$.
\end{remark}

\begin{example}[\v{C}ech local system]\label{example:Cech-system} Take a presheaf $F^\vee$ over the topological space $X$ and consider an open covering $\mathcal{U}$ of $X$. Then for $u\in K(\mathcal{U})$, define $F_u^\vee=F^\vee(U_u)$, and for $\varphi:[m]\rightarrow [n]$, $v\in K_n$, take for
$F^\vee(\varphi,v)$  the restriction from $F^\vee(U_{\varphi^{*}_v})$ to $F^\vee(U_v)$. This defines a cosimplicial
system over $K(\mathcal{U})$.
\end{example}

\begin{example}[Upper and lower systems associated with a  functor]\label{example:systems-functor}
Let $F$ be a contravariant functor from $\mathcal{C}$ to the category of sets $\mathcal{S}$. We can define a cosimplicial local system $F^{*}$ over the nerve $N(\mathcal{C})$ (see Example \ref{example:nerve-category} for the definition and the notation), named the \emph{upper} system, by taking $F^{*}(a)=F(\alpha_n)$, and for a morphism $\psi:n\rightarrow p$, and an element $b=\beta_0\rightarrow ...\rightarrow \beta_p$ in $N_p(\mathcal{C})$, denoting by $\alpha_n$
the last element of $a=\psi^{*}(b)$, we have $\alpha_n=\beta_{\varphi(n)}$, and this comes with a canonical arrow $f$ going to $\beta_p$, then we take
$\psi^{*}_b=f^{*}$, going from $F_b=F(\beta_p)$ to $F_a=F(\alpha_n)$.

 In the dual manner, if $F$ is covariant, we can define the lower cosimplicial system $F_*$, by taking $F_*(a)=F(\alpha_0)$. Taking again the element $b$ and the morphism $\psi$, we use now the fact that the first element of $a=\psi^{*}(b)$ is $\alpha_0=\beta_{\varphi(0)}$, which comes with a canonical arrow $g$ in $\mathcal{C}$ from $b_0$ to it, and we can take $\varphi^{*}_b=g_{*}$ from $F_*(b)$ to $F_*(a)$.

 Replacing $\mathcal{C}$ by the opposite (or dual) category $\mathcal C^{op}$, we exchange contravariant functors with covariant ones, and lower systems with upper ones.
\end{example}

\begin{remark}
Introduce the category $\mathcal{S}(K)$, having for objects elements of $K$, that is pairs $(m, a)$ where $a\in K_m$, and for arrows between two elements $v\in K_n$ and $u\in K_m$
the elements $\varphi$ of $\Delta(m,n)$ such that $\varphi^{*}(v)=u$. Then a cosimplicial local system $F$ over $K$ is a contravariant functor (i.e. a presheaf)
from $\mathcal{S}(K)$ to the category of sets.
\end{remark}

\begin{definition}
Let $F$ be a cosimplicial local system over the simplicial set $K$; for each $n\in \mathbb{N}$,
a \emph{simplicial cochain} of $F$ of degree $n$ is an element $(c_u)_{u\in K_n}$ of
the product $C^{n}(K;F)= \prod_{u\in K_n} F_u$.
\end{definition}

 When $F$ is a local system of an abelian groups, $C^{n}(K;F)$ has a natural structure of abelian group. In what follows, we stay
in this abelian context.

The \emph{coboundary operator} $\delta: C^{n}(K;F)\rightarrow C^{n+1}(K;F)$ is defined by
\begin{equation}
(\delta c)(v)=\sum_{i=0}^{n+1}(-1)^{i}F(d_i,v)(c(d_i^{*}(v))),
\end{equation}
for any element $v$ of $C^{n+1}(K;F)$. In the expression, $d_i^{*}$ is $K(d_i)$; remark that the sum takes place in $F_v$.  When we want to be  more precise, we write $\delta=\delta^{n+1}_{n}$ at degree $n$. The operator $\delta$ is also named the differential of the cochain
complex $C^{n}(K;F),n\in \mathbb{N}$.

\begin{proposition} For all $n\in \mathbb N$, the equality $\delta^{n+2}_{n+1}\circ \delta^{n+1}_{n}=0$ holds. In short, $\delta\circ\delta = 0$.
\end{proposition}
\begin{proof}
The expression of $\delta\circ\delta (c)(w)$ is the sum of elementary terms of the form $(-1)^{k}F(d^{n+1}_j\circ d^{n}_i,w)(c(d^{*}_i\circ d^{*}_jw))$,
with $i\neq j$ and $k=i+j$ if $j< i$, $k=i+j+1$ if $j> i$.

It is easy to verify that the maps $d$ satisfy the relation $d_j d_i = d_i d_{j-1}$ if $i<j$ \cite[Ex.~8.1.1]{Weibel1994}. It follows that the terms in the sum  cancel two by two.\footnote{This is a short argument with big consequences. When did it appear for the first time? Who came up with it? Euler, Poincaré, Noether, Lefschetz, Alexander?}
\end{proof}

Recall that a sequence $\{C^{n}\}_{n\in \mathbb{N}}$ of abelian groups with an operator $\delta$ of degree $+1$ and square zero, is named a differential complex, or a cochain complex.

 \begin{definition}[Cohomology of a cosimplicial local system]\label{def:cohomology-local-system}
  The cohomology group in degree $n\in \mathbb N$ of the local system $F$ over the simplicial set $K$  is the quotient abelian group
$$H^{n}(K;F)=\Ker (\delta^{n+1}_{n})/\im (\delta^{n}_{n-1}).$$
By convention $\delta_{-1}^0=0$.

Equivalently, the cohomology $H^\ast(K;F)$ of $F$ over $K$, seen as graded vector space, is the cohomology of the complex of simplicial cochains $(C^\ast(K,F),\delta)$.
 \end{definition}

As usual, the elements of $\ker \delta^{n}$ are called $n$-cocycles, and those in the image of $\delta^{n-1}$ are the $n$-coboundaries.

 For example, a collection $(c_u)_{u\in K_0}$ is a $0$-cochain, and it is a $0$-cocycle if for any $v\in K_1$,
\begin{equation}
0 = F(d_0,v)(c_{d_0^* v}) - F(d_1,v)(c_{d_1^*v}).
\end{equation}

In the particular case of an open covering $\mathcal{U}$ of a topological space $X$, and a presheaf of abelian groups $F$ on $X$, the group
$H^{0}(K_\mathcal{U};F)$ (for the associated local system on the nerve of the covering $\mathcal{U}$) coincides with the set of global sections of $F$ on $X$,
i.e. the families $(c_U)\in \prod_{ U\in \mathcal{U}} F(U)$ of local sections of $F$ whose restrictions coincide on the non-empty intersections $U\cap V$, $ U,V\in \mathcal{U}$.

\begin{remark}\label{rmk:comparison-local-system-and-sheaf-resolution}
Let $G$ be a contravariant functor over $\cat{C}_\leq(\mathcal A)$. The simplicial cochain complex $(C^n(N(\mathcal A^{op}), G^*),\delta)$ associated with the upper local system of $G: \cat{C}_\leq(\mathcal A)^{op}\to \mathcal E$, in the sense of the preceding definitions, is precisely the cochain complex introduced by \eqref{eq:complex-sheaf-cohomology}-\eqref{eq:coboundary-sheaf-cohomology}. One could also compute this cohomology from the complex $(C^n(N(\mathcal A), G_*),\delta)$: the cochains are the same, and the differential only differs by a sign when $n$ is odd.

Similarly, if $G$ is covariant over $\cat{C}_\leq(\mathcal A)$, then one can see it as a presheaf on $\cat{C}_\leq(\mathcal A)^{op}= \cat{C}_\geq(\mathcal A) $; its sheaf cohomology can be computed as the cohomology of  $(C^n(N(\mathcal A), G^*),\delta)$ or $(C^n(N(\mathcal A^{op}), G_*),\delta)$

As mentioned before, these complexes were rediscovered by O. Peltre \cite{Peltre2019} for understanding geometrically the generalized Belief Propagation algorithm
of \cite{Yedidia2005}.
\end{remark}

Given two cochain complexes of abelian groups $C^{\ast}$ and $D^{\ast}$, whose differential operators have degree $+1$
(i.e. ascending complexes) a cochain map (or cochain morphism) is a collection $\{f^n:C^n\to D^n\}_{n\in \mathbb Z}$ of morphisms of groups
that commute with the differentials. In other words, it is a morphism of graded abelian groups of degree zero that commute with the differentials.

A chain map between two cochain complexes sends coboundaries to coboundaries and cocycles to cocycles, thus it induces a map at the level of cohomology.

\begin{example}
Let $\mathcal{U}$ be a refinement of $\mathcal{U}'$, and $\lambda: \mathcal{U} \rightarrow\mathcal{U}'$ an adapted projection. The simplicial morphism $\lambda_{*}$ of the
simplicial set $K(\mathcal{U})$ to the simplicial set $K(\mathcal{U}')$ in the last section induces, for each integer $n$, a map $\lambda^{*}$ from $C^{n}(\mathcal{U}';F)$ to $C^{n}(\mathcal{U};F)$ defined by $\lambda^{*}(c')=c'\circ \lambda_*$. More concretely,
\begin{equation}
(\lambda^{*}c')(U_0,...,U_n)=c'(\lambda(U_0),...,\lambda(U_n)).
\end{equation}
This map commutes with the \v{C}ech differential, then it induces a map in cohomology
\begin{equation}
\lambda^{*}:H^{n}(\mathcal{U}';F)\rightarrow H^{n}(\mathcal{U};F).
\end{equation}
\end{example}

 Given two cochain maps $f^{\ast},g^{\ast}$ from $C^\ast$ to $D^\ast$, a cochain homotopy from $f^{\ast}$ to $g^{\ast}$ is a morphism of graded groups $h$
from $C^{\ast}$ to $D^{\ast}$ of degree $-1$ such that
\begin{equation}
d_D\circ h+h\circ d_C=g-f.
\end{equation}
This defines an equivalence relation on cochain maps (of degree zero) which is compatible with the composition of maps.

\begin{proposition}[\protect{\cite[Thm.~4.4]{Eilenberg1952}}] If two cochain morphisms are homotopic, they give the same application in cohomology.
\end{proposition}
\begin{proof}
If $c$ is a cocycle of $C$, and if $h$ is a homotopy from $f$ to $g$, then $dh(c)=g(c)-f(c)$, thus
$g(c)$ and $f(c)$ have the same classes in cohomology.
\end{proof}

\begin{definition}[Lift of simplicial map to a local system]\label{def:lift-local-system}
let $F$ (resp. $G$) be a cosimplicial local system over the simplicial set $K$ (resp. $L$), and $f:K\rightarrow L$ a simplicial map,
a \emph{lift} $\widetilde{f}$ of $f$ from $G$ to $F$ is family of maps $\widetilde{f}_u:G_{f(u)}\to F_u$, for each $u\in K$, such that,
for any morphism $\varphi: m\rightarrow n$,
any $v\in K_n$, $u=\varphi^{*}v \in K_m$,
\begin{equation}
F(\varphi,v)\circ \widetilde{f}_u=\widetilde{f}_v\circ G(\varphi,f(v)).
\end{equation}
\end{definition}

An example is given by a morphism $(f,\varphi)$ between a presheaf $\mathcal{F}$ over $X$ and a presheaf $\mathcal{G}$ over $Y$, when
we consider two open coverings $\mathcal{U}$ and $\mathcal{V}$ of $X$ and $Y$  respectively, such that $\mathcal{U}$ is finer than
the open covering $f^{-1}(\mathcal{V})$. In this case, we choose a projection $\lambda$ from $\mathcal{U}$ to $\mathcal{V}$ i.e.
$\forall U\in \mathcal{U}$, $U\subseteq f^{-1}\lambda(U)$. As we have seen, this defines a simplicial map from $K(\mathcal{U})$ to $K(\mathcal{V})$,
that we write as $f_\lambda$; then, for $u=(U_0,...,U_n)$, the group $G(f_\lambda(u))$ can be identified with $G(\bigcap_{i=0}^{n}(\lambda (U_i))$,
and for every element $g$ in this group, we define
\begin{equation}
\widetilde{f_\lambda}_u(g)=\varphi (g)\in F(\bigcap_{i=0}^{n}(U_i));
\end{equation}
because $\varphi$ can be seen as a morphism of $f^{-1}\mathcal{G}$ to $\mathcal{F}$.\\

\begin{definition}
Two pairs $(f,\widetilde{f})$, $(g,\widetilde{g})$ of morphisms and lifts are \emph{simplicially homotopic} if there exists a simplicial homotopy
$h:K\times \Delta^{1}\rightarrow L$
from $f$ to $g$, and a family of maps $\widetilde{h}_{u,s}, u\in K, s\in \Delta^{1}$ from $G_{h(u,s)}$ to $F_{u}$, such that
$\widetilde{f}=\widetilde{j_0}\circ\widetilde{h}$
and $\widetilde{g}=\widetilde{j_1}\circ \widetilde{h}$, where $j_0=\id_K\times s_0$ and $j_1=\id_K\times s_1$ (here we used the Yoneda construction for defining $s_0$ and $s_1$.)
\end{definition}

As a consequence of Proposition \ref{prop:homotopy_from_projection}, two choices of projections in the construction of the map of local systems
associated with a morphism of presheaves gives two homotopic morphisms in the simplicial sense.

 Suppose given two local systems $F,G$ over $K,L$ respectively, and a lift $\widetilde{f}$ of $f:K\rightarrow L$. The following formula defines a natural
morphism $\widetilde{f}^{*}$ from $C^{\ast}(L;G)$ to $C^{\ast}(K;F)$:
\begin{equation}
\widetilde{f}^{*}(c_L)(u)=\widetilde{f}_u(c_L(f(u)).
\end{equation}

\begin{lemma}\label{lem3-1}
The natural morphism $\widetilde{f}^{*}$ commutes with the differentials.
\end{lemma}

\begin{proof}

    Let $c$ be a $n$ co-chain of $G$ on $L$, we have, for any element $u$ of $K$
\begin{align*}
\widetilde{f}^{*}(\delta_L c)(u)&=\widetilde{f}_u(\delta c)(f(u))\\
&=\sum_{i=0}^{n+1}(-1)^{i}\widetilde{f}_u(G(d_i,f(u))(c(d_i^{*}f(u)))\\
&=\sum_{i=0}^{n+1}(-1)^{i}F(d_i,u).\widetilde{f}_u(c(d_i^{*}f(u)))\\
&=\sum_{i=0}^{n+1}(-1)^{i}F(d_i,u).\widetilde{f}_u(c(f(d_i^{*}u)))\\
&=\sum_{i=0}^{n+1}(-1)^{i}F(d_i,u).\widetilde{f}^{*}(c(d_i^{*}u))\\
&=\delta_K(\widetilde{f}^{*}(c))(u).
\end{align*}
\end{proof}

\begin{definition}
Let $G$ be a cosimplicial local system over a simplicial set $L$, and $f$ a
simplicial map from $K$ to $L$, the family $F_u=G_{f(u)}$, for $u\in K$, with the maps $F(\varphi, v)=G(\varphi, f(v))$ is a cosimplicial local
system over $K$, named the pullback of $G$, and denoted by $f^{*}(G)$.
\end{definition}

\begin{example}
 Start with a cosimplicial local system $F$ over a simplicial set $K$; then, over the product $K\times \Delta^{1}$, we define a
local system $\pi^{*}F$ by taking, for $u\in K_n$ and $s\in \Delta^{1}$, $\pi^{*}F(u,s)=F(u)$. Then consider the two injections $j_0=Id_K\times s_0$ and
$j_1=Id_K\times s_1$ from $K=K\times \Delta^{0}$ to $K\times \Delta^{1}$; the two pullbacks $j_0^{*}\pi^{*}F$ and $j_1^{*}\pi^{*}F$ coincide with $F$.
Thus we have evident lifts $\widetilde{j_0}$ and $\widetilde{j_1}$ from $F$ to $\pi^{*}F$. They are homotopic in the simplicial sense, the map $h$
from $K\times \Delta^{1}$ to itself being the identity, and the lift being the natural identification.
\end{example}

From $C^{\ast}(K\times \Delta^{1};\pi^{*}F)$
to $C^{\ast}(K;F)$, the two chain maps $\widetilde{j_0}^{*}$ and $\widetilde{j_1}^{*}$ are given by the following formulas, let $c$ be a $n$-cochain of
$K\times \Delta^{1}$ with value in $\pi^{*}F$, and $u$ an element of $K_{n}$
\begin{align}
\widetilde{j_0}^{*}c(u)&=c(u,(0,...,0)),\\
\widetilde{j_1}^{*}c(u)&=c(u,(1,...,1)).
\end{align}

\begin{lemma}\label{lem3-2}
 $\widetilde{j_0}^{*}$ and $\widetilde{j_1}^{*}$ are homotopic as chain maps.
 \end{lemma}
 \begin{proof}
 Let $C$ be a $(n+1)$-cochain of $K\times \Delta^{1}$ with value in $\pi^{*}F$, and $u$ an element of $K_{n}$,
we define
\begin{equation}
H(C)(u)=\sum_{j=0}^{n+2}(-1)^{j}F(s_j,u)(C(s_j^{*}(u),1^{n+1}_j)),
\end{equation}
where, for $n\in \mathbb{N}$ and $j\in [n+2]$,  $1^{n+1}_j$ denotes the element $(0,...,0,1,...,1)$ of $\Delta^{1}_{n+1}$, where the first $1$ is at the place $j$.
That gives $1^{n+1}_{n+2}=(0,...,0)=s_0(0)$ and $1^{n+1}_{0}=(1,...,1)=s_1(0)$.\\
This defines an endomorphism $H$ of degree $-1$ of $C^{\ast}(K;F)$. Now we compute, for $c\in C^{n}(K\times \Delta^{1};\pi^{*}F)$:
\begin{equation}
H(\delta c)(u)=\sum_{j=0}^{n+1}\sum_{k=0}^{n+1}(-1)^{j+k}F(s_j,u)\circ F(d_k,s_j^{*}u)c(d_k^{*}s_j^{*}u,d_k^{*}1^{n+1}_j),
\end{equation}
and
\begin{equation}
\delta (H(c)) (u)=\sum_{j=0}^{n}\sum_{k=0}^{n}(-1)^{j+k}F(d_k,u)\circ F(s_j,d_k^{*}u)c(s_j^{*}d_k^{*}u,1^{n+1}_j).
\end{equation}
Let us add $H(\delta c)(u)$ and $\delta (H(c)) (u)$, in virtue of the relations \eqref{rcds}, most of the terms annihilate.
The only ones that survive correspond to the terms $(s_jd_j)^{*}$ and $(s_jd_{j-1})^{*}$. Note that $d_j^{*}1^{n+1}_j=1^{n}_{j}$
and that $d_{j-1}^{*}1^{n+1}_j=1^{n}_{j-1}$, then, due to the signs, they annihilate two by two,
except the extreme terms, for $j=0$ and $j=n+1$, giving
\begin{align}
\delta (H(c)) (u)+H(\delta c)(u)&=c(u,1^{n}_0)-c(u,1^{n}_{n+1})\\&=c(u,(1,...,1))-c(u,(0,...,0)).
\end{align}
Then $H$ is a chain homotopy operator from $\widetilde{j_0}^{*}$ to $\widetilde{j_1}^{*}$, as we desired.
 \end{proof}

\begin{proposition}
Let $f,g$ be two simplicial maps from the simplicial set $K$ to the simplicial set $L$,
let $F,G$ be cosimplicial systems over $K$ and $L$ respectively, and $\widetilde{f}, \widetilde{g}$ two lifts over $K$;
suppose that the pairs $(f,\widetilde{f})$, $(g,\widetilde{g})$ are simplicially homotopic, then the induced maps of cochains complexes
are homotopic.
\end{proposition}
\begin{proof}
The map $\widetilde{h}$ is a pullback of the simplicial map $h$ to the local systems $ \pi^{*}F $ on $K\times \Delta^{1}$
and $G$ on $L$. Thus we get a chain-map
\begin{equation}
\widetilde{h}^{*}:C^{\ast}(L;G)\rightarrow C^{\ast}(K\times \Delta^{1};\pi^{*}F).
\end{equation}
On the other side, we have two natural cochain maps, for $k=0,1$,
\begin{equation}
\widetilde{j_k}^{*}:C^{\ast}(K\times \Delta^{1};\pi^{*}F)\rightarrow C^{\ast}(K;F).
\end{equation}
Applying the Lemma \ref{lem3-2}, there exists a homotopy from $\widetilde{j_0}^{*}$ to $\widetilde{j_1}^{*}$:
\begin{equation}
H:C^{\ast}(K;F)\rightarrow C^{\ast-1}(K;F);
\end{equation}
therefore, applying the Lemma \ref{lem3-1}:
{\color{black}
\begin{multline}
\widetilde{g}^{*}-\widetilde{f}^{*}=\widetilde{h}^{*}\circ(\widetilde{j_1}^{*}-\widetilde{j_0}^{*})\\
=\widetilde{h}^{*}\circ(d\circ H+H\circ d)=d\circ(\widetilde{h}^{*}\circ H)+(\widetilde{h}^{*}\circ H)\circ d.
\end{multline}
}
\end{proof}

\begin{corollary}
The induced morphisms in cohomology are the same.
\end{corollary}

\begin{theorem}[cf. \protect{\cite[Ch.~IX]{Eilenberg1952}}]
If $\mathcal{U}$ is a refinement of $\mathcal{U}'$, two projections $\lambda, \mu$
from $\mathcal{U}$ to $\mathcal{U}'$ give the same application in cohomology.
\end{theorem}
\begin{proof}
The simplicial maps $\lambda_{*}$ and $\mu_{*}$ from $K(\mathcal{U})$ to $K(\mathcal{U}')$ are homotopic in the simplicial sense, then the maps $\lambda^{*}$ and $\mu^{*}$
are homotopic in the sense of maps of differential complexes, cf. the previous section.
\end{proof}

\subsection{Comparison theorems}

Let  $\mathcal U$ be a covering of a topological space $X$; we denote by $\A({\mathcal U})$ the poset, ordered by inclusion, formed by the finite non-empty intersections of the elements of ${\mathcal U}$. We denote by $N({\mathcal U})$ the nerve of the associated category $\cat{C}_\subseteq({\mathcal A}({\mathcal U}))$, cf. the preamble of section \ref{sec:nerves-comparison} and the example \ref{example:barycentric-nerve-covering}. Note that the set of open sets in $\mathcal U$ giving an element of $\mathcal{A}({\mathcal U})$ is not an element of the structure.

 The elements of $\mathcal{A}({\mathcal U})$ make an open covering of $X$ which is finer than $\mathcal U$. By choosing
{for each element $a$ of $\mathcal{A}(\mathcal{U})$ an element of $\mathcal{U}$ which contains $a$}, we obtain a map from $\mathcal{A}(\mathcal{U})$ to $\mathcal{U}$, that we denote $\pi$; it is a projection in the sense of Eilenberg-Steenrod, see Example \ref{example:projections-covering}. In what follows we will always assume that for $U\in \mathcal{U}$, $\pi(U)=U$. The map $\pi$ induces a simplicial map $\pi_*$ from the simplicial set $N(\mathcal{U})$ to the simplicial set $K(\mathcal{U})$ (which is the usual nerve of the covering $\mathcal U$ in the sense of Example \ref{example:nerve-covering}); it maps the sequence $V_0 \to \cdots \to V_n$ in $\cat{C}_\subseteq(\mathcal{A}(\mathcal U))$ to the sequence $(\pi(V_0),...,\pi(V_n))$
of elements of $\mathcal{U}$.

Given a presheaf $F$ of abelian groups over $X$, we have defined the cosimplicial local system of \v{C}ech $F^{\vee}$
 over $K(\mathcal U)$ (cf. Example \ref{example:Cech-system}). To define a local system over $N({\mathcal U})$, we restrict $F$ to a presheaf on $\cat{C}_\subseteq(\mathcal{A}(\mathcal U))$ and we take the \emph{lower} cosimplicial local system $F_*$ over $N({\mathcal U})$, as in  Example \ref{example:systems-functor}. Given an element $v=(V_0\to \cdots \to V_n)$ of $N_n( \mathcal{U})$, remark that $V_0 =\bigcap_{i=0}^n V_i \subseteq \bigcap_{i=0}^n \pi(V_i)$, hence there is a well-defined restriction map from $F^\vee(\pi_*v)$ to $F_*(v)$. This defines a lift $\tilde \pi$ of $\pi_*$ from $F^\vee$ to $F_*$ in the sense of Definition \ref{def:lift-local-system}, hence a  morphism
 \begin{equation}
  \pi^*: C^{\ast}(K(\mathcal{U}); F^\vee) \to C^{\ast}(N(\mathcal{U});F_*)
 \end{equation}

\begin{theorem}\label{them7}
The map $ \pi^{*}$ is a homotopy equivalence between $C^{\ast}(K(\mathcal{U}); F^\vee)$ and $C^{\ast}(N(\mathcal{U});F_*)$.
\end{theorem}

Before presenting the proof, let us see how this implies the isomorphism between topos cohomology and \v{C}ech cohomology in the case of abelian presheaves on a  poset,  provided it is a conditional meet semilattice i.e. that products exist conditionally. Let  $G$ be a contravariant functor of abelian groups on $\cat{C}_\leq(\mathcal A)$, and let $\tilde G$ be the induced sheaf on the lower A-space  $X_{\mathcal A}$. We have seen that the topos cohomology of $G\in \operatorname{PSh}(\mathcal A)$ is isomorphic to the cohomology of the cochain complex $(C^\ast(N(\mathcal A), G_*),\delta)$, whereas the \v{C}ech cohomology of $\tilde G$ is the cohomology of the complex $(C^\ast(K(\mathcal U_{\mathcal A}),G^\vee),\delta)$.

The space $X_{\mathcal{A}}$ has a finest canonical open covering $\mathcal{U}_{\mathcal{A}}$ made by the lower sets $U_{\alpha};\alpha\in \mathcal{A}$. An inclusion
$U_{\alpha}\subseteq U_{\beta}$ corresponds to an inequality $\alpha\leq \beta$, then the natural inclusion $\cat{C}_\leq(\mathcal{A})\hookrightarrow \mathcal{A}(\mathcal{U}_{\mathcal{A}})$
is a covariant functor, and induces an injective simplicial covariant functor $\iota:N(\mathcal{A})\hookrightarrow N(\mathcal{U}_{\mathcal{A}})$.

We have a diagram of simplicial sets
\begin{equation}
N(\mathcal A) \overset\iota\to N(\mathcal {U}_{\mathcal{A}}) \overset{\pi_*}\to K(\mathcal U_{\mathcal{A}})
\end{equation}
where the last arrow is induced by any projection $\pi:\mathcal{A}(\mathcal U_{\mathcal{A}}) \to \mathcal U_{\mathcal{A}}$ that is the identity on $\mathcal U_{\mathcal{A}}$. Let us denote by $j$ the simplicial map $ \pi_* \circ \iota$. Given $a=(\alpha_0\to \cdots \to \alpha_n)$ in $N(\mathcal A)$, we have
\begin{equation}
G(\alpha_0) = G_*(a) = G^\vee(j(a)) = \widehat{G}(\cap_{i=0}^n U^{\alpha_i}) = \widehat{G}(U^{\alpha_0}),
\end{equation}
then there is a lift of $j$ from $G^\vee$ to $G_*$ (cf. Definition \ref{def:lift-local-system}) given by identities. Thus we deduce a morphism of chain complexes
\begin{equation}
j^{*}:C^{\ast}(K(\mathcal{A});G^\vee)\rightarrow C^{\ast}(N(\mathcal{A});G_*)
\end{equation}
that induces a morphism in cohomology.

\begin{corollary}\label{cor1:isomorphsm-poset} If meets exist conditionally in $\mathcal A$, the chain map $j^{*}$ is a chain equivalence up to homotopy, thus induces an isomorphism in cohomology.
\end{corollary}
\begin{proof}
Under the hypothesis, one has  $\mathcal{A}(\mathcal {U}_{\mathcal{A}}) = \mathcal U_{\mathcal{A}}$, since every intersection $U_{\alpha_0}\cap \cdots \cap U_{\alpha_n}$ equals $U_{\alpha_0 \wedge \cdots \wedge \alpha_n}$. Hence $N(\mathcal A)\cong N(\mathcal U_{\mathcal{A}})$. The map $\pi_{*}$ is induced by $\pi=\id_{\mathcal{U}_{\mathcal{A}}}$. The claim then follows from Theorem \ref{them7}.
\end{proof}

The equivalence defines a natural transformation from the category of posets equipped with presheaves to the category of homotopy classes of isomorphisms of cochain complexes.

We close this section with the proof of Theorem \ref{them7}, inspired by a classical argument of Eilenberg and Steenrod:  starting with a simplicial complex $K$, they introduce the poset $\mathcal{A}(K)$, whose elements are the faces (simplices) of $K$; the nerve $N$ of this poset is naturally
isomorphic to the barycentric subdivision of $K$ (cf. \cite{Segal1968}). In \cite[pp.~177-178]{Eilenberg1950}, the authors proved that there exists a homotopy equivalence between $N$ and $K$. The following proof is an adaptation of their argument to this more general setting; it uses the concept of \emph{descendant}: we say that an $v=(V_1,...,V_n)$ of $K_n(\mathcal U)$ is the descendant of $u=(U_1,...,U_n)\in K_n(\mathcal U)$ if  for each $i\in \{0,...,n\}$, there exists a $j\in\{0,...,n\}$ such that $V_i$ is included in $U_j$.

\begin{proof}[Proof of Theorem \ref{them7}]
$1)$ First we construct by recurrence over $n$ a linear application $Sd^{n}$ from $C^{n}(N(\mathcal{U}))$ to $C^{n}(K(\mathcal{U}))$,
having the two following properties:
\begin{enumerate}[label=(\roman*)]
    \item\label{proof:prop_homotopy:locality} (locality) for any $c\in C^{n}(N(\mathcal{U}))$ and any collection $u=(U_0,...,U_n)$, the value $(Sd^{n}c)(u)$ in $F(U_u)$ depends only
of the values of $c$ on the \emph{descendants} of the open sets $u$, this is, the values $c(v)\in F(V_v)$ for the sequences $v=(V_0,..,V_n)$ that are descendants of $u$;
\item\label{proof:prop_homotopy:morphism}  (morphism of cochain complex) $d\circ Sd^{*}=Sd^{*}\circ d$.
\end{enumerate}

 For $n=0$, and $U_0\in \mathcal{U}=K_0(\mathcal{U})$, we take $Sd^{0} (c) (U_0)=c(U_0)$. This is allowed because $U_0$ is also an element of $N_0(\mathcal{U})$. The condition \ref{proof:prop_homotopy:locality} is evidently satisfied, and \ref{proof:prop_homotopy:morphism} is empty in this degree.

 For $n=1$, $c\in C^{1}(N(\mathcal{U}))$ and $u=(U_0,U_1)$, we set $Sd^{1}c(U_0,U_1)=c(U_0,U_u)-c(U_1,U_u)$, where $U_u=U_0\cap U_1$. This is local, and  for $c_0\in C^{0}(N(\mathcal{U}))$:
\begin{multline}
Sd^{1}(dc_0)(U_0,U_1)=(c_0(U_0)-c_0(U_u))-(c_0(U_1)-c_0(U_u))\\=dc_0(U_0,U_1)=d\circ Sd^{0}(c)(U_0,U_1).
\end{multline}

 Then take $n\geq 2$,
and suppose that a map $Sd^{q}$ is constructed for every $q\leq n-1$, satisfying \ref{proof:prop_homotopy:locality}  and \ref{proof:prop_homotopy:morphism}. Take a cochain $c$ in $C^{n}(N(\mathcal{U}))$,
and consider an element $u=(U_0,...,U_n)$ of $K_n(\mathcal{U})$; recall that we denote $U_u$ the intersection (necessarily non-empty) of the $\{U_i\}_{0\leq i\leq n}$.
We define an element $c_u\in C^{n-1}(N(\mathcal{U}))$ by taking on every decreasing sequence $v=(V_0,...,V_{n-1})$,
\begin{equation}
c_u(v)=c(V_0,...,V_{n-1}, U_u),
\end{equation}
if $V_{n-1}$
contains $U_u$

and taking $c_u(V_0,...,V_{n-1})=0$ in the opposite case. 
Then we define
\begin{equation}\label{eq:recurrence_Sd}
Sd^{n}(c)(u)=\sum_{i=0}^{n}(-1)^{n-i}(Sd^{n-1}c_u)(U_0,...,\widehat{U_i},...,U_n)|U_u.
\end{equation}

 To establish \ref{proof:prop_homotopy:locality}, we apply the recurrence hypothesis: each  quantity $(Sd^{n-1}c_u)(U_0,...,$ $\widehat{U_i},...,U_n)$ depends on $c_u$ evaluated on descendants $v=(V_0,...,V_{n-1})$ of $(U_0,...,\widehat{U_i},...,U_n)$; because of the definition of $c_u$, this translates to a dependence on $0$ or on $c(V_0,...,V_{n-1},U_u)$, the latter only when $v'=(V_0,...,V_{n-1},U_u)$ is a descendant of $(U_0,...,\widehat{U_i},...,U_n, U_u)$ and the $V_i$ are in decreasing order. In particular, such $v'$ is always a descendant of $u=(U_0,...,U_n)$ (but remark that the converse is not true). The sum in \eqref{eq:recurrence_Sd} means that $Sd^nc(u)$ depends on descendants $v''$ of $(U_0,...,\widehat{U_i},...,U_n, U_u)$ for any $i$; but since $U_u$ is the intersection of all the $U_i$, this simply translates into a dependence on the descendants of $u$.

 For \ref{proof:prop_homotopy:morphism}, we have to compute $Sd^{n}(dc)(u)$ for a cochain $c\in C^{n-1}(N(\mathcal{U});F)$. For a decreasing
sequence $V_0,..., V_{n-1}$, then, by writing $V_{n}=U_u$, we have
\begin{align*}
(dc)_u(V_0,..., V_{n-1})&=dc(V_0,...,V_{n-1},U_{u})\\
&=\sum_{j=0}^{n}(-1)^{j}c(V_0,...,\widehat{V_j},...,V_{n})|U_u\\
&=d(c_u)(V_0,..., V_{n-1})|U_u+(-1)^{n}c(V_0,...,V_{n-1})|U_u;
\end{align*}
where $c_u$ is also defined by $c_u(V_0,...,V_{n-2})=c(V_0,...,V_{n-2}, U_u)$ if $V_{n-2}$ contains $U_u$ and
$c_u(V_0,...,V_{n-2})=0$ in the opposite case.
Which gives for reference, when $c$ belongs to $C^{n-1}(N(\mathcal{U});F)$:
\begin{equation}\label{dcu}
d(c_u)=(dc)_u+(-1)^{n-1}c.
\end{equation}
It follows from the recurrence hypothesis that
\begin{align*}
Sd^{n}(dc)(u)&=\sum _{i=0}^{n}(-1)^{n-i}Sd^{n-1}((dc)_u)(U_0,...,\widehat{U_i},...,U_n)|U_u\\
&=\sum _{i=0}^{n}(-1)^{n-i}(Sd^{n-1}d(c_u))(U_0,...,\widehat{U_i},...,U_n)|U_u\\
&+\sum _{i=0}^{n}(-1)^{i}(Sd^{n-1}c)(U_0,...,\widehat{U_i},...,U_n)|U_u\\
&=\sum _{i=0}^{n}(-1)^{n-i}(d\circ Sd^{n-1}(c_u))(U_0,...,\widehat{U_i},...,U_n)|U_u\\
&+\sum _{i=0}^{n}(-1)^{i}(Sd^{n-1}c)(U_0,...,\widehat{U_i},...,U_n)|U_u\\
&=(-1)^{n}(d\circ d)Sd^{n}(c_u)(u)+d(Sd^{n-1}(c))(u)\\
&=d \circ  Sd^{n-1}(c)(u).
\end{align*}
Therefore $Sd^{n}$ verifies \ref{proof:prop_homotopy:morphism}.

\noindent $2)$ Let us prove that the composition $Sd^{*}\circ \pi^{*}$ is homotopic to the identity of $C^{\ast}(K(\mathcal{U});F)$.

For that purpose we construct a sequence of homomorphisms,
\begin{equation}
D^{n+1}_K: C^{n+1}(K(\mathcal{U});F)\rightarrow C^{n}(K(\mathcal{U});F),
\end{equation}
for $n\geq 0$, by recurrence over the integer $n$, such that
\begin{equation}
Id-Sd^{n}\circ \pi^{n}=d\circ D^{n}_K+D^{n+1}_K\circ d.
\end{equation}
For $n=0$, and $c\in C^{1}(K(\mathcal{U});F)$, we simply take $D^{1}_K (c)(U)=0$. This works because, if $c$ is a $0$-cochain of $C^0(K(\mathcal{U}),F)$,
\begin{equation}
(Sd^{0}\circ \pi^{0}c)(U_0)=c(U_0).
\end{equation}
For $n=1$, and $c\in C^{2}(K(\mathcal{U});F)$, take
\begin{equation}
D^{2}_Kc(U_0,U_1)=c(U_0,U_1,\pi(U_0\cap U_1))|U_0\cap U_1.
\end{equation}
This gives for $c'\in C^{2}(K(\mathcal{U});F)$:
\begin{equation}
D^{2}_K(dc')(U_0,U_1)=c'(U_1,\pi(U_u))|U_u-c'(U_0,\pi(U_u))|U_u+c'(U_0,U_1);
\end{equation}
where as usual we have denoted $U_0\cap U_1$ by the symbol $U_u$. On the other side,
\begin{align*}
(Sd^{1}\circ \pi^{1}c')(U_0,U_1)&=-(\pi^{1}c')_u(U_1)|U_u+(\pi^{1}c')_u(U_0)|U_u\\
&=-(\pi^{1}c')(U_1,U_u)|U_u+(\pi^{1}c')(U_0,U_u)|U_u\\
&=-c'(U_1,\pi(U_u))|U_u+c'(U_0,\pi(U_u))|U_u.
\end{align*}
Then $Id-Sd^{1}\circ \pi^{1}=D^{2}_K\circ d+d\circ D_K^{1}$, as we expected.

 More generally, for any $n$, consider a $n$-cochain $c$
of $K(\mathcal{U})$ with respect to the local system $F$.  For a sequence $u'=(U'_0,...,U'_{n-1})$ in $\mathcal{U}$, let us define
\begin{equation}
c^{\pi}_u(U'_0,...,U'_{n-1})=c(U'_0,...,U'_{n-1},\pi(U_u))
\end{equation}
if $U_{u'}\supseteq \pi(U_u)$, and $c^{\pi}_u(U'_0,...,U'_{n-1})=0$ if not.

Then consider a decreasing sequence $v=(V_0,...,V_{n-1})$ in $\mathcal{A}(\mathcal{U})$. If $U_u\subseteq U_{\pi(v)}=\bigcap_{i=0}^{n-1}\pi (V_{i})$,
\begin{align*}
(\pi^{n}c)_u(V_0,...,V_{n-1})&=\pi^{n}c(V_0,...,V_{n-1},U_{u})\\
&=c(\pi(V_0),...,\pi(V_{n-1}),\pi (U_u))\\
&=c^{\pi}_u(\pi(V_0),...,\pi(V_{n-1}))\\
&=\pi^{n-1}(c^{\pi}_u)(V_0,...,V_{n-1}).
\end{align*}
If $U_u\nsubseteq U_{\pi(v)}$, we have $(\pi^{n}c_n)_u(v)=0=(\pi^{n-1}(c^{\pi}_u))(v)$. Therefore, in all cases
\begin{equation}\label{cpiu}
(\pi^{n}c)_u(v)=\pi^{n-1}(c^{\pi}_u)(v).
\end{equation}

 Now assume that $D^{q+1}_K$ is defined for $q\leq n$, satisfying the homotopy relation for $Id-Sd^{q}\circ \pi^{q}$, and consider a $n$-cochain $c$
of $K(\mathcal{U})$ with respect to the local system $F$; for every sequence $u=(U_0,...,U_n)$ in $\mathcal{U}$, the chosen definition of $Sd^{n}$ gives
\begin{equation}
(Sd^{n}\circ \pi^{n}(c))(U_0,...,U_n)=\sum_{i=0}^{n}(-1)^{n-i}(Sd^{n-1}(\pi^{n}c)_u)(U_0,...,\widehat{U_i},...,U_n)|U_u.
\end{equation}
Thus, applying \eqref{cpiu} we get
\begin{equation}
(Sd^{n}\circ \pi^{n}(c))(U_0,...,U_n))=\sum_{i=0}^{n}(-1)^{n-i}Sd^{n-1}\circ \pi^{n-1}(c^{\pi}_u)(U_0,...,\widehat{U_i},...,U_n)|U_u.
\end{equation}
By applying the hypothesis of recurrence, we get
\begin{multline}\label{formule}
(Sd^{n}\circ \pi^{n}(c))(U_0,...,U_n))=(-1)^{n}\sum_{i=0}^{n}(-1)^{i}c^{\pi}_u(U_0,...,\widehat{U_i},...,U_n)|U_u\\
+\sum_{i=0}^{n}(-1)^{n+1-i}D^{n}_K\circ d(c^{\pi}_u)(U_0,...,\widehat{U_i},...,U_n)|U_u\\
+\sum_{i=0}^{n}(-1)^{n+1-i}d(D^{n-1}_K(c^{\pi}_u))(U_0,...,\widehat{U_i},...,U_n)|U_u.
\end{multline}
The last sum is zero due to $d\circ d=0$, and the first sum is $(-1)^{n}d(c^{\pi}_u)$, therefore
\begin{multline}
(Sd^{n}\circ \pi^{n}(c))(U_0,...,U_n)=(-1)^{n}d(c^{\pi}_u)(U_0,...,U_n)\\
+\sum_{i=0}^{n}(-1)^{n+1-i}D^{n}_K\circ d(c^{\pi}_u)(U_0,...,\widehat{U_i},...,U_n)|U_u
\end{multline}
As we obtained a formula for $d(c_u)$, we obtain a formula for $d(c^{\pi}_u)|U_u$.
In fact, writing $U_{n+1}=\pi(U_{u})$,
\begin{align*}
(dc)^{\pi}_u(U_0,..., U_{n})|U_u&=dc(U_0,...,U_{n},\pi(U_{u}))|U_u\\
&=\sum_{j=0}^{n+1}(-1)^{j}c(U_0,...,\widehat{U_j},...,U_n)|U_u\\
&=d(c^{\pi}_u)(U_0,..., U_{n})|U_u+(-1)^{n+1}c(U_0,...,U_{n})|U_u;
\end{align*}
Then, replacing $d(c^{\pi}_u)|U_u$ by $(dc)^{\pi}_u+(-1)^{n}c$ in the formula \eqref{formule}, we get
{\color{black}
\begin{multline}
(Sd^{n}\circ \pi^{n}(c))(U_0,...,U_n)=c(U_0,...,U_n)+(-1)^{n}(dc)^{\pi}_u(U_0,...,U_n)\\
+(-1)^{n+1}d\circ D^{n}_K\circ (dc)^{\pi}_u(U_0,...,\widehat{U_i},...,U_n)|U_u-d\circ D^{n}_K(c)(U_0,...,U_n).
\end{multline}
}
Assuming that we have defined $D^{n}_K$ on $C^{n}(K(\mathcal{U});F)$, we define $D^{n+1}_K$ on $C^{n+1}(K(\mathcal{U});F)$ by the following formula:
\begin{equation}
D^{n+1}_K(c')=(-1)^{n+1}(c')^{\pi}_u+(-1)^{n+1}dD^{n}_K(c')^{\pi}_u.
\end{equation}
This gives the awaited result.

\noindent $3)$ To finish the proof of the theorem, we have to demonstrate that the composition $\pi^{*} \circ Sd^{*}$ is homotopic to the identity of $C^{\ast}(N(\mathcal{U});F)$.
For that, we construct a sequence of homomorphisms,
\begin{equation}
D^{n+1}_N: C^{n+1}(N(\mathcal{U});F)\rightarrow C^{n}(N(\mathcal{U});F),
\end{equation}
by recurrence over the integer $n\geq 0$, such that
\begin{equation}
Id-\pi^{n}\circ Sd^{n}=d\circ D^{n}_N+D^{n+1}_N\circ d.
\end{equation}
For $n=0$, and $c\in C^{1}(N(\mathcal{U});F)$, we define $D^{1}_N(c)(V_0)=c(\pi(V_0),V_0)$. Remember that if $c$ is a zero cochain
for $N(\mathcal{U})$, $Sd^{0}c(U_0)=c(U_0)$.   Then
\begin{equation}
\pi^{0}Sd^{0}c(V_0)=c(\pi(V_0))=c(V_0)+(c(\pi(V_0))-c(V_0))=c(V_0)-D^{1}_N(dc)(V_0);
\end{equation}
which gives $c-\pi^{0}Sd^{0}c=D^{1}_N(dc)$ as desired.

 Now assume the recurrence hypothesis, that there exist operators $D^{q+1}_K$ for $q\leq n$, satisfying the homotopy relation for $Id-\pi^{q}\circ Sd^{q}$, and consider a $n$-cochain $c$ of $N(\mathcal{U})$ with respect to the local system $F$; for every decreasing sequence $v=(V_0,...,V_n)$ in $\mathcal{A}_\mathcal{U}$, we have
\begin{align*}
(\pi^{n} \circ Sd^{n}(c))(v)&=(Sd^{n}c)(\pi(V_0),...,\pi(V_n))|V_n\\
&=\sum_{i=0}^{n}(-1)^{n-i}Sd^{n-1}(c)_{\pi(v)}(\pi(V_0),...,\widehat{\pi(V_i)},...,\pi(V_n))|V_n\\
&=\sum_{i=0}^{n}(-1)^{n-i}\pi^{n-1}(Sd^{n-1}(c)_{\pi(v)})(V_0,...,\widehat{V_i},...,V_n)|V_n;
\end{align*}
which gives by applying the hypothesis of recurrence:
\begin{multline}
(\pi^{n} \circ Sd^{n}(c))(V_0,...,V_n)=\sum_{i=0}^{n}(-1)^{n-i}c_{\pi(v)}(V_0,...,\widehat{V_i},...,V_n)|V_n\\
+\sum_{i=0}^{n}(-1)^{n+1-i}D^{n}_N( dc_{\pi(v)})(V_0,...,\widehat{V_i},...,V_n)|V_n\\
+\sum_{i=0}^{n}(-1)^{n+1-i}d\circ D^{n-1}_N(c_{\pi(v)})(V_0,...,\widehat{V_i},...,V_n)|V_n.
\end{multline}
The last sum is zero due to $d\circ d=0$, the first one is equal to $(-1)^{n}d(c_{\pi(v)})(v)$, and the second one to
$(-1)^{n+1}d(D^{n}_N( dc_{\pi(v)})(v)$, that is
\begin{equation}
(\pi^{n} \circ Sd^{n}(c))(v)=(-1)^{n}d(c_{\pi(v)})(v)+(-1)^{n+1}d(D^{n}_N( dc_{\pi(v)})(v).
\end{equation}
But the relation \eqref{dcu} tells that
\begin{equation}
d(c_{\pi(v)})(v)=(dc)_{\pi(v)}(v)+(-1)^{n}c(v).
\end{equation}
Thus by substituting, we get

{\color{black}
\begin{multline}
(\pi^{n} \circ Sd^{n}(c))(v)=c(v)+(-1)^{n}(dc)_{\pi(v)}(v)\\-d(D^{n}_N(c))(v)-(-1)^{n}d(D^{n}_N( d(c)_{\pi(v)})(v);
\end{multline}
}
which gives the expected result,
{\color{black}
\begin{equation}
c(v)-(\pi^{n} \circ Sd^{n}(c))(v)=d(D^{n}_N(c))(v)+D^{n+1}_N(dc)(v);
\end{equation}
}
if we define, for any $c'\in C^{n+1}(N(\mathcal{U});F)$ and any $v$ in $N_n(\mathcal{U})$:
{\color{black}
\begin{equation}
D^{n+1}_N(c')(v)=(-1)^{n+1}c'_{\pi(v)}(v)+(-1)^{n}dD^{n}_N(c'_{\pi(v)})(v).
\end{equation}
}
This ends the proof.
\end{proof}

The constructions made in the proof show that the homotopy equivalence is natural in the category
of open covering of topological spaces and morphisms of local systems. More precisely: the homotopy equivalence of the theorem defines a natural transformation from the category whose objects are the open coverings of $X$ and whose morphisms are the adapted projections of the refinements $\mathcal{U}$ of $\mathcal{U}'$, to the category of homotopy classes of isomorphisms of cochain-complexes.

\section*{Acknowledgements} 
We would like to thank the anonymous reviewers for their insightful comments, which greatly improved the quality of this paper.

\appendix

\section{ Functoriality of sheaves}\label{appendix1:topology}

 Let $(X,\mathcal T)$ be a topological space. 
 We view $\mathcal T$ as a category with a unique arrow $U\leq V$ whenever $U\subseteq V$.
 Then a \emph{presheaf of sets} $\mathcal F$ on 
 $(X,\mathcal T)$  is a functor from 
 $\mathcal{T}^{op}$ to the category 
of sets. One denotes by $s|V$ the image of $s\in \mathcal{F}(U)$ under $\mathcal{F}(V\to U)$, and calls it the \emph{restriction}.  A \emph{sheaf} is a presheaf which satisfies the two following axioms:
 \begin{enumerate}
 \item For every $V\in \mathcal{T}$ and every open covering  $\mathcal{U}\subseteq \mathcal{T}$  of $V$ consisting of open subsets of \(V\), if $s, t$ are two elements of $F_V$
such that for any $U\in \mathcal{U}$ we have $s|U=t|U$, then $t=s$.
\item For every $V\in \mathcal{T}$ and every open covering  $\mathcal{U}\subseteq \mathcal{T}$  of $V$, if a family $(s_U)_{U\in \mathcal U} \in \prod_{U\in \mathcal U}  \mathcal{F}_U$
is such that for all  $ U,U'\in \mathcal{U}$, $s_U|(U\cap U')=s_{U'}|(U\cap U')$, then there exists $s\in \mathcal{F}_V$ such that for all $ U \in \mathcal{U}$, $s|U=s_{U}$.
 \end{enumerate}
 The notion extends to sets with additional structure, such as groups, modules, etc.

 There exists a canonical sheaf $\mathcal{F}^{\sim}$ associated with a presheaf $\mathcal{F}$ on $(X,\mathcal T)$, see \cite[Sec.~II.5]{MacLane1992}. The \emph{sheafification} $\mathcal{F}\mapsto \mathcal{F}^{\sim}$ defines a functor, which is left adjoint to the forgetful functor that sends a sheaf to itself, as it is already a presheaf.

Let $F$ be a presheaf on a topological space $X$, and let $x \in X$ be a given point. The \emph{stalk} of $\mathcal{F}$ at $x$ is defined as the directed colimit  
$\varinjlim_{U \ni x} \mathcal{F}(U)$,
where the colimit is taken over the directed system $(\mathcal{F}(U))_{U \ni x}$ for the order induced by reversing inclusion (see Definition 10.1 in \cite{gallier2022homology}).

We consider now the functoriality of sheaves.  Let  $f:X\rightarrow Y$ be a continuous map; it induces a functor $f^{-1} : \mathcal T_Y \to \mathcal T_X$ between the topologies (seen as categories) of $Y$ and $X$, respectively.
\begin{enumerate}
\item  If $\mathcal{F}$ is a presheaf   over $X$, the \emph{direct image} $f_*\mathcal{F}$
is defined on $Y$ by the formula: $f_*\mathcal{F}(V)=\mathcal{F}(f^{-1}(V))$. If $\mathcal{F}$ is a sheaf, this is also
the case for $f_*\mathcal{F}$ \cite{MacLane1992}. $f_{\ast}\mathcal{F}(V\to U)$ is defined as $\mathcal{F}(f^{-1}(V)\to f^{-1}(U))$. (In fact, $f_*F$ is also the pullback $(f^{-1})^*F$ of $F$ under the functor $f^{-1}$ according to \cite[Sec.~I.5]{Artin1972}). 
\item If $\mathcal{G}$ is a presheaf  over $Y$, the \emph{inverse image} $f^{-1}\mathcal{G}$
is defined on $X$ by the formula: $f^{-1}\mathcal{G}(U)=\varinjlim_{V\supseteq f(U)}\mathcal{G}(V)$, where the limit is taken
over the directed family of opens subsets $V$ of $Y$ which contain $f(U)$. $f^{-1}\mathcal{G}(V\to U)$ is the quotient map induced on the inclusion map $\bigoplus_{V_1\supseteq f(U)}\mathcal{G}(V_1)\to 
\bigoplus_{V_1\supseteq f(V)}\mathcal{G}(V_1)$. Even if $\mathcal{F}$ is a sheaf, in general
$f^{-1}\mathcal{G}$ is not a sheaf. We make use of the sheafification, and define the \emph{pullback} of $G$
by $f^{*}\mathcal{G}=(f^{-1}\mathcal{G})^{\sim}$.
\end{enumerate}

 The  functors $f_*, f^{-1}$ between the corresponding categories of presheaves are adjoint i.e. for any presheaves $\mathcal{F}$ on $X$ and $\mathcal{G}$ on $Y$, there exist natural bijections
\begin{equation}
\Hom_X(f^{-1}\mathcal{G},\mathcal{F})\cong \Hom_Y(\mathcal{G},f_*\mathcal{F}).
\end{equation}
Similarly,  $f^{*}$ is left adjoint to $f_*$ in the categories of sheaves.

{\color{black}\section{Relative \v{C}ech Cohomology and Long Exact Sequence}\label{appendix3:cech}}

{\color{black} In this section we show that the inclusion of one poset into another induces a short exact sequence of sheaves, which in turn induces a long exact sequence on the respective cohomology groups. We use this fact in the proof of Theorem~\ref{thm5}.}

We consider now a category whose objects are pairs $(X,\mathcal F)$ given by a topological space and a sheaf (resp. a presheaf) on it. A morphism from $(X,\mathcal{F})$ to $(Y,\mathcal{G})$ is a pair $(f,\varphi)$,
where $f:X\rightarrow Y$ is continuous, and $\varphi$ is a natural transformation from $\mathcal{G}$ to $f_*\mathcal{F}$, or equivalently (given the adjunction between $f_*$ and $f^*)$) a natural transformation $\varphi^{*}$ from $f^{-1}\mathcal{G}$ (resp. $f^{*}\mathcal{G}$) to $\mathcal{F}$.

 Suppose we are given a morphism $(f,\varphi):(X,\mathcal{F})\rightarrow (Y,\mathcal{G})$, and two open coverings
$\mathcal{U}$, $\mathcal{V}$ of $X$ and $Y$ respectively, such that $\mathcal{U}$ is a refinement of $f^{-1}(\mathcal{V})$. 
We can choose a  map $\lambda$ from $\mathcal{U}$ to $\mathcal{V}$  such that $\forall U\in \mathcal{U}$, $U\subseteq f^{-1}(\lambda (U))$. Such a $\lambda$ generalizes a projection map; the proof of 
 Proposition \ref{prop:homotopy_from_projection} establishes that any two such maps are homotopic in the simplicial sense. This induces a natural map of chain complexes
\begin{equation}
(f,\varphi,\lambda)^{*}:C^{\ast}(\mathcal{V};\mathcal{G})\rightarrow C^{\ast}(\mathcal{U};\mathcal{F}),
\end{equation}
which commutes with the coboundary operators.

Let us define
\begin{equation}
C^{\ast}(\mathcal{V},\mathcal {U};\mathcal{G},\mathcal{F})=\Ker((f,\varphi,\lambda)^{*}).
\end{equation}

{\color{black}\textbf{Hypothesis}: the map $\varphi$ is surjective, i.e. for any open set $V$ in $Y$ the map $\varphi_V:\mathcal{G}(V)\rightarrow f_*\mathcal{F}(V)$
is surjective.}

By the snake's lemma, we obtain a natural long exact sequence in cohomology:
\begin{multline}
...\rightarrow H^{q}(\mathcal{V};\mathcal{G})\rightarrow H^{q}(\mathcal{U};\mathcal{F})\rightarrow H^{q+1}(\mathcal{V},\mathcal {U};\mathcal{G},\mathcal{F})\\
\rightarrow H^{q+1}(\mathcal{V};\mathcal{G})\rightarrow H^{q+1}(\mathcal{U};\mathcal{F})\rightarrow ...
\end{multline}

 This sequence survives to the {\color{black}directed colimit} over coverings and gives an exact sequence of \v{C}ech cohomology of the pair
$(\mathcal{F},\mathcal{G})$ over the pair $(X,Y)$. {\color{black}We denote $H^*(X,Y; \mathcal{G},\mathcal{F})$ the relative \v{C}ech cohomology groups}.\\

{\color{black}
We use the previous snake lemma’s result in the proof of Theorem~\ref{thm5} in the following setting.
Consider the particular case of an inclusion $J:X\hookrightarrow Y$. A covering $\mathcal{V}$ of $Y$ induce a covering  $\mathcal{U}$ of $X$,
made of the (non-empty) intersections $V\cap X$
for $V\in \mathcal{V}$; there is an evident projection $\lambda$ from $\mathcal{U}$ to $\mathcal{V}$. The hypothesis above is satisfied if $\mathcal{G}=J_*\mathcal{F}$ over $Y$ and $\varphi$ is the identity.

If $c\in C^{n}(\mathcal{U};\mathcal{F})$, there exists $\widetilde{c}\in C^{n}(\mathcal{V};\mathcal{G})$ such that, for any family
$V_0,...,V_n$ of elements of $\mathcal{V}$, we have
\begin{equation}
\varphi(\widetilde{c}(V_0,...,V_n))=c(V_0\cap X,...,V_n\cap X)\in \mathcal{F}(\bigcap_{i=0}^{n}(V_i\cap X))=\varphi(\mathcal{G}(\bigcap_{i=0}^{n}V_i)).
\end{equation}
This gives $\widetilde{f_\lambda}^{*}(\widetilde{c})=c$, then the map $\widetilde{f_\lambda}^{*}=(f,\varphi,\lambda)^{*}$ is surjective. Then the long exact sequence holds.
}

\section{Möbius inversion and Index formula}\label{appendix5:finite_proba_functors}

 This is a continuation of Section \ref{sec:free-sheaves} on free sheaves. Our aim is to give a proof of the Theorem \ref{thm6}.

 We introduce now the hypothesis that the $\{E_i\}_{i\in I}$ are finite sets, of respective cardinality $N_i$.

 If $N_\alpha$ denotes the cardinality of $E_\alpha$, we have $N_\alpha=\prod_{i\in \alpha}N_i$.\\

If we suppose that {$\mathcal{A}$ is finite}, and that it satisfies the strong intersection property, the sheaf $F$ has a
decomposition in direct sum:
\begin{equation}
\forall\alpha\in \mathcal{A},\quad F_\alpha=\bigoplus_{\beta\subseteq \alpha}T_\beta,
\end{equation}
Let us denote by $D_\alpha$ the dimension of $T_\alpha$, for $\alpha\in \mathcal{A}$. We have $N_\alpha=\sum_{\alpha\rightarrow\beta}D_\beta$.
Then the M\"obius inversion formula gives
\begin{equation}
\forall\alpha\in \mathcal{A},\quad D_\alpha=\sum_{\alpha\rightarrow\beta}\mu_{\alpha,\beta}N_\beta;
\end{equation}
where the integral numbers $\mu_{\alpha\beta}$ are the M\"obius coefficients of $\mathcal{A}$.

 Let us remind what are these coefficients \cite{Rota1964}.  For any locally finite poset $\mathcal{A}$, they are
defined by the two following equations:
\begin{equation}
\forall \alpha,\gamma,\quad \delta_{\alpha,\gamma}=\sum_{\beta|\alpha\rightarrow \beta \rightarrow \gamma} \mu_{\alpha,\beta}
=\sum_{\alpha\rightarrow \beta \rightarrow \gamma} \mu_{\beta,\gamma}.
\end{equation}

where \(\delta_{\alpha,\gamma} = 1\) whenever \(\alpha = \gamma\) and \(0\) otherwise.

\begin{equation}
\forall \alpha,\beta,\quad \alpha\nrightarrow \beta \Rightarrow \mu_{\alpha,\beta}=0.
\end{equation}

 This gives a function from $\mathcal{A}\times \mathcal{A}$ to $\mathbb{Z}$, which is named the M\"obius function of the poset.  The M\"obius function of $\mathcal{A}^{op}$ is given by $\mu^{*}_{\beta,\alpha}=\mu_{\alpha, \beta}$.

For  example, if $\mathcal{A}$ is the full set of parts of a finite set $I$, including the empty set or not, we have, for $\beta\subseteq\alpha$:
\begin{equation}\label{combinatorialindex}
\mu_{\alpha,\beta}=(-1)^{|\alpha| -|\beta|},
\end{equation}
where $|\alpha|$ denotes the cardinality of $\alpha$, for any $\alpha\in \mathcal{A}$. This formula is called the inclusion-exclusion principle.
When $\beta=\emptyset$, the above formula holds true if we set $|\emptyset|=-1$.\\

If $\alpha\supseteq\omega$ are two elements of $\mathcal{A}$, and if $\mathcal{A}(\alpha,\omega)$ is the sub-poset of $\mathcal{A}$
made by the elements $\beta$ such that $\alpha\rightarrow\beta\rightarrow\omega$, the restriction of the M\"obius function of $\mathcal{A}$
to $\mathcal{A}(\alpha,\omega)$ coincides with the M\"obius function of $\mathcal{A}(\alpha,\omega)$.

 The formula \eqref{combinatorialindex} extends to the poset associated with any simplicial complex. This follows
from the preceding assertion, because in the case of a manifold,
for every pair of elements $\alpha, \omega$ of $\mathcal{A}$ such that $\omega\subseteq \alpha$, the elements $\beta$  between
$\alpha$ and $\omega$ are the same in $\mathcal{A}$ or in the simplex defined by $\alpha$.

The end of this section is dedicated to proving Theorem \ref{thm6}.

 If $\mathcal{A}$ verifies the strong intersection property, for each $\alpha\in \mathcal{A}$, the dimension of $H^{0}(\mathcal{U}_{\mathcal A};T_\alpha)$ is $D_\alpha$, then
Theorem \ref{thm3:IP_implies_extrafine} and Proposition \ref{prop5} imply:

\begin{proposition}\label{prop5-1:dimH0-V}
If the poset $\mathcal{A}$ is finite and satisfies the strong intersection property, and if the $\{E_i\}_{i\in I}$ are finite sets,
\begin{equation}
\dim_{\mathbb{K}}H^{0}(\mathcal{U}_{\mathcal{A}};F)=\sum_{\alpha, \beta\in \mathcal{A}}\mu_{\alpha\beta}N_\beta.
\end{equation}
\end{proposition}

 In particular for the full simplex $\Delta(n-1)=\mathcal{P}(J)$, if $J$ has cardinality $n$, and if $N_i=N$ for any vertex, the dimension of $H^{0}(\mathcal{A};V)$ is $N^{n}$.
 \begin{proof}
Since we include the empty set, with $V_\emptyset=S_\emptyset$ of dimension $1$, we get:
\begin{align*}
\sum_{\alpha, \beta \in \mathcal{A}}\mu_{\alpha\beta}N_\beta &=\sum_{k=0}^{n}C_n^{k}\sum_{l=0}^{k}C_k^{l}(-1)^{k-l}N^{l}\\
&=\sum_{k=0}^{n}(-1)^{k}C_n^{k}(1-N)^{k}
=\sum_{k=0}^{n}C_n^{k}(N-1)^{k}\\
&=(N-1+1)^{n}=N^{n}.
\end{align*}
 \end{proof}

\begin{remark}
In this case, if we remove the empty set, and compute the expression
we get the same result
\begin{align*}
\sum_{\alpha, \beta\in \mathcal{A}}\mu_{\alpha\beta}N_\beta&=\sum_{k=1}^{n}C_n^{k}\sum_{l=1}^{k}C_k^{l}(-1)^{k-l}N^{l}
=\sum_{k=1}^{n}(-1)^{k}C_n^{k}((1-N)^{k}-1)\\
&=\sum_{k=1}^{n}C_n^{k}(N-1)^{k}-\sum_{k=1}^{n}C_n^{k}(-1)^{k}\\
&=((N-1+1)^{n}-1)-((1-1)^{n}-1)=N^{n}.
\end{align*}
\end{remark}

 Let us now delete the maximal face $\alpha=I$, then the poset $\mathcal{A}$ becomes the boundary $\partial \Delta(n-1)$
of the $(n-1)$-simplex. If we include the empty set in $\mathcal{A}$, and compute the dimension of $H^{0}$; we obtain
\begin{align*}
\sum_{\alpha, \beta\in \mathcal{A}^{\times +}}\mu_{\alpha\beta}N_\beta&=\sum_{k=0}^{n-1}C_n^{k}\sum_{l=0}^{k}C_k^{l}(-1)^{k-l}N^{l}\\
&=\sum_{k=0}^{n-1}(-1)^{k}C_n^{k}(1-N)^{k}
=\sum_{k=0}^{n-1}C_n^{k}(N-1)^{k}\\
&=(N-1+1)^{n}-(N-1)^{n}\\
&=N^{n}-(N-1)^{n}.
\end{align*}

\begin{remark}
Now the expression $\sum_{\alpha, \beta\in \mathcal{A}}\mu_{\alpha\beta}N_\beta$ is not the same if we exclude $\emptyset$,
because in this case, we have
\begin{align*}
\sum_{\alpha, \beta\in \mathcal{A}^{\times}}\mu_{\alpha\beta}N_\beta&=\sum_{k=1}^{n-1}C_n^{k}\sum_{l=1}^{k}C_k^{l}(-1)^{k-l}N^{l}\\
&=\sum_{k=1}^{n-1}(-1)^{k}C_n^{k}((1-N)^{k}-1)
=\sum_{k=1}^{n-1}C_n^{k}(N-1)^{k}-\sum_{k=1}^{n-1}C_n^{k}(-1)^{k}\\
&=((N-1+1)^{n}-1-(N-1)^{n})-((1-1)^{n}-1-(-1)^{n})\\
&=N^{n}-(N-1)^{n}+(-1)^{n}.
\end{align*}
\end{remark}

We will see just below why there is a difference for the boundary $\partial \Delta(n-1)$ and not for the simplex $\Delta(n-1)$!

 If $\mathcal{A}$ satisfies the weak intersection property, then
\begin{equation}
\forall\alpha\in \mathcal{A},\quad \overline{{F}}_\alpha=\bigoplus_{\beta\subseteq \alpha}\overline{{T}}_\beta.
\end{equation}

\begin{proposition}\label{prop5-2}
If the poset $\mathcal{A}$ is finite and satisfies the weak intersection property, and if the $\{E_i\}_{i\in I}$ are finite sets,
\begin{equation}
\dim_{\mathbb{K}}H^{0}(\mathcal{U}_{\mathcal{A}};\overline{{F}})=\sum_{\alpha, \beta\in \mathcal{A}}\mu_{\alpha\beta}(N_\beta-1).
\end{equation}
\end{proposition}
\begin{proof}
We apply Proposition \ref{prop6} as we applied Proposition \ref{prop5} to prove Proposition \ref{prop5-1:dimH0-V}.
\end{proof}

\begin{definition}
Let $\mathcal{A}$ be a finite poset, the \emph{Euler characteristic} of $\mathcal{A}$ is defined by
\begin{equation}
\chi(\mathcal{A})=\sum_{\alpha, \beta\in \mathcal{A}}\mu_{\alpha\beta}.\\
\end{equation}
\end{definition}

 In fact, the Euler characteristic was defined by Rota \cite{Rota1964}, when $\mathcal{A}$ contains a maximal element $I$ and a minimal element $\emptyset$,
by the formula
\begin{equation}
E(\mathcal{A})=1+\mu_{I,\emptyset}.
\end{equation}
But take any finite poset $\mathcal{A}$, and add formally to $\mathcal{A}$ a maximal element $1$ and a minimal element $0$, obtaining a poset $\mathcal{A}^{+}$. Then,
for any $\alpha \in \mathcal{A}$,
\begin{equation}
0=\mu(\alpha,0)+\sum_{\beta\in \mathcal{A}, \beta\subseteq\alpha}\mu(\alpha,\beta),
\end{equation}
and
\begin{equation}
0=\mu(1,0)+\mu(0,0)+\sum_{\alpha\in \mathcal{A}}\mu(\alpha,0).
\end{equation}
Consequently
\begin{equation}
\chi(\mathcal{A})=\mu(1,0)+\mu(0,0)=E(\mathcal{A}^{+}).
\end{equation}
Therefore the two definitions accord. See also the categorical extension of these ideas by Tom Leinster  \cite{Leinster2008}.

 The \emph{Hall formula} (cf. \cite{Rota1964, Leinster2008}), tells that
\begin{equation} \chi(\mathcal{A})=r_0-r_1+r_2-...;
\end{equation}
where each $r_k$ is the number of non degenerate chains of length $k$ in $\mathcal{A}$.
This number is the Euler characteristic of the nerve $N(\mathcal{A})$ of the category $\mathcal{A}$, therefore
$\chi(\mathcal{A})$ coincides with the Euler characteristics of $N(\mathcal{A})$.
But we have seen in Section \ref{sec:nerves-comparison}  that the \v{C}ech cohomology of the (lower) Hausdorff space $\mathcal{A}$  with coefficients in $\mathbb{Z}$,
is isomorphic to the simplicial cohomology  of $N(\mathcal{A})$. Then $\chi(\mathcal{A})$ also coincides with the Euler-\v{C}ech characteristic
of the (lower) Hausdorff space $\mathcal{A}$. By duality of the M\"obius function, this is also true for the upper topology.

One can show that for the poset of a simplicial complex, the number $\chi(\mathcal{A})$
is the alternate sum of the numbers of faces of each dimension:
\begin{equation}
\chi(\mathcal{A})=a_0-a_1+...
\end{equation}
as in the original definition by Euler.

 Now consider $\mathcal{A}$ (finite) as a topological subspace $\mathcal{A}_t$ of the simplex $\mathcal{P}(I)$; its closure $\overline{\mathcal{A}}$
is a simplicial complex. Moreover, if $\mathcal{A}$ satisfies the weak intersection property, the inclusion of $\mathcal{A}_t$ in $\overline{\mathcal{A}}$ is
an equivalence of homotopy; therefore, in this case, $\chi(\mathcal{A})$ is also the usual Euler characteristic of the metric
space $\mathcal{A}$.

 Consequently, Proposition \ref{prop5-2} can be rephrased by the following formula
\begin{equation}
\dim_{\mathbb{K}}H^{0}(\mathcal{U}_{\mathcal{A}};\overline{F})+\chi(\mathcal{A})=\sum_{\alpha, \beta\in \mathcal{A}}\mu_{\alpha\beta}N_\beta.
\end{equation}

 Applying Theorem \ref{thm3:IP_implies_extrafine}, we get the following result:

 \begin{reptheorem}{thm6}
If the poset $\mathcal{A}$ is finite and satisfies the weak intersection property, and if the $\{E_i\}_{i\in I}$ are finite sets, then
\begin{equation}
\chi(\mathcal{A};F)=\sum_{k=0}^\infty (-1)^{k}\dim_{\mathbb{K}}H^{k}(\mathcal{U}_{\mathcal{A}};F)=\sum_{\alpha, \beta\in \mathcal{A}}\mu_{\alpha\beta}N_\beta.
\end{equation}
\end{reptheorem}

 Remark that we also have
\begin{equation}
\chi(\mathcal{A};F)=\dim_{\mathbb{K}}H^{0}(\mathcal{U}_{\mathcal{A}};\overline{F})+\chi(\mathcal{A}).
\end{equation}

 In the example of $\Delta(n-1)$, we have $\chi(\mathcal{A})=1$, and when $\mathcal{A}=\partial \Delta(n-1)$
we have $\chi(\mathcal{A})=1+(-1)^{n}$, therefore, with all the $N_i$ equals to $N$, this explains the results obtained in the previous remarks.

{The presheaf $\overline F$ corresponds to signed measures of sum 1. One can deduce, for instance, that  compatible measures of sum $1$ over the poset $\partial \Delta(n-1)$ come from a global measure, that always exists (Marginal Theorem) and depends on $(N-1)^{n}$ degrees of freedom. However, in general,
none of these measures is positive.}

\section{Upper Extra-Fine Sheaves are Not Čech Acyclic on \texorpdfstring{$\mathcal{A}^{\mathrm{op}}$}{A-op}}\label{counter-example}

 In this section we show that a presheaf $S$ from a posetal category $\A$ to $\mathbb R$ vector spaces, generated at an element, is not acyclic for the associated sheaf on the upper $A-$topology $X^\A$. This counter example complements the acyclicity result on the associated sheaf over $X_\A$ (Lemma \ref{lemma-extra-fine}).

Consider the poset \(\mathcal{A}\) with four elements \(\alpha, \beta, \gamma, \delta\). The order relation of $\mathcal{A}$ are the following: \(\alpha \geq \beta\), \(\alpha \geq \gamma\), and \(\alpha \geq  \delta\). 
Consider the functor \(S:\cat{C}_\geq(\mathcal{A})\to \mathbb{R}\text{-}\textbf{Vect}\) defined as $S(\alpha)=S(\beta)=\mathbb{R}$ and $S(\gamma)=S(\delta)=0$. The map \(F_{\alpha \beta} = \mathrm{id}\) is the identity map. $S$ is a presheaf from $\cat{C}_\leq(\mathcal{A})$ to the category of $\mathbb{R}$ vector spaces. 

Recall that \(X^{\mathcal{A}}\) is the topological space with underlying set \(\mathcal{A}\), whose open sets are the upper sets. For \(\mathcal{O}\) an open subset of \(X^{\mathcal{A}}\), define $\widetilde{F}(\mathcal{O}) = \lim F\vert_{\mathcal{O}}$. Here, \(\mathcal{O} \subseteq \mathcal{A}\) is seen as a sub-poset of \(\mathcal{A}\). The notation \(F\vert_{\mathcal{O}}\) refers to the restriction of \(F\) to \(\mathcal{O}\), now viewed as a subcategory of $\cat{C}_\geq(\mathcal A)$. $\widetilde{F}$ is a sheaf over \(X^{\mathcal{A}}\). We shall denote $\widetilde{F}$ as $F$. Recall that the space \(X^{\mathcal A}\) has as a finest covering (i.e., minimal for the refinement ordering), $\mathcal{U}$,  the cover containing the open subsets $U = U^\beta, \quad V = U^\gamma, \quad W = U^\delta$ and their intersection $Z = U \cap V \cap W= U^\alpha$. The pairwise intersections of any two of these open sets are all equal to \(Z\). In particular for any open set $\mathcal{O}_1\in \mathcal{U}$, then $\mathcal{O}_1\not \subseteq U_\beta$ then $S(\mathcal{O}_1)=0$.

We now consider the $C^{1}(X^A;S)$. Consider the alternating co-chain \(\psi\) defined as follows:
\[
\psi(U, V) = x, \quad \psi(Z, V) = 2x, \quad \psi(Z, U) = x,
\]
\[
\psi(W, U) = x, \quad \psi(W, V) = 2x, \quad \psi(Z, W) = 0.
\]

What we mean by alternating is that \(\psi(A, B) = -\psi(B, A)\). In particular, for pairs \((A, A)\), the value of \(\psi\) is zero.

The co-cycle identity is satisfied for $\psi$ when evaluated on the triple \((Z, U, V)\) and for its permuted forms due to antisymmetry. We also verify the identity for the triples \((U, V, W)\), \((Z, U, W)\), and \((Z, V, W)\). For triples with a repetition, i.e. at least two of the three opens are the same, antisymmetry is equivalent to the co-cycle identity.

If \(\phi\) were a co-chain in degree 0 such that \(\delta(\phi) = \psi\), we must have the following conditions:
\[
\phi_V - \phi_U = x, \tag{14}
\]
\[
\phi_V - \phi_Z = 2x, \tag{15}
\]
\[
\phi_U - \phi_Z = x. \tag{16}
\]

But \(\phi_V = 0\) as \(S(V)=0\). Then \(\phi_U = -x\) and \(\phi_Z = -2x\).

Furthermore:
\[
\phi_U - \phi_W = x, \tag{17}
\]
\[
\phi_V - \phi_W = 2x, \tag{18}
\]
\[
\phi_W - \phi_Z = x. \tag{19}
\]

However, since \(\phi_W = 0\) because \(S(W) = 0\), it follows that \(x = 0\). Therefore, $0 = \mathrm{im}\, \delta_{0}^{1} \neq \mathrm{ker}\, \delta_{0}^{1}$ which shows that \(H^{1}(X^{\mathcal{A}}; F) \neq 0\). Therefore, \(F\) is not acyclic.

\bibliographystyle{amsplain}
\bibliography{biblio}

\end{document}